\documentclass[10pt]{amsart}

\usepackage{amsthm}
\usepackage{amssymb,amsmath,amsfonts,microtype}
\usepackage{graphicx}
\usepackage{pdfsync}

\newcommand{\comment}[1]{}

\usepackage{esint}

\newtheorem{thm}{Theorem}[section]

\newtheorem{lem}{Lemma}[section]
\newtheorem{prop}{Proposition}[section]

\newtheorem*{thm1'}{Theorem 1'}

\newtheorem*{thm2'}{Theorem 2'}
\newtheorem*{thmB}{Theorem B}
\newtheorem*{propnB}{Proposition B}
\newtheorem*{corB}{Corollary B}
\newtheorem*{thmA}{Theorem A}
\newtheorem*{thmC}{Theorem C}
\newtheorem*{thmB'}{Theorem B$^\prime$}
\newtheorem*{thmA'}{Theorem A$^\prime$}
\newtheorem*{propnA'}{Proposition A$^\prime$}
\newtheorem*{propnB'}{Proposition B$^\prime$}
\theoremstyle{remark}

\theoremstyle{definition}

\theoremstyle{remark}

\DeclareMathOperator{\lcm}{lcm}
\newcommand{\R}{\mathbb{R}}

\newcommand{\Z}{\mathbb{Z}}
\newcommand{\N}{\mathbb{N}}
\newcommand{\FF}{\mathbb{F}}
\newcommand{\E}{\mathbb{E}}

\newcommand{\RR}{\mathcal{R}}
\newcommand{\MM}{\mathcal{M}}
\newcommand{\NN}{\mathcal{N}}
\newcommand{\HH}{\mathcal{H}}
\newcommand{\BB}{\mathcal{B}}
\newcommand{\PP}{\mathcal{P}}
\newcommand{\GG}{\mathcal{G}}

\newcommand{\F}{\mathcal{F}}

\newcommand{\f}{\mathfrak{f}}

\newcommand{\uu}{\underline{2}}
\newcommand{\ux}{\underline{x}}
\newcommand{\uy}{\underline{y}}
\newcommand{\un}{\underline{n}}
\newcommand{\ur}{\underline{r}}
\newcommand{\ut}{\underline{t}}
\newcommand{\us}{\underline{s}}

\newcommand{\of}{\bar{f}}

\newcommand{\al}{\alpha}
\newcommand{\la}{\lambda}
\newcommand{\lm}{\lambda}
\newcommand{\D}{\delta}

\newcommand{\de}{\delta}

\newcommand{\De}{\Delta}

\newcommand{\Ga}{\Gamma}
\newcommand{\eps}{\varepsilon}
\newcommand{\si}{\sigma}

\newcommand{\VE}{\varepsilon}

\newcommand{\1}{\mathbf{1}}

\newcommand{\subs}{\subseteq}
\newcommand{\ls}{\lesssim}

\newcommand{\p}{\prime}
\newcommand{\wh}{\widehat}

\numberwithin{equation}{section}
\newcommand{\eq}{\begin{equation}
\newcommand{\ee}{\end{equation}}}

\newcommand{\comp}{\operatorname{complex}}
\newcommand{\ind}{\operatorname{index}}

\begin{document}

\title[Weak hypergraph regularity 
and applications  to geometric Ramsey theory]{Weak hypergraph regularity 
and applications \\ to geometric Ramsey theory}
\author{Neil Lyall\quad\quad\quad \'{A}kos Magyar}
\thanks{The first and second authors were partially supported by grants NSF-DMS 1702411 and NSF-DMS 1600840, respectively.}

\address{Department of Mathematics, The University of Georgia, Athens, GA 30602, USA}
\email{lyall@math.uga.edu}
\email{magyar@math.uga.edu}



\maketitle

\setlength{\parskip}{3pt}

\begin{abstract} Let $\Delta=\Delta_1\times\ldots\times \Delta_d\subseteq\R^n$, where $\R^n=\R^{n_1}\times\cdots\times\R^{n_d}$ with each $\Delta_i\subs\R^{n_i}$ a non-degenerate simplex of $n_i$ points.
 We prove that any set $S\subs \R^n$, with $n=n_1+\cdots +n_d$ of positive upper Banach density necessarily contains an isometric copy of all sufficiently large dilates of the configuration $\Delta$. In particular any such set $S\subs \R^{2d}$ contains a $d$-dimensional cube of side length $\lm$, for all $\lm\geq \lm_0(S)$. 
 We also prove analogous results with the underlying space being the integer lattice.
The proof is based on a weak hypergraph regularity lemma and an associated counting lemma developed in the context of Euclidean spaces and the integer lattice.

\end{abstract}


\section{Introduction}

\subsection{Existing Results I: Distances and Simplices in Subsets of $\R^n$}\label{d&s}
Recall that the \emph{upper Banach density} of a measurable set $S\subseteq\R^n$ is defined by
\eq\label{density}
\de^*(S)=\lim_{N\rightarrow\infty}\sup_{t\in\R^n}\frac{|S\cap(t+Q(N))|}{|Q(N)|},\ee
where $|\cdot|$ denotes Lebesgue measure on $\R^n$ and $Q(N)$ denotes the cube $[-N/2,N/2]^n$.

A result of Furstenberg, Katznelson, and Weiss \cite{FKW86} states that
if $S\subseteq\mathbb{R}^2$ has positive upper Banach density, then its distance set
 $\{|x-x'|\,:\, x,x'\in S\}$ 
contains all sufficiently large numbers.
Note that the distance set of any set of positive Lebesgue measure in $\R^n$ automatically contains all sufficiently small numbers (by the Lebesgue density theorem) and  that it is easy to construct a set of positive upper density which does not contain a fixed distance  by  placing small balls centered on an appropriate square grid.

\begin{thmA}[Furstenberg, Katznelson, and Weiss \cite{FKW86}]\label{Dist}
If $S\subseteq\R^2$ with $\D^*(S)>0$, 
then there exists a  $\lm_0=\lm_0(S)$ such that $S$ is guaranteed to contain pairs of points $\{x_1,x_2\}$ with $|x_2-x_1|=\lm$ for all $\lambda\geq \lm_0$.
\end{thmA}

This result was later reproved using Fourier analytic techniques by Bourgain in \cite{Bourg87} where he established the following more general result for all configurations of $n$ points in $\R^n$ whose affine span is $n-1$ dimensional, namely for all  non-degenerate simplices.

\begin{thmB}[Bourgain \cite{Bourg87}]\label{BourSimp}
Let $\Delta\subseteq\R^{n}$ be a non-degenerate simplex of $n$ points.
If $S\subseteq\R^n$  with $\D^*(S)>0$, then
there exists a threshold $\lm_0=\lm_0(S,\Delta)$ such that $S$ contains an isometric copy of $\lm\Delta$ for all $\lambda\geq \lm_0$.
\end{thmB}

Recall that a finite point configuration $\Delta'$ is said to be an isometric copy of $\lm \Delta$
if there exists a bijection $\phi:\Delta\to \Delta'$ such that $|\phi(v)-\phi(w)|=\lm\,|v-w|$ for all $v,w\in\Delta$, i.e. if $\Delta'$ is obtained from $\lambda\Delta$ (the dilation of $\Delta$ by a factor $\lambda$) via a rotation and translation.

 
 

Bourgain deduced Theorem B as an immediate consequence of the following stronger quantitative result for measurable subsets of the unit cube of positive measure. In the proposition below, and throughout this article, we shall refer to a decreasing sequence $\{\lm_j\}_{j=1}^J$ as \emph{lacunary} if  $\lm_{j+1}\leq\lm_j/2$ for all $1\leq j<J$.

\begin{propnB}[Bourgain \cite{Bourg87}]\label{thmS2}  Let $\Delta\subseteq\R^{n}$ be a non-degenerate simplex of $n$ points.
For any $0<\delta\leq 1$  there exists a constant $J=O_\Delta(\delta^{-3n})$ such that if $1\geq \lm_1\geq\cdots\geq\lm_J$ is any lacunary sequence and $S\subseteq[0,1]^n$ with $|S|\geq\delta$, then there exists $1\leq j< J$ such that $S$ contains an isometric copy of $\lm\Delta$ for all $\lambda\in [\lm_{j+1},\lm_j]$.
\end{propnB}

In \cite{LM17} the authors provided a short direct proof of Theorem B without using Proposition B. It is based on the observation that uniformly distributed sets $S\subs \R^d$ contain the expected ``number" of isometric copies of dilates $\la\De$ and that all sets of positive upper density become uniformly distributed at sufficiently large scales. 
However, for the purposes of this paper it will be important to recall Bourgain's indirect approach.

To see that Proposition B implies Theorem B notice that if Theorem B were not to hold for some set $S\subseteq\R^n$ of upper Banach density $\delta^*(S)>\delta>0$, then there must exist a lacunary sequence $\la_1\geq \cdots\geq \lm_J\geq1$, with $J$ the constant in Proposition B, such that $S$ does not contain an isometric copy of $\la_j \Delta$ for any $1\leq j\leq J$. Taking a sufficiently large cube $Q$ with side length  $N\geq \lm_1$ and $\,|S\cap Q|\geq \de |Q|\,$ and scaling back $Q\to [0,1]^n$ contradicts Proposition B. 

We further note that by taking $\la_j=2^{-j}$ in Proposition B we obtain the following ``Falconer-type" result for subsets of $[0,1]^n$ of positive Lebesgue measure.

\begin{corB}\label{BourSimpComp}
If  $\Delta\subseteq\R^{n}$ is a non-degenerate simplex of $n$ points,
then any $S\subseteq[0,1]^n$ with $|S|>0$ will necessarily contain an isometric copy of $\lm\Delta$ for all $\lambda$ in some interval of length at least $\exp(-C_{\Delta} |S|^{-3n})$. 
\end{corB}

Bourgain further demonstrated in \cite{Bourg87} that no result along the lines of Theorem B can hold for configurations that contain any three points in arithmetic progression along a line, specifically   showing that for any $n\geq1$ there are sets of positive upper Banach density in $ \R^n$ which do not contain an isometric copy of configurations of the form $\{0, y, 2y\}$ with $|y|=\lm$ for all sufficiently large $\lambda$. 
This should be contrasted with the following remarkable result of Tamar Ziegler.

\begin{thmC}[Ziegler \cite{Ziegler06}]

Let $\F$ be any configuration of $k$ points in $\R^n$ with $n\geq2$.

If $S\subseteq\R^n$ has positive upper  density, then there exists a threshold $\lm_0=\lm_0(S,\F)$ such that $S_\VE$ contains an isometric copy of $\lm \F$ for all $\lambda\geq \lm_0$ and any $\VE>0$, where $S_\VE$ denotes the $\VE$-neighborhood of $S$.
\end{thmC}


Bourgain's example was later generalized by Graham \cite{Graham90} to establish that the condition that $\eps>0$ in Theorem C is necessary and cannot be strengthened to $\eps=0$  for any given  \emph{non-spherical} configuration $\F$ in $\R^n$ for any $n\geq1$, that is for any finite configuration of points that cannot be inscribed in some sphere. We  note that the sets constructed by Bourgain and Graham have the property that for any $\VE>0$ their $\VE$-neighborhoods will contain arbitrarily large cubes and hence trivially satisfy Theorem C with $\lm_0=0$.

It is natural to ask if \emph{any} spherical configuration $\F$, beyond the known example of simplices, has the property that every positive upper Banach density subset of $\R^n$, for some sufficiently large $n$, contains an isometric copy of $\lm \F$ for all sufficiently large $\lambda$, and even to conjecture that this ought to hold for \emph{all} spherical configurations. 
The first breakthrough in this direction came in  \cite{LM17} when the authors established this  for configurations of four points forming a $2$-dimensional rectangle in $\R^4$ and more generally for any configuration that is the direct product of two non-degenerate simplices in $\R^n$ for suitably large $n$.   

The purpose of this article is to present a strengthening of the results in \cite{LM17} and to extend them to cover configurations with a higher dimensional product structure   in both the Euclidean and discrete settings.


\subsection{New Results I: Rectangles and Products of Simplices in Subsets of $\R^n$}\


The first main result of this article is the following

\begin{thm}\label{Rect}
Let $\RR$ be $2^d$ points forming the vertices of a  fixed $d$-dimensional rectangle in $\R^{2d}$.
 \begin{itemize}
\item[(i)]
If $S\subseteq\R^{2d}$ has positive upper Banach density, then
 there exists a threshold $\lm_0=\lm_0(S,\RR)$ such that $S$ contains an isometric copy of $\lm\RR$ for all $\lambda\geq \lm_0$.
 
\medskip
 
\item[(ii)]
For any $0<\delta\leq 1$ there exists a constant $c=c(\delta,\RR)>0$ such that any $S\subseteq[0,1]^{2d}$ with $|S|\geq\delta$ is guaranteed to contain an isometric copy of $\lm\RR$ for all $\lambda$ in some interval of length at least $c$.
\end{itemize}
Moreover, if $\RR$ has sidelengths given by $t_1,\dots,t_d$, then the isometric copies of $\lm\RR$ in both \emph{(i)} and \emph{(ii)} above can all be realized in  the special form $\{x_{11},x_{12}\}\times\cdots\times\{x_{d1},x_{d2}\}\subseteq\R^{2}\times\cdots\times\R^{2}$ with each $|x_{j2}-x_{j1}|=\lm t_j$.
\end{thm}

The multi-dimensional extension of Szemer\'{e}di's theorem on arithmetic progressions in sets of positive density due to Furstenberg and Katznelson \cite{FK78} implies, and is equivalent to the fact, that there are isometric copies of $\la\RR$ in $S$ for \emph{arbitrarily} large $\la$, with sides parallel to the coordinate axis. Theorem \ref{Rect} states that there is an isometric copy of $\la \RR$ in $S$ for \emph{every} sufficiently large $\la$, but only with sides parallel to given 2-dimensional coordinate subspaces which provides an extra degree of freedom for each side vector of the rectangle $\RR$.

A weaker version of Theorem \ref{Rect}, with $\R^{2d}$ replaced with $\R^{5d}$, was later established by Durcik and Kova\v{c} in \cite{DK} using  an adaptation of arguments of the second author with Cook and Pramanik in \cite{CMP}. This approach also makes direct use of the full strength of the multi-dimensional Szemer\'edi theorem  and as such leads to quantitatively weaker results.   


Our arguments work for more general patterns where $d$-dimensional rectangles are replaced with direct products of non-degenerate simplices.

\begin{thm}\label{Prod}
Let $\Delta=\Delta_1\times\cdots\times\Delta_d\subseteq\R^n$, where $\R^n=\R^{n_1}\times\cdots\times\R^{n_d}$ and each $\Delta_j\subseteq\R^{n_j}$ is a non-degenerate simplex of $n_j$ points.

\begin{itemize}
\item[(i)]
If $S\subseteq\R^{n}$ has positive upper Banach density, then
 there exists a threshold $\lm_0=\lm_0(S,\Delta)$ such that $S$ contains an isometric copy of $\lambda\Delta$ for all $\lambda\geq \lm_0$.
 
\medskip
 
\item[(ii)]
For any $0<\delta\leq 1$ there exists a constant $c=c(\delta,\Delta)>0$ such that any $S\subseteq[0,1]^{n}$ with $|S|\geq\delta$ is guaranteed to contain an isometric copy of $\lm\Delta$ for all $\lambda$ in some interval of length at least $c$.
\end{itemize}
Moreover the isometric copies of $\lm\Delta$ in both \emph{(i)} and \emph{(ii)} above can all be realized in  the special form $\Delta_1'\times\cdots\times\Delta_d'$ with each $\Delta_j'\subseteq\R^{n_j}$ an isometric copy of $\lm\Delta_j$.
\end{thm}

\emph{Quantitative Remark.} A careful analysis of our proof reveals that the constant $c(\delta,\Delta)$ can be taken greater than $W_d(C'_\Delta\delta^{-3n_1\cdots n_d})^{-1}$ 
where
$W_k(m)$ is  a tower of exponentials defined by $W_1(m)=\exp(m)$ and $W_{k+1}(m)=\exp(W_k(m))$ for $k\geq1$.



\subsection{Existing Results II: Distances and Simplices in Subsets of $\Z^n$}\label{d&sZ}

The problem of counting isometric copies of a given non-degenerate simplex in $\Z^n$ (with one vertex fixed) has been extensively studied via its equivalent formulation as the number of ways  a quadratic form can be represented as a sum of squares of linear forms, see \cite{Kitaoka} and \cite{Siegel}. This was exploited by the second author in \cite{Magy08} and \cite{Magy09} to establish
analogous results to those described in Section \ref{d&s} above for subsets of the integer lattice $\Z^n$ of positive upper density.

Recall that the \emph{upper Banach density} of a set $S\subseteq\Z^n$ is analogously defined by
\eq\label{BD}
\D^*(S)=\lim_{N\rightarrow\infty}\sup_{t\in\R^n}\frac{|S\cap(t+Q(N))|}{|Q(N)|},\ee
where $|\cdot|$ now denotes counting measure on $\Z^n$ and $Q(N)$ the discrete cube $[-N/2,N/2]^n\cap\Z^n$.

In light of the fact that any pairs of distinct points 
$\{x_1,x_2\}$ in $\Z^n$ has the property that the square of the distance between them $|x_2-x_1|^2$ is always a positive integer we introduce the convenient notation
\[\sqrt{\N}:=\{\lm\,:\,\lm>0 \text{ and } \lm^2\in\Z\}.\]

\begin{thmA'}[Magyar \cite{Magy08}]\label{discrete_distance} Let $0<\delta\leq1$.

If $S\subseteq\Z^{5}$ has  upper Banach density at least $\delta$, then
there exists an integer $q_0=q_0(\delta)$ and $\lm_0=\lm_0(S)$ such that
 $S$  contains pairs of points 
$\{x_1,x_2\}$ with $|x_2-x_1|=q_0\lm$ for all $\lm\in\sqrt{\N}$ with $\lm\geq\lm_0$.
\end{thmA'}

\smallskip

\begin{thmB'}[Magyar \cite{Magy09}]\label{Discrete_Simplices}
Let $0<\delta\leq1$ and $\Delta\subseteq\Z^{2n+3}$ be a non-degenerate simplex of $n$ points.

\begin{itemize}
\item[(i)]
If $S\subseteq\Z^{2n+3}$ has  upper Banach density at least $\delta$, 
then
there exists an integer $q_0=O(\exp(C_\Delta\delta^{-13n}))$ and $\lm_0=\lm_0(S, \Delta)$ such that $S$ contains an isometric copy of $q_0\lm\Delta$ for all $\lm\in\sqrt{\N}$ with $\lambda\geq \lm_0$.
 
\medskip
 
\item[(ii)]
If $N\geq \exp(2C_{\Delta} \delta^{-13n})$, then any $S\subseteq\{1,\dots,N\}^{2n+3}$ with cardinality $|S|\geq\delta N^{2n+3}$ will necessarily contain an isometric copy of $\lm\Delta$ for some $\lm\in\sqrt{\N}$ with $1\leq\lambda\leq N$.
\end{itemize}
\end{thmB'}


Note that the fact that $S\subseteq\Z^n$ could fall entirely into a fixed congruence class of some integer $1\leq q\leq\delta^{-1/n}$  ensures that the $q_0$ that appears in Theorems A$^\prime$ and B$^\prime$ above must be divisible by the  least common multiple of all integers $1\leq q\leq \delta^{-1/n}$. Indeed if $S=(q\Z)^n$ with $1\leq q\leq \delta^{-1/n}$ then $S$ has upper Banach density at least $\delta$, however the distance between any two points  $x,y\in S$ is of the form $|x-y|=q\la$ for some $\lm\in\sqrt{\N}$.

However, in both Theorems A$^\prime$ and Part (i) of Theorem B$^\prime$, one can take $q_0=1$ if the sets $S$ are assumed to be suitably uniformly distributed on congruence classes of small modulus. This leads via an easy density increment strategy to short new proofs, see \cite{LM19} for Theorem A$^\prime$ and Section \ref{Appendix} for Part (i) of Theorem B$^\prime$. 

The original argument in \cite{Magy09} deduced Theorem B$^\prime$ from the following discrete analogue of Proposition B.

\begin{propnB'}[Magyar \cite{Magy09}]\label{PropB'}  Let $\Delta\subseteq\Z^{2n+3}$ be a non-degenerate simplex of $n$ points.

For any $0<\delta\leq 1$  there exist constants  $J=O_\Delta(\delta^{-3n})$ and $q_0=O(\exp(C_\Delta\delta^{-13n}))$  such that if $N\geq \lm_1\geq\cdots\geq\lm_J\geq1$ is any lacunary sequence in $q_0\sqrt{\N}$ and $S\subseteq\{1,\dots,N\}^{2n+3}$ with cardinality $|S|\geq\delta N^{2n+3}$, then $S$ will necessarily contain an isometric copy of $\lm_j\Delta$ for some $1\leq j\leq J$.
\end{propnB'}


To see that Proposition B$^\prime$ implies Theorem B$^\prime$ notice that if Part (i) of Theorem B$^\prime$ were not to hold for some set $S\subseteq\Z^{2n+3}$ of upper Banach density $\delta^*(S)>\delta>0$ with $q_0$ from Proposition B$^\prime$, then there must exist a lacunary sequence $\la_1\geq\cdots\geq\lm_J\geq 1$ in $q_0\sqrt{\N}$, with $J$ the constant from Proposition B$^\prime$, such that $S$ does not contain an isometric copy of $\la_j\Delta$ for any $1\leq j\leq J$. Since we can find a sufficiently large cube $Q$ with integer side length $N$ that is divisible by $q_0$ and greater than $\lm_1$ such that   $\,|S\cap Q|\geq \de |Q|\,$, this contradicts Proposition B$^\prime$. 
Part (ii) of Theorem B$^\prime$ follows from Proposition B$^\prime$ by taking $\lm_j=2^{J-j}q_0$.



\subsection{New Results II: Rectangles and Products of Simplices in Subsets of $\Z^n$}\

We will also establish the following discrete analogues of Theorem \ref{Rect} and \ref{Prod}.

\begin{thm}\label{RectZ}
Let $0<\delta\leq 1$ and $\RR$ be $2^d$ points forming the vertices of a $d$-dimensional rectangle in $\Z^{5d}$.
 \begin{itemize}
\item[(i)]
If $S\subseteq\Z^{5d}$ has upper Banach density at least $\delta$, then
there exist integers $q_0=q_0(\delta,\RR)$ and $\lm_0=\lm_0(S,\RR)$ such that
$S$ contains an isometric copy of $q_0\lm\RR$ for all $\lm\in\sqrt{\N}$ with $\lambda\geq \lm_0$.
 
\medskip
 
\item[(ii)]
There exists a constant $N(\delta,\RR)$ such that if $N\geq N(\delta,\RR)$, then any $S\subseteq\{1,\dots,N\}^{5d}$ with cardinality $|S|\geq\delta N^{5d}$ will necessarily contain an isometric copy of $\lm\RR$ for some $\lm\in\sqrt{\N}$ with $1\leq\lambda\leq N$.
\end{itemize}
If $\RR$ has side lengths given by $t_1,\dots,t_d$, then each of the isometric copies 
in \emph{(i)} and \emph{(ii)} above can be realized in  the form $\{x_{11},x_{12}\}\times\cdots\times\{x_{d1},x_{d2}\}\subseteq\Z^{5}\times\cdots\times\Z^{5}$ with each $|x_{j2}-x_{j1}|=q_0\lm t_j$ and $\lm  t_j$, respectively.
\end{thm}

Our arguments again work for more general patterns where $d$-dimensional rectangles are replaced with direct products of non-degenerate simplices.

\begin{thm}\label{ProdZ}
Let $0<\delta\leq 1$ and $\Delta=\Delta_1\times\cdots\times\Delta_d\subseteq\Z^n$, where $\Z^n=\Z^{2n_1+3}\times\cdots\times\Z^{2n_d+3}$ and each $\Delta_i\subseteq\Z^{2n_i+3}$ is a non-degenerate simplex of $n_i$ points.

\begin{itemize}
\item[(i)]
If $S\subseteq\Z^{n}$ has upper Banach density at least $\delta$, then
there exist integers $q_0=q_0(\delta,\Delta)$ and $\lm_0=\lm_0(S,\Delta)$ such that
$S$ contains an isometric copy of $q_0\lm\Delta$ for all $\lm\in\sqrt{\N}$ with $\lambda\geq \lm_0$.

\medskip
 
\item[(ii)]
There exists a constant $N(\delta,\Delta)$ such that if $N\geq N(\delta,\Delta)$, then any $S\subseteq\{1,\dots,N\}^{n}$ with cardinality $|S|\geq\delta N^{n}$ will necessarily contain an isometric copy of $\lm\Delta$ for some $\lm\in\sqrt{\N}$ with $1\leq\lambda\leq N$.\end{itemize}
Moreover, each of the isometric copies 
in \emph{(i)} and \emph{(ii)} above can be realized in  the special form $\Delta_1'\times\cdots\times\Delta_d'$ with each $\Delta_i'\subseteq\Z^{2n_i+3}$ an isometric copy of $q_0\lm\Delta_j$ and $\lm\Delta_j$, respectively.
\end{thm}

\emph{Quantitative Remark.} A careful analysis of our proof reveals that the constant $q_0(\delta,\Delta)$ (and consequently also $N(\delta,\Delta)$) can be taken less than $W_d(C'_\Delta\delta^{-13n_1\cdots n_d})$ where $W_k(m)$ is  a tower of exponentials defined by $W_1(m)=\exp(m)$ and $W_{k+1}(m)=\exp(W_k(m))$ for $k\geq1$.


\subsection{Notations and Outline.} We will consider the parameters $d,n_1,\ldots,n_d$ fixed and will not indicate the dependence on them. Thus we will write $f=O(g)$ 
if $|f|\leq C(n_1,\ldots,n_d) g$. If the implicit constants in our estimates depend on additional parameters $\eps,\de,K,\ldots$ the we will write $f=O_{\eps,\de,K,\dots} (g)$. 
We will use the notation $f\ll g$ to indicate that $|f|\leq c\,g$ for some constant $c>0$ sufficiently small for our purposes. 


Given an $\eps>0$ and a (finite or infinite) sequence $L_0\geq L_1\geq \cdots >0$, we will say that the sequence is $\eps$-\emph{admissible} if $L_j/L_{j+1}\in\N$ and $L_{j+1}\ll\eps^2 L_j$ for all $j\geq 1$. Moreover, if $q\in\N$ is given and $L_j\in\N$ for all $1\leq j\leq J$, then we will call the sequence $L_0\geq L_1\geq\cdots \geq L_J$ $(\eps,q)$-\emph{admissible} if in addition $L_J/q\in\N$. Such sequences of scales will often appear in our statements both in the continuous and the discrete case.


Our proofs are based on a weak hypergraph regularity lemma and an associated counting lemma developed in the context of Euclidean spaces and the integer lattice.
In Section \ref{FFSECTION} we introduce our approach in the model case of finite fields and prove an analogue of Theorem \ref{Rect} in this setting. In Section \ref{Base} we review Theorem \ref{Prod} for a single simplex and ultimately establish the base case of our general  inductive approach to Theorem \ref{Prod}. In Section \ref{two} we address Theorem \ref{Prod} for the direct product of two simplices, this provides a new proof (and strengthening) of the main result of \cite{LM17} and serves as a gentle preparation for the more complicated general case which we present in the Section \ref{generalcase}. The proof of Theorem \ref{ProdZ} is outlined in 
Sections \ref{BaseZ} and \ref{generalcaseZ}, while a short direct proof of Part (i) of Theorem B$^\prime$ is presented in Section \ref{Appendix}.


\section{Model case: vector spaces over finite fields.}\label{FFSECTION}

In this section we will illustrate our general method by giving a complete proof of Theorem \ref{Rect} in the model setting of $\FF_q^n$ where $\FF_q$ denotes the finite field of $q$ elements. We do this as the
notation and arguments are more transparent in this setting yet many of the main ideas are still present.

We say that two vectors $u,v\in\FF_q^n$ are orthogonal, if $x\cdot y=0$, where ``$\cdot$" stands for the
usual dot product. A rectangle in $\FF_q^n$ is then a set $\RR=\{x_1,y_1\}\times\cdots\times \{x_n,y_n\}$ with side vectors $y_i-x_i$
being pairwise orthogonal. 

The finite field analogue of Theorem \ref{Rect} is the following

\begin{prop}\label{thmF1} For any $0<\delta\leq 1$ there exists an integer $q_0=q_0(\delta)$ with the following property:

If  $q\geq q_0$ and $t_1,\dots,t_d\in\FF_q^*$, then any $S\subs \FF_q^{2d}$ with $|S|\geq \delta \,q^{2d}$ will contain points
 \[\{x_{11},x_{12}\}\times\cdots\times\{x_{d1},x_{d2}\}\subseteq V_1\times\cdots\times V_d \quad\text{with \ $|x_{j2}-x_{j1}|^2=t_j$  for  $1\leq j\leq d$}\]
  where
we have written $\FF_q^{2d}=V_1\times\cdots\times V_d$ with $V_j\simeq \FF_q^2$ pairwise orthogonal coordinate subspaces.
\end{prop}

\subsection{Overview of the proof of Proposition \ref{thmF1} }

Write $\FF_q^{2d}=V_1\times\ldots\times V_d$ with $V_j\simeq \FF_q^2$ pairwise orthogonal coordinate subspaces.
For any $\ut:=(t_1,\ldots,t_d)\in\FF_q^*$ and $S\subseteq \FF_q^{2d}$ we define
\[\NN_{\ut}(1_S):= \E_{\ux_{1}\in V_1^2, \dots,\ux_{d}\in V_d^2}\!\!\!\!\!\!\!\!\prod\limits_{(\ell_1,\dots,\ell_d)\in\{1,2\}^d} 1_S(x_{1\ell_1},\dots,x_{d\ell_d})\,\prod\limits_{j=1}^d\sigma_{t_j}(x_{j2}-x_{j1})
\]
where we used the shorthand notation $\ux_j:=(x_{j1},x_{j2})$ for each $1\leq j\leq d$ and  the averaging notation: \[\E_{x\in A} f(x):=\frac{1}{|A|}\sum_{x\in A} f(x)\] for a finite set $A\neq\emptyset$. We have also used the notation
\[\sigma_t(x)=\begin{cases}
q \quad \text{if $|x|^2=t$}\\
0 \quad \text{otherwise}
\end{cases}\] 
for each $t\in\FF_q^*$. Note that the function $\si_t$ may be viewed as the discrete analogue of the normalized surface area measure on the sphere  of radius $\sqrt{t}$. It is
well-known, see \cite{IR07}, that
\[\E_{x\in\FF_q^2}\ \sigma_t(x)=1+O(q^{-1/2})\]
and for all $\xi\neq 0$ one has
\[\hat{\sigma}_t(\xi):=\E_{x\in\FF_q^2}\ \sigma_t(x)\,e^{2\pi
i\frac{x\cdot\xi}{q}}=O(q^{-1/2}).\]

Note that if $\NN_{\ut}(1_S)>0$, then this  implies that $S$ contains a rectangle of the form $\{x_{11},x_{12}\}\times \cdots\times \{x_{d1},x_{d2}\}$ with $x_{j1},x_{j2}\in V_j$ and  $|x_{j2}-x_{j1}|^2=t_j$  for  $1\leq j\leq d$. 

\smallskip

Our approach to  Proposition  \ref{thmF1} in fact  establishes the following quantitatively stronger result.

\begin{prop}\label{propF1} For any $0<\VE\leq 1$ there exists an integer $q_0=q_0(\VE)$ with the following property:

If  $q\geq q_0$, then for  any $S\subs \FF_q^{2d}$ and  $t_1,\dots,t_d\in\FF_q^*$ one has
\[\NN_{\ut}(1_S)>\left(\frac{|S|}{q^{2d}}\right)^{2^d}-\VE\]
 where we have written $\FF_q^{2d}=V_1\times\ldots\times V_d$ with $V_j\simeq \FF_q^2$ pairwise orthogonal coordinate subspaces.
\end{prop}



A crucial observation in the proof of Proposition \ref{propF1} is that the  averages  $\NN_{\ut}(1_S)$ can be compared to ones which can be easily estimated from below. 
We define, for any $S\subseteq \FF_q^{2d}$,  the (unrestricted) count
\[
\MM(1_S):=\E_{\ux_{1}\in V_1^2, \dots,\ux_{d}\in V_d^2}\!\!\!\!\!\!\!\!\prod\limits_{(\ell_1,\dots,\ell_d)\in\{1,2\}^d} 1_S(x_{1\ell_1},\dots,x_{d\ell_d}).\]
It is easy to see, by carefully applying Cauchy-Schwarz $d$ times to $\E_{x_{11}\in V_1,\dots,x_{d1}\in V_d}1_S(x_{11},\dots,x_{d1})$, that 
\eq\label{maintermff}
\MM(1_S)\geq \left(\frac{|S|}{q^{2d}}\right)^{2^d}.\ee

Our approach to Proposition \ref{propF1} therefore reduces to establishing that for any $\VE>0$ one has
\eq\label{Fd=k}
\NN_{\ut} (1_S)\,=\,\MM(1_S)+O(\eps) + O_{\eps}(q^{-1/2}).
\ee

The validity of (\ref{Fd=k}) will follow immediately from the $d=k$ case of Proposition \ref{main-rect-ff} below. 
However, before we can state this \emph{counting lemma} we need to introduce some further notation from the theory of hypergraphs, notation that we shall ultimately  make use of   throughout the paper.

\subsection{Hypergraph Notation and a Counting Lemma}\label{hbnotation}\

In order to streamline our notation we will make use the language of hypergraphs. 
For $J:=\{1,\ldots,d\}$ and  $1\leq k\leq d$, we let
$\HH_{d,k}=\{e\subs J;\ |e|=k\}$ denote the full $k$-regular hypergraph on the vertex set $J$. For $K:=\{jl;\ j\in J,\, l\in\{1,2\}\}$
we define the projection $\pi:K\to J$ as $\pi(jl):=j$ and use this in turn to define the hypergraph bundle 
\[
\HH_{d,k}^{\uu}:=\{e\subs K;\ |e|=|\pi(e)|=k\}\] using the shorthand notation $\underline{2}=(2,2,\ldots,2)$ to indicate that
$|\pi^{-1}(j)|=2$ for all $j\in J$. 

Notice when $k=d$ then $\HH_{d,d}$ consists of one element, the set $e=\{1,\ldots,d\}$, and
\[\HH_{d,d}^{\underline{2}}=\{\,\{1l_1,\ldots,dl_d\};\ (l_1,\ldots,l_d)\in \{1,2\}^d\}.\]

Let $V:=\FF_q^{2d}$ and $V=V_1\times\ldots\times V_d$ with $V_j\simeq \FF_q^2$ pairwise orthogonal coordinate subspaces. For a given $\underline{x}=(x_{11},x_{12},\ldots,x_{d1},x_{d2})\in V^2$ with $x_{j1},\,x_{j2}\in V_j$ and a given edge
$e=\{1l_1,\ldots,dl_d\}$, we write \[\underline{x}_e:=(x_{1l_1},\ldots,x_{dl_d}).\] Note that the map $\ux\to\ux_e$ defines a
projection $\pi_e:V^2\to V$.
With this notation, we can clearly now write
\[\NN_{\ut}(1_S)=\E_{\ux\in V^2} \prod_{e\in \HH_{d,d}^{\uu}} 1_S(\ux_e)\ \prod_{j=1}^d
\sigma_{t_j}(x_{j2}-x_{j1})\]
\[\MM(1_S)=\E_{\ux\in V^2} \prod_{e\in \HH_{d,d}^{\uu}} 1_S(\ux_e).\]

Now for any $1\leq k\leq d$ and any  edge $e'\in\HH_{d,k}$, i.e. $e'\subs \{1,\ldots,d\}$, $|e'|=k$, we let $V_{e'}:=\prod_{j\in e'} V_j$. For every $\ux\in V^2$ and $e\in\HH_{d,k}^{\uu}$, we define $\ux_e:=\pi_e(\ux)$ where  $\pi_e: V^2\to V_{\pi(e)}$ is the natural projection map. 

Our key \emph{counting lemma}, Proposition \ref{main-rect-ff} below, which we will establish by induction on $1\leq k\leq d$ below, is then the statement that given a family of functions $f_e:V_{\pi(e)}\to [-1,1],\,e\in\HH_{d,k}^{\uu}$, the averages (generalizing those discussed above) which are defined by
\eq\label{multi-count-k}
\NN_{\ut} (f_e;\ e\in \HH_{d,k}^{\uu}):=\E_{\ux\in V^2} \prod_{e\in \HH_{d,k}^{\uu}} f_e(\ux_e)\ \prod_{j=1}^d
\sigma_{t_j}(x_{j2}-x_{j1})
\ee
\eq\label{multi-box-k}
\MM(f_e;\,e\in \HH_{d,k}^{\uu}):=\E_{\ux\in V^2} \prod_{e\in \HH_{d,k}^{\uu}} f_e(\ux_e).
\ee
are approximately equal. Specifically, one has

\begin{prop}[Counting Lemma]\label{main-rect-ff}  Let $1\leq k\leq d$ and $0<\eps\leq 1$. For any collection of functions \[\text{$f_e: V_{\pi(e)}\to [-1,1]$ with $e\in\HH_{d,k}^{\uu}$}\] 
one has
\eq\label{main-rect-f}
\NN_{\ut} (f_e;\ e\in \HH_{d,k}^{\uu})\,=\,\MM(f_e;\,e\in \HH_{d,k}^{\uu})+O(\eps) + O_{\eps}(q^{-1/2}).
\ee
\end{prop}

If we apply this Proposition with $d=k$ and $f_e=1_S$ for all $e\in\HH_{d,d}^{\uu}$, then Theorem \ref{thmF1} clearly follows given the lower bound (\ref{maintermff}).

\subsection{Proof of Proposition \ref{main-rect-ff}}

We will establish  Proposition \ref{main-rect-ff} by inducting on $1\leq k\leq d$.

 For $k=1$ the result follows from the
 basic observation that
if $f_1,f_2:\FF_q^2\to [-1,1]$ and let $t\in\FF_q^*$, then
\begin{align}\E_{x_1,x_2\in\FF_q^2}\ f_1(x_1)f_2(x_2)\,\si_t(x_2-x_1)&= \sum_{\xi\in\FF_q^2} \hat{f}_1(\xi)\hat{f}_2(\xi)\hat{\si}_t(\xi) \nonumber \\
& =\hat{f}_1(0)\hat{f}_2(0) + O(q^{-1/2})\label{k=1} \\
&= \E_{x_1,x_2\in\FF_q^2}\ f_1(x_1)f_2(x_2)+O(q^{-1/2})\nonumber
\end{align}
by the properties of the function $\hat{\si}$ given above.

To see how this implies Proposition \ref{main-rect-ff} for $k=1$ we note that since $\HH^{\uu}_{d,1}=\{jl:\ 1\leq j\leq d,\,1\leq l\leq 2\}$ it follows that
\begin{align*}
\NN_{\ut} (f_e;\ e\in \HH_{d,1}^{\uu}) & = \prod_{j=1}^d \E_{x_{j1},x_{j2}\in\FF_q^2}\ f_{j1}(x_{j1})
f_{j2}(x_{j2})\,\si_t(x_{j2}-x_{j1})\\
 & =\,\prod_{j=1}^d \E_{x_{j1},x_{j2}\in\FF_q^2}\ f_{j1}(x_{j1}) f_{j2}(x_{j2})+O(q^{-1/2})=\MM(f_e;\,e\in
 \HH_{d,1}^{\uu})+O(q^{-1/2}).
\end{align*}

The induction step has two main ingredients, 
the first is an estimate of the type which is often referred to as a  generalized von-Neumann
inequality, namely

\begin{lem}\label{vN-rect-ff}
Let $1\leq k\leq d$. For any collection of functions $f_e: V_{\pi(e)}\to [-1,1]$ with $e\in\HH_{d,k}^{\uu}$
 one has
\eq\label{vN-rect-f}
\NN_{\ut} (f_e;\ e\in \HH_{d,k}^{\uu})\,\leq\,\min_{e\in\HH_{d,k}^{\uu}} \|f_e\|_{\Box(V_{\pi(e)})} + O(q^{-1/2})
\ee
where for any $e\in\HH_{d,k}^{\uu}$ and $f: V_{\pi(e)}\to [-1,1]$ we define 
\eq
\|f\|^{2^k}_{\Box(V_{\pi(e)})}:=\E_{\ux\in V_{\pi(e)}^2} \prod_{e\in \HH_{d,k}^{\uu}} f(\ux_e).
\ee
\end{lem}

The corresponding inequality for the multilinear expression $\MM (f_e;\ e\in \HH_{d,k}^{\uu})$, namely the fact that
\[\MM (f_e;\ e\in \HH_{d,k}^{\uu})\leq \prod_{e\in\HH_{d,k}^{\uu}} \|f_e\|_{\Box(V_{\pi(e)})}\leq \min_{e\in\HH_{d,k}^{\uu}}
\|f_e\|_{\Box(V_{\pi(e)})}\]
is well-known and is referred to as the Gowers-Cauchy-Schwarz inequality \cite{Gow06}.

\smallskip

The second and main ingredient is an approximate decomposition of a graph to simpler ones, and is essentially the so-called weak (hypergraph) regularity lemma of Frieze and Kannan \cite{FK96}. We choose to state this from a somewhat more abstract/probabilistic point of view, a perspective that will be particularly helpful when we consider our general results in the continuous and discrete settings.

We will first introduce this in the case $d=2$. A bipartite graph with (finite) vertex sets $V_1$, $V_2$ is a set $S\subs V_1\times V_2$ and a function $f:V_1\times V_2\to \R$ may be viewed as weighted bipartite graph with weights $f(x_1,x_2)$ on the edges $(x_1,x_2)$. If $\PP_1$ and $\PP_2$ are partitions of $V_1$ and $V_2$ respectively then $\PP=\PP_1\times \PP_2$ is a partition $V_1\times V_2$ and we let $\E(f|\PP)$ denote the function that is constant and equal to $\E_{x\in A} f(x)$ on each atom $A=A_1\times A_2$ of $\PP$. 
The weak regularity lemma states that for any $\eps>0$ and for any weighted graph $f:V_1\times V_2\to [-1,1]$ there exist partitions $\PP_i$ of $V_i$ with $|\PP_i|\leq 2^{O(\eps^{-2})}$ for $i=1,2$, so that
\eq\label{R1}
\quad |\E_{x_1\in V_1}\E_{x_2\in V_2} (f-\E(f|\PP))(x_1,x_2)\,\,\1_{U_1}(x_1)\1_{U_2}(x_2)|\leq \eps\ee
for all $U_1\subs V_1$ and $U_2\subs V_2$. Informally this means that the graph $f$ can be approximated with precision $\eps$ with the ``low complexity" graph $\E(f,\PP)$. 
If we consider the $\si$-algebras $\BB_i$ generated by the partitions $\PP_i$ and the $\si$-algebra $\BB=\BB_1\vee \BB_2$ generated by $\PP_1\times\PP_2$ then we have $\E(f|\BB)$, the so-called conditional expectation function of $f$. Moreover it is easy to see, using Cauchy-Schwarz, that estimate \eqref{R1} follows from
\eq\label{R2} \|f-\E(f|\BB_1\vee \BB_2)\|_{\Box(V_1\times V_2)} \leq \eps.\ee
With this more probabilistic point of view the weak regularity lemma says that the function $f$ can be approximated with precision $\eps$ by a low complexity function $\E(f|\BB_1\bigvee \BB_2)$, corresponding to $\si$-algebras $\BB_i$ on $V_i$ generated by $O(\eps^{-2})$ sets. This formulation is also referred to as a Koopman- von Neumann type decomposition, see Corollary 6.3 in \cite{TaoMontreal}.


We will need a natural extension to $k$-regular hypergraphs. See \cite{Tao06,Gow06}, and also \cite{CFZ} for extension to sparse hypergraphs. Given an edge $e'\in\HH_{d,k}$ of $k$ elements we define its boundary $\partial e':=\{\f'\in\HH_{d,k-1};\ \f'\subs e'\}$. For each $\f'=e'\backslash\{j\}\in\partial e'$ let
 $\BB_\f'$ be a  $\si$-algebra on $V_{\f'}:=\prod_{j\in\f'} V_j$ and $\bar{\BB}_{\f'}:=\{U\times V_j;\ U\in\BB_{\f'}\}$ denote its pull-back over the space $V_{e'}$. 
The $\si$-algebra $\,\BB=\bigvee_{\f'\in\partial e'} \BB_{\f'}\,$ is the smallest $\si$-algebra on $\partial e'$ containing $\bar{\BB}_{\f'}$ for all
$\f'\in\partial e'$. Note that the atoms of $\BB$ are of the form $A=\bigcap_{\f'\in\partial e'} A_{\f'}$ where $A_{\f'}$ is an atom
of $\bar{\BB}_{\f'}$. We say that the complexity of a $\si$-algebra $\BB_{\f'}$ is at most $m$, and write $\comp(\BB_{\f'})\leq m$, if it is generated by $m$ sets.

\begin{lem}[Weak hypergraph regularity lemma]\label{KvN-ff} 
Let $1\leq k\leq d$ and $f_e: V_{\pi(e)}\to [-1,1]$ be a given function for each $e\in\HH_{d,k}^{\uu}$.
For any $\VE>0$
there exists $\si$-algebras $\BB_{\f'}$ on $V_{\f'}$ for each  $\f'\in\HH_{d,k-1}$ such that
\eq\label{compl-f} \comp(\BB_{\f'})=O(\eps^{-2^{k+1}})\ee
and
\eq\label{KvN-f} \|f_e-\E(f_e|\bigvee_{\f'\in\partial \pi(e)} \BB_{\f'})\|_{\Box (V_{\pi(e)})}\,\leq\,\eps \quad\text{for all \,$e\in\HH_{d,k}^{\uu}$}.
\ee
\end{lem}

The proof of Lemmas \ref{vN-rect-ff} and \ref{KvN-ff}  are presented in Section \ref{fflemma1} below.
We close this subsection by demonstrating how these lemmas can be combined to establish  Proposition \ref{main-rect-ff}. 

\begin{proof}[Proof of Proposition \ref{main-rect-ff}]\

Let $\VE>0$, $2\leq k\leq d$ and assume that the lemma holds for $k-1$. It follows from Lemma \ref{KvN-ff} that  there exists  $\si$-algebras $\BB_{\f'}$ of complexity $O(\eps^{-2^{k+1}})$ on $V_{\f'}$ for each $\f'\in\HH_{d,k-1}$ for which \eqref{KvN-f} holds for all \,$e\in\HH_{d,k}^{\uu}$. 
For each  $e\in\HH_{d,k}^{\uu}$ we let
 $\of_e:=\E(f_e|\bigvee_{\f'\in\partial \pi(e)} \BB_{\f'})$ and write $f_e= \of_e+h_e$.
By Lemma \ref{vN-rect-ff} and multi-linearity we have that
\eq\label{reduction1}
\NN_{\ut}(f_e;\ e\in\HH_{d,k}^{\uu}) = \NN_{\ut}(\of_e;\ e\in\HH_{d,k}^{\uu}) +
O(\eps)+O(q^{-1/2})
\ee
and also by the Gowers-Cauchy-Schwarz inequality
\eq\label{reduction2}
\MM(f_e;\ e\in\HH_{d,k}^{\uu})  = \MM(\of_e;\ e\in\HH_{d,k}^{\uu}) + O(\eps).
\ee

The conditional expectation functions $\of_e$ are linear combinations of the indicator functions $1_{A_e}$ of the atoms $A_e$ of
the $\si$-algebras $\BB_e:= \bigvee_{\f'\in\partial \pi(e)} \BB_{\f'}$. Since the number of terms in this linear combination is at most
$2^{C\eps^{-2^{k+1}}}$, with coefficients at most 1 in modulus, plugging these into the multi-linear 
expressions $\NN_{\ut}(\of_e;\ e\in\HH_{d,k}^{\uu})$ and $\MM(\of_e;\ e\in\HH_{d,k}^{\uu})$ one obtains a linear combination of expressions of the form
$\NN_{\ut}(1_{A_e};\ e\in\HH_{d,k}^{\uu})$ and $\MM(1_{A_e};\ e\in\HH_{d,k}^{\uu})$ respectively with each $A_e$ being an atoms of $\BB_e$ for all
$e\in \HH_{d,k}^{\uu}$.

The key observation is that these expressions are at level $k-1$ instead of $k$. Indeed, $1_{A_e}=\prod_{\f'\in\partial \pi(e)}
1_{A_{e\f'}}$ where $A_{e\f'}=A'_{e\f'}\times V_j$, with $A'_{e\f'}$ being an atom of $\BB_{\f'}$ when $\f'=\pi(e)\backslash\{j\}$. If
$e=(j_1l_1,\ldots,jl,\ldots,j_kl_k)$, let $p_{\f'}(e):=(j_1l_1,\ldots,j_kl_k)\in\HH_{d,k-1}^{\uu}$, obtained from $e$ by removing
the $jl$-entry. Then we have $1_{A_{e\f'}}(\ux_e)=1_{A'_{e\f'}}(\ux_{p_\f'(e)})$ since $x_{jl}\in V_j$, and hence
\[1_{A_e}(\ux_e) = \prod_{\f'\in \partial \pi(e)} 1_{A'_{e\f'}}(\ux_{p_\f'(e)}).\]
It therefore follows that
\begin{align*}
\NN_{\ut}(1_{A_e};\ e\in\HH_{d,k}^{\uu}) &= \E_{\ux\in V^2} \prod_{e\in\HH_{d,k}^{\uu}} \prod_{\f'\in \partial \pi(e)}
1_{A'_{e\f'}}(\ux_{p_{\f'}(e)})\ \prod_{j=1}^d \si_{t_j}(x_{j2}-x_{j1})\\
&= \E_{\ux\in V^2} \prod_{\f\in\HH^{\uu}_{d,k-1}} \underbrace{\prod_{\substack{e\in\HH_{d,k}^{\uu},\,\f'\in \partial
\pi(e)\\ p_{\f'}(e)=\f}}1_{A'_{e\f'}}(\ux_{p_{\f'}(e)})
}_{=:g_{\f}} \ \prod_{j=1}^d \si_{t_j}(x_{j2}-x_{j1})
= \NN_{\ut}(g_{\f};\ \f\in\HH_{d,k-1}^{\uu})
\end{align*}
and similarly that
\[\MM(1_{A_e};\ e\in\HH_{d,k}^{\uu})=\MM(g_{\f};\ \f\in\HH_{d,k-1}^{\uu}).\]

It then follows from the induction hypotheses that
\[\NN_{\ut}(1_{A_e};\ e\in\HH_{d,k}^{\uu}) = \MM(1_{A_e};\ e\in\HH_{d,k}^{\uu}) + O(\eps_1)+O_{\eps_1}(q^{-1/2})\]
for any $\VE_1>0$. If we choose $\eps_1:=2^{-C_1\,\eps^{-2^{k+1}}}$\!\!, with $C_1\gg 1$ sufficiently large, then $\eps_1\,2^{C\eps^{-2^{k+1}}}\!\!\!\!=O(\eps)$ and it follows that
\[\NN_{\ut}(\of_e;\ e\in\HH_{d,k}^{\uu}) = \MM(\of_e;\ e\in\HH_{d,k}^{\uu}) + O(\eps)+O_{\eps}(q^{-1/2}).\]

This, together with \eqref{reduction1} and \eqref{reduction2}, establishes that \eqref{main-rect-f} hold for $d=k$ as required.\end{proof}


\subsection{Proof of Lemmas \ref{vN-rect-ff} and \ref{KvN-ff}}\label{fflemma1}

\begin{proof}[Proof of Lemma \ref{vN-rect-ff}]
We start by observing the following consequence of (\ref{k=1}), namely that
\eq\label{beginning}
\left|\E_{x_1,x_2\in\FF_q^2} f_1(x_1) f_2(x_2)\si_t(x_2-x_1) \right|^2 
\leq 
\E_{x_1,x_2\in\FF_q^2} f_1(x_1) f_1(x_2) + O(q^{-1/2})\ee
for any $f_1,f_2:\FF_q^2\to [-1,1]$ and  $t\in\FF_q^*$.

Now, fix an edge, say $e_0=(11,21,\ldots,k1)$. Partition the edges $e\in \HH_{d,k}^{\uu}$ into three groups; the first group
consisting of edges $e$ for which $1\notin \pi(e)$, the second where $11\in e$ and write $e=(11,e')$ with $e'\in
\HH_{d-1,k-1}^{\uu}$ and the third when $12\in e$, using the notation $\HH_{d-1,k-1}^{\uu}:=\{(j_2l_2,\ldots,j_kl_k)\}$.
Accordingly we can write
\eq\label{vN-step1-f}
\NN_t(f_e;\ e\in\HH_{d,k}^{\uu})=\E_{\ux\in V^2} \prod_{1\notin \pi(e)}f_e(\ux_e) \!\!\!\!\!\!\prod_{e'\in \HH_{d-1,k-1}^{\uu}}\!\!\!\!\!\!\!\!\!
f_{(11,e')}(x_{11},\ux_{e'}) \!\!\!\!\!\!\prod_{e'\in \HH_{d-1,k-1}^{\uu}} \!\!\!\!\!\!\!\!f_{(12,e')}(x_{12},\ux_{e'}) \,\prod_{j=1}^d \si_{t_j}(x_{j2}-x_{j1}).
\ee
If for given $x\in V_1$ and $\ux'=(x_{21},x_{22},\ldots,x_{d1},x_{d2})\in V_2^2\times\ldots\times V_d^2$ we define
\[g_1(x,\ux'):=\prod_{e'\in \HH_{d-1,k-1}^{\uu}} f_{(11,e')}(x,\ux_{e'})\quad\quad\text{and}\quad\quad g_2(x,\ux'):=\prod_{e'\in \HH_{d-1,k-1}^{\uu}}
f_{(12,e')}(x,\ux_{e'})\]
then we can write
\begin{align}\label{vN-step2-f}
\NN_t(f_e;\ e\in\HH_{d,k}^{\uu}) &= \E_{x_{21},x_{22},\ldots,x_{d1},x_{d2}} \prod_{1\notin\pi(e)} f_e(\ux_e) \prod_{j=2}^d \si_{t_j}(x_{j2}-x_{j1})\\
&\quad\quad\quad\quad\quad\quad\quad\quad\times \E_{x_{11},x_{12}}\ g_1(x_{11},\ux') g_2(x_{12},\ux')\,\si_{t_1}(x_{12}-x_{11}).\nonumber
\end{align}
By (\ref{beginning}) we can estimate the inner sum in \eqref{vN-step2-f} by the square root of
\[\E_{x_{11},x_{12}}\ g_1(x_{11},\ux') g_1(x_{12},\ux') + O(q^{-1/2}).\]
Thus by Cauchy-Schwarz, and the fact that $f_e: V_{\pi(e)}\to [-1,1]$ for all $e\in\HH_{d,k}^{\uu}$, we can conclude that
\eq\label{vN-step3-f}
\NN_t(f_e;\ e\in\HH_{d,k}^{\uu})^2 \leq \E_{x_{11},x_{12},\dots,x_{d1},x_{d2}}\prod_{e'\in \HH_{d-1,k-1}^{\uu}}\!\!\!\!
f_{(11,e')}(x_{11},\ux_{e'}) f_{(11,e')}(x_{12},\ux_{e'}) \,\prod_{j=2}^d \si_{t_j}(x_{j2}-x_{j2}).
\ee

The expression on the right hand side of the inequality above is similar to that in (\ref{vN-step1-f}) except for the following changes. The functions $f_e$ for
$1\notin e$ are eliminated i.e. replaced by 1, as well as the factor $\si_{t_1}$. The functions $f_{(12,e')}$, are replaced by
$f_{(11,e')}$ for all $e'\in \HH^{\uu}_{d-1,k-1}$. Repeating the same procedure for $j=2,\ldots,k$ one eliminates all the factors
$\si_{t_j}$ for $1\leq j\leq k$, moreover all the functions $f_e$ for edges $e$ such that $j\notin \pi(e)$ for some $1\leq j\leq k$, which leaves only the edges $e$ so that $\pi(e)=(1,2,\ldots,k)$, moreover for such edges the functions $f_e$ are eventually replaced by
$f_{e_0}=f_{11,21,\ldots,k1}$. The factors $\si_{t_j}(x_{j2}-x_{j1})$ are not changed for $j>k$ however as the function $f_{e_0}$ does not depend on the variables $x_{jl}$ for $j>k$, averaging over these variables gives rise to a factor of $1+O(q^{-1/2})$. Thus one obtains the following final estimate
\eq\label{vN-step4-f}
\NN_t(f_e;\ e\in\HH_{d,k}^{\uu})^{2^k} \leq \E_{x_{11},x_{12},\ldots,x_{k1},x_{k2}} \prod_{\pi(e)=(1,\ldots,k)} f_{e_0}(\ux_e) + O
(q^{-1/2}) = \|f_{e_0}\|_{\Box(V_{\pi(e_0)})}^{2^k} + O(q^{-1/2}).
\ee
This proves the lemma, as it is clear that the above procedure can be applied to any edge in place of $e_0=(11,21,\ldots,k1)$.
\end{proof}

\begin{proof}[Proof of Lemma \ref{KvN-ff}]
For a function $f_e: V_{\pi (e)}\to [-1,1]$ and a $\si$-algebra $\BB_{\pi (e)}$ on
$V_{\pi (e)}$ define the \emph{energy} of $f_e$ with respect to $\BB_{\pi(e)}$ as
\[\mathcal{E} (f_e,\BB_{\pi (e)}):= \|\E(f_e|\BB_{\pi (e)})\|_2^2 = \E_{x\in V_{\pi (e)}}\, |\E(f_e|\BB_{\pi (e)})(x)|^2,\]
and for a family of functions $f_e$ and $\si$-algebras $\BB_{\pi (e)}$, $e\in\HH_{d,k}^{\uu}$ its total energy as
\[\mathcal{E} (f_e,\BB_{\pi (e)};\, e\in\HH_{d,k}^{\uu}):= \sum_{e\in\HH_{d,k}^{\uu}} \mathcal{E} (f_e,\BB_{\pi (e)}).\]
We will show that if \eqref{KvN-f} does not hold for a family of $\si$-algebras $\,\BB_{\pi (e)} = \bigvee_{\f'\in\partial {\pi
(e)}} \BB_{\f'}\,$, then the $\si$-algebras $\BB_{\f'}$ can be refined so that the total energy of the system increases by a
quantity depending only on $\eps$. Since the functions $f_e$ are bounded the total energy of the system is $O(1)$, the
energy increment process must stop in $O_{\eps}(1)$ steps, and \eqref{KvN-f} must hold. The idea of this procedure appears
already in the proof of Szemer\'{e}di's regularity lemma \cite{Szem75}, and have been used since in various places
\cite{FK96,Tao06,Gow06}.

\smallskip

Initially set $\BB_{\f'}:=\{\emptyset, V_{\f'}\}$ and hence $\BB_{\pi (e)}= \{\emptyset, V_{\pi (e)}\}$ to be the trivial
$\si$-algebras. Assume that in general \eqref{KvN-f} does not hold for a family of $\si$-algebras $\BB_{\f'}$, with $\f'\in\HH_{d,k-1}$.
Then there exists an edge $e\in\HH_{d,k}^{\uu}$ so that $\|g_e\|_{\Box(V_{\pi(e)})}\geq\eps$, with $g_e:=f_e- \E(f_e|\BB_{\pi (e)})$. Let
$e=(11,\ldots,k1)$ for simplicity of notation, hence $\pi(e)=(1,\ldots,k)$. Then, with notation $\ux'=(x_{12},\ldots,x_{k2})$, one
has
\begin{align*}
\eps^{2^k}\leq\|g_e\|_{\Box(V_{\pi(e)})}^{2^k} &= \E_{x_{11},x_{12},\ldots,x_{k1},x_{k 2}} \prod_{l_1,\ldots,l_k=1,2} g_e(x_{1l_1},\ldots,x_{kl_k})\\
&\leq \E_{x_{12},\ldots,x_{k 2}} \Bigl|\E_{x_{11},\ldots,x_{k1}} g_e(x_{11},\ldots,x_{k1})\,\prod_{j=1}^k
h_{j,\ux'}(x_{11},\ldots,x_{j-1\, 1},x_{j+1\, 1},\ldots,x_{k1})\Bigr|
\end{align*}
for some functions $h_{j,\ux'}$ that are bounded by 1 in magnitude.
Indeed if and edge $e\neq (11,\ldots,k1)$ then $x_e$ does not depend at least one of the variables $x_{j1}$. Thus there must be
an $\ux'$ for which the inner sum in the above expression is at least $\eps^{2^k}$. Fix such an $\ux'$. Decomposing the  functions $h_{j,\ux'}$  into their positive and negative parts and then writing them as an
average of indicator functions, one obtains that there sets $B_j\subs V_{\pi(e)\backslash\{j\}}$ such that
\[
\Bigl|\E_{x_{11},\ldots,x_{k1}} g_e(x_{11},\ldots,x_{k1})\,\prod_{j=1}^k 1_{B_j}(x_{11},\ldots,x_{j-1\, 1},x_{j+1\,
1},\ldots,x_{k1})\Bigr|\,\geq\, 2^{-k}\,\eps^{2^k}
\]
which can be written more succinctly, using the inner product notation, as
\eq\label{KvN-step1-f}
\Bigl|\langle f_e-\E(f_e|\BB_{\pi(e)}),\prod_{j=1}^k 1_{B_j}\rangle \Bigr|\,\geq\,2^{-k}\,\eps^{2^k}.
\ee

For $\f'=\partial\pi(e)\backslash\{j\}$ let $\BB_{\f'}'$ be the $\si$-algebra generated by $\BB_{\f'}$ and the set $B_j$ and let
$\BB_{\pi(e)}':= \bigvee_{\f'\in\partial {\pi (e)}} \BB_{\f'}'\,$. Since the functions $1_{B_j}$ are measurable with respect to the
$\si$-algebra $\BB_{\pi(e)}'$ for all $1\leq j\leq k$, we have that
\eq\label{KvN-step2-f}
\langle f_e-\E(f_e|\BB_{\pi(e)}'),\prod_{j=1}^k 1_{B_j}\rangle=0
\ee
and hence, by Cauchy-Schwarz, that
\eq\label{KvN-step3-f}
\|\E(f_e|\BB_{\pi(e)}')-\E(f_e|\BB_{\pi(e)})\|_2^2 = \|\E(f_e|\BB_{\pi(e)}')\|_2^2 - \|\E(f_e|\BB_{\pi(e)})\|_2^2
\geq\,2^{-2k}\,\eps^{2^{k+1}}.
\ee

Note that the first equality above follows from the fact that conditional expectation function $\E(f|\BB)$ is the orthogonal projection
of $f$ to the subspace of $\BB$-measurable functions in $L^2$. This also implies that energy of a function is always increasing
when the underlying $\si$-algebra is refined, and \eqref{KvN-step3-f} tells us that the energy of $f_e$ is increased by at least
$c_k\,\eps^{2^{k+1}}$. 

For $\f'\notin\partial\pi(e)$ we set $\BB_{\f'}':=\BB_{\f'}$. Then the total energy of the family $f_e$ with
respect to the system $\,\BB_{\pi(e)}'=\bigvee_{\f'\in\partial \pi(e)} \BB_{\f'}'$, $e\in\HH^{\uu}_{d,k}$ is also increased by at least
$c_k\,\eps^{2^{k+1}}$. 

It is clear that the complexity of the $\si$-algebras $\BB_{\f'}$ are increased by at most 1, hence, as
explained above, the lemma follows by applying this energy increment process at most $O(\eps^{-2^{k+1}})$ times.
\end{proof}


\section{The base case of an inductive strategy to establish Theorem \ref{Prod}}\label{Base}

In this section we will ultimately establish the base case of our more general inductive argument.
We however  start by giving a quick review of the proof of  Theorem \ref{Prod} when $d=1$ (which contains both Theorem B and Corollary B as stated in Section \ref{d&s}), namely the case of a single simplex. This was originally addressed in \cite{Bourg87} and revisited in \cite{LM17} and \cite{LM18}.

\subsection{A Single Simplex in $\R^n$}\label{S1}\label{S4.1}

Let $Q\subs\R^n$ be a fixed cube and let $l(Q)$ denotes its side length.

Let $\De^0=\{v_1=0,v_2,\ldots,v_n\}\subs\R^n$ be a fixed non-degenerate simplex and define $t_{kl}:=v_k\cdot v_l$ for $2\leq
k,l\leq n$ where $``\cdot"$ is the dot product on $\R^n$. 
Given $\la>0$, a simplex $\De=\{x_1=0,x_2,\ldots,x_n\}\subs\R^n$ is
isometric to $\la\De^0$ if and only if $x_k\cdot x_l=\lm^2 t_{kl}$ for all $2\leq k,l\leq n$. Thus the configuration space
$S_{\la\De^0}$ of isometric copies of $\la\De^0$ is a non-singular real variety given by the above equations. Let $\si_{\la\De^0}$
be natural normalized surface area measure on $S_{\la\De^0}$, described in \cite{Bourg87}, \cite{LM17}, and \cite{LM18}. 
It is clear that the variable
$x_1$ can be replaced by any of the variables $x_i$ by redefining the constants $t_{kl}$.

For any family of functions $f_1,\ldots,f_n:Q\to [-1,1]$ and $0<\lm\ll l(Q)$ we define the multi-linear expression
\eq\label{multilin-config-1}
\NN^1_{\la\De^0,Q}(f_1,\ldots,f_n) := \fint_{x_1\in Q} \int_{x_2,\ldots,x_n} f_1(x_1)\ldots
f_n(x_n)\,d\si_{\la\De_0}(x_2-x_1,\ldots,x_n-x_1)\,dx_1.
\ee

We note that all of our functions are 1-bounded and both integrals, in fact all integrals in this paper, are normalized.
Recall that we are using the normalized integral notation
$\fint_A f:=\frac{1}{|A|}\int_A f.$
Since the normalized measure
$\si_{\la\De^0}$ is supported on $S_{\la\De_0}$ we will not indicate the support of the variables $(x_2,\ldots,x_n)$ explicitly.

Note also that if $S\subs Q$ is a measurable set and $\NN^1_{\la\De^0,Q}(1_S,\dots,1_S)>0$ then $S$ must contain an isometric copy of $\la\De^0$. 
The following proposition (with $Q=[0,1]^n$) is a quantitatively stronger version of Proposition B that appeared in Section \ref{d&s} and hence immediately establishes Theorem \ref{Prod} for $d=1$.

\begin{prop}\label{PropnSimp}

For any $0<\VE\leq 1$  there exists an integer $J=O(\VE^{-2}\log\VE^{-1})$ with the following property: 

Given any lacunary sequence $l(Q)\geq \lm_1\geq\cdots\geq\lm_J$  and $S\subseteq Q$, there is some $1\leq j< J$ such that 
\eq
\NN^1_{\la\De^0,Q}(1_S,\dots,1_S)>\left(\frac{|S|}{|Q|}\right)^n-\VE
\ee 
for all $\lambda\in [\lm_{j+1},\lm_j]$.
\end{prop}

Our approach to establishing Proposition \ref{PropnSimp} is to compare the above expressions to simpler ones for which it is easy to
obtain lower bounds.
Given a scale $0<\lm\ll l(Q)$ we define the multi-linear expression
\eq\label{multilin-box-1}
\MM^1_{\lm,Q}(f_1,\ldots,f_n) := \fint_{t\in Q}\fint_{x_1,x_2,\ldots,x_n\in t+Q(\lm)} f_1(x_1)\ldots f_n(x_n)\,dx_1\dots dx_n\,dt
\ee
where $Q(\lm)=[-\frac{\lm}{2},\frac{\lm}{2}]^n$ and $t+Q(\lm)$ is the shift of the cube $Q(\lm)$ by the vector $t$. 
Note that if $S\subs Q$ is
a set of
measure $|S|\geq \de |Q|$ for some $\de>0$, then for a given $\eps >0$, H\"older implies
 \eq\label{ML-lowerb-1}
\MM^1_{\lm,Q}(1_S,\dots,1_S) = \fint_{t\in Q} \left(\fint_{x\in t+Q(\lm)} 1_S(x)\,dx\right)^n dt \geq \left(\fint_{t\in Q}\fint_{x\in
t+Q(\lm)} 1_S(x)\,dx\,dt\right)^n \geq\de^n-O(\eps),\ee
for all scales $0<\lm\ll\eps\, l(Q)$.

Recall that for any $\VE>0$ we call a sequence $L_1\geq \cdots \geq L_J$ $\eps$-\emph{admissible} if $L_j/L_{j+1}\in\N$ and $L_{j+1}\ll\eps^2 L_j$ for all $1\leq j<J$.  Note that given any lacunary sequence $l(Q)\geq \lm_1\geq\cdots\geq\lm_{J'}$ with $J'\gg ( \log\VE^{-1})\,J$,
one can always finds an $\eps$-admissible sequence of scales  $l(Q)\geq L_1\geq\cdots \geq L_{J}$ such that for each $1\leq j<J$ the interval $[L_{j+1}, L_j]$  contains at least two consecutive elements  from the original lacunary sequence.

In light of this observation, and the one above regarding a lower bound for  $\MM^1_{\la,Q}(1_S,\dots,1_S)$, our proof of Proposition \ref{PropnSimp}   reduces to establishing the following ``counting lemma".



\begin{prop}\label{approx-1}

Let $0<\eps<1$. There exists an integer $J_1=O(\eps^{-2})$ such that for any $\eps$-admissible sequence of scales $l(Q)\geq L_1\geq\cdots \geq L_{J_1}$ and $S\subseteq Q$ there is some $1\leq j< J_1$ such that
\eq\label{1.4}
\NN^1_{\la\De^0,Q}(1_S,\dots,1_S)=\MM^1_{\la,Q}(1_S,\dots,1_S)+O(\eps)\ee
for all $\lm\in[L_{j+1},L_j]$.
\end{prop}


There are two main ingredients in the proof of Proposition \ref{approx-1}, this will be typical to all of our arguments. The first ingredient is a result which establishes that the our multi-linear forms $\NN^1_{\la\De^0,Q}(f_1,\dots,f_n)$ are controlled by an appropriate norm which measures the
uniformity of distribution of functions $f:Q\to [-1,1]$ with respect to particular scales $L$. This is analogous to estimates in additive combinatorics \cite{Gow06}  which are often referred to as generalized von-Neumann inequalities. 

The result below was proved in \cite{LM17} for $Q=[0,1]^n$, however a simple scaling of the variables $x_i$ transfers the result to an
arbitrary cube $Q$.

\begin{lem}[A Generalized von-Neumann inequality \cite{LM17}]\label{vN-1}

Let $\VE>0$, $0<\la\ll l(Q)$, and $0<L\ll \eps^6 \la$. 

For any collections of functions $f_1,\ldots,f_n:\,Q\to [-1,1]$ we have
\eq\label{vN-11a}
|\NN^1_{\la\De^0,Q}(f_1,\ldots,f_n)| \leq \min_{i=1,\ldots,n} \|f_i\|_{U^1_L(Q)} + O(\eps)
\ee
where for any $f:Q\to[-1,1]$ we define
\eq\label{U-norm}
\|f\|_{U^1_L(Q)}^2 := \fint_{t\in Q} \Bigl|\fint_{x\in t+Q(L)} f(x)\,dx\Bigr|^2 dt
\ee
with $t+Q(L)$ denoting the shift of the cube $Q(L)=[-\frac{L}{2},\frac{L}{2}]^n$  by the vector $t$. 
\end{lem}

The corresponding inequality for the multilinear expression $\MM^1_{\la,Q}(f_1,\dots,f_n)$, namely the fact that
\[\MM^1_{\la,Q}(f_1,\dots,f_n)\leq  \min_{i=1,\ldots,n} \|f_i\|_{U^1_L(Q)} + O(\eps)\]
whenever $0<L\ll \eps^6 \la$ follows easily from Cauchy-Schwarz 
together with the simple observation that 
\[\|f\|_{U^1_L(Q)}\leq \|f\|_{U^1_{L'}(Q)}+O(\eps)\]
whenever $L'\ll\eps L$.

The second key ingredient, proved in \cite{LM18} and generalized in Lemma \ref{KvN-1'} below, is a  Koopman-von Neumann type decomposition of functions
where the underlying $\si$-algebras are generated by cubes of a fixed length. To recall it, let $Q\subs\R^n$ be a cube, 
$L>0$ be scale that divides $l(Q)$, $Q(L)=[-\frac{L}{2},\frac{L}{2}]^n$, and $\GG_{L,Q}$ denote the collection of cubes $t+Q(L)$ partitioning the cube $Q$ and $\Ga_{L,Q}$ denote the grids corresponding to the centers of the cubes. 
By a slightly abuse of notation we also write $\mathcal{G}_{L,Q}$ for the $\si$-algebra generated by the grid.
Recall that the conditional expectation function $\E(f|\GG_{L,Q})$ is constant and equal to $\fint_A f$ on each cube $A\in\GG_{L,Q}$.

\begin{lem}[A Koopman-von Neumann type decomposition \cite{LM18}]\label{KvN-1}

Let $0<\VE\leq 1$ and $Q\subseteq\R^n$ be a cube.

There exists an integer $\bar{J}_1=O(\VE^{-2})$ such that for any $\eps$-admissible sequence $l(Q)\geq L_1\geq\cdots \geq L_{\bar{J}_1}$ and function $f:Q\to[-1,1]$ there is some $1\leq j< \bar{J}_1$ such that
\eq\label{KvN-11}
\|f-\E(f|\GG_{L_j,Q})\|_{U^1_{L_{j+1}}(Q)} \leq \eps
\ee
\end{lem}


\begin{proof}[Proof of Proposition \ref{approx-1}] Let $\GG_{L_j,Q}$ be the grid obtained from Lemma \ref{KvN-1} for the functions $f=1_S$ for some fixed  $\eps>0$. Let $\of:=\E(f|\GG_{L_j,Q})$, then by \eqref{vN-11a} and multi-linearity, we have 
\[\NN^1_{\la\De^0,Q}(f,\ldots,f) = \NN^1_{\la\De^0,Q}(\of,\ldots,\of) + O(\eps),\]
and also
\[\MM^1_{\la,Q}(f,\ldots,f) = \MM^1_{\la,Q}(\of,\ldots,\of) + O(\eps)\]
provided for $ \eps{^{-6}} L_{j+1}\ll \lm$.
Thus in showing \eqref{1.4} one can replace the functions $f$ with $\of$. 
If we make the additional assumption that $\la\ll \eps L_j$ then it is easy to see,
using the fact that the function $\of$ is constant on the cubes $Q_t(L_j)\in\GG_{L_j,Q}$, that
\[\NN^1_{\la\De^0,Q}(\of,\ldots,\of) = \MM^1_{\la,Q}(\of,\ldots,\of) + O(\eps).\]

Since the condition $\VE^{-6}L_{j+1}\ll\lm\ll\VE L_j$ can be replaced with $L_{j+1}\ll\lm\ll L_j$ if one passes to a subsequence of scales, for example $L_j'=L_{5j}$, this  completes the proof of Proposition  \ref{approx-1}.
\end{proof}

\subsection{The base case of a general inductive strategy}\label{base}\

In this section, as preparation to handle the case of products of simplices, we prove a parametric version of Proposition \ref{approx-1}, namely Proposition \ref{approx-par-1} below, which will serve as the base case for later inductive arguments.

Let $Q=Q_1\times\cdots\times Q_d$ with $Q_i\subs\R^{n_i}$ be cubes of equal side length $l(Q)$. Let $L$ be a scale dividing $l(Q)$ and for each $\ut=(t_1,\ldots,t_d)\in\Ga_{L,Q}$ let $Q_{\ut}(L)=\ut+Q(L)$ and $Q_{t_i}(L)=t_i+Q_i(L)$. Note that $Q_{\ut}(L)=Q_{t_1}(L)\times\dots\times Q_{t_d}(L)$. Here $Q(L)=[-\frac{L}{2},\frac{L}{2}]^n$ and $Q_i(L)=[-\frac{L}{2},\frac{L}{2}]^{n_i}$ for each $1\leq i\leq d$.

Let $\De_i^0=\{v^i_1,\ldots,v^i_{n_i}\}\subs\R^{n_i}$ be a non-degenerate simplex for each $1\leq i\leq d$.


\begin{prop}[Parametric  Counting Lemma on $\R^n$ for Simplices]\label{approx-par-1}\

Let $0<\VE\leq 1$ and $R\geq1$.
There exists an integer $J_1=J_1(\VE,R)=O(R\,\VE^{-4})$ such that for any $\eps$-admissible sequence of scales $L_0\geq L_1\geq\cdots \geq L_{J_1}$ with the property that  $L_0$ divides $l(Q)$ and collection of functions  
\[\text{$f_{k,\ut}^{i,r}:\,Q_{t_i}(L_0)\to [-1,1]$ \ with \ $1\leq i\leq d$, $1\leq k\leq n_i$, $1\leq r\leq R$ and $\ut\in \Ga_{L_0,Q}$}\] 
there exists  $1\leq j< J_1$ and a set $T_\eps\subs\Ga_{L_0,Q}$ of size $|T_\eps |\leq \eps |\Ga_{L_0,Q}|$ such that
\eq\label{approx-p11}
\NN^1_{\la\De_i^0,Q_{t_i}(L_0)}(f_{1,\ut}^{i,r},\ldots,f_{n_i,\ut}^{i,r}) = \MM^1_{\la,Q_{t_i}(L_0)}(f_{1,\ut}^{i,r},\ldots,f_{n_i,\ut}^{i,r}) + O(\eps)
\ee
for all $\lm\in[L_{j+1}, L_j]$ and $\ut\notin T_\eps$ 
uniformly in $1\leq i\leq d$ and $1\leq r\leq R$.
\end{prop}

The proof of Proposition \ref{approx-par-1} will follow from Lemma \ref{vN-1} and the 
 following generalization of Lemma \ref{KvN-1} in which we simultaneously consider a family of functions supported on the subcubes in a partition of an original cube $Q$.


\begin{lem}[A simultaneous Koopman-von Neumann type decomposition]\label{KvN-1'}\

Let $0<\VE\leq 1$, $m\geq1$, and $Q\subseteq\R^n$ be a cube.
There exists an integer $\bar{J}_1=O(m\VE^{-3})$ such that for any $\eps$-admissible sequence $L_0\geq L_1\geq\cdots \geq L_{\bar{J}_1}$ with the property that  $L_0$ divides $l(Q)$, and collection of functions \[f_{1,t},\dots,f_{m,t}:\,Q_t(L_0)\to [-1,1]\] defined for each $t\in\Ga_{L_0,Q}$, there is some $1\leq j< \bar{J}_1$ and a set $T_{\eps}\subs \Ga_{L_0,Q}$ of size $|T_{\eps}|\leq \eps |\Ga_{L_0,Q}|$ such that 
\eq\label{KvN-12}
\|f_{i,t}-\E(f_{i,t}|\GG_{L_j,Q_t(L_0)})\|_{U^1_{L_{j+1}}(Q_t(L_0))} \leq\eps\ee
for all $1\leq i\leq m$ and $t\notin T_{\eps}$.
\end{lem}

\smallskip

\begin{proof}[Proof of Proposition \ref{approx-par-1}] Fix $1\leq i\leq d$. For $1\leq k\leq n_i$ and $\ut=(t_1,\ldots,t_d)\in \Ga_{L_0, Q}\,$, we will abuse notation and write \[f_{k,\ut}^{i,r}(x_1,\ldots,x_d):= f_{k,\ut}^{i,r}(x_i)\] for $(x_1,\ldots,x_d)\in Q_{\ut}(L_0)$.

If we apply Lemma \ref{KvN-1'} to the family of functions $f_{k,\ut}^{i,r}$ on $Q_{\ut}(L_0)$ for
$\,1\leq i\leq d$, $1\leq  k\leq n_i$, and $1\leq r\leq R$, with $m=(n_1+\ldots +n_d)R$, then this produces a grid $\GG_{L_j,Q}$ for some $1\leq j\leq \bar{J}_1=O(\eps^{-3}R)$, and a set $T_\eps\subs\Ga_{L_0,Q}$ of size $|T_\eps |\leq\,\eps |\Ga_{L_0,Q}|$, such that
\[ \|f_{k,\ut}^{i,r} - \E(f_{k,\ut}^{i,r}|\GG_{L_j,Q})\|_{U^1_{L_{j+1}}(Q_{\ut}(L_0))} \leq \eps\]
uniformly for $1\leq i\leq d$, $1\leq k\leq n_i$ and $1\leq r\leq R$ for $t\notin T_{\eps}.$

Since $f_{k,\ut}^{i,r}(x_1,\ldots,x_d)=f_{k,\ut}^{i,r}(x_i)$ for $(x_1,\ldots,x_d)\in Q_{\ut}(L_0)$ it is easy to see that
\[\|f_{k,\ut}^{i,r}-\E(f_{k,\ut}^{i,r}|\GG_{L_j,Q})\|_{U^1_{L_{j+1}}(Q_{\ut}(L_0))} = \|f_{k,\ut}^{i,r}-\E(f_{k,\ut}^{i,r}|\GG_{L_j,Q_i})\|_{U^1_{L_{j+1}}(Q_{t_i}(L_0))}.\]

Let $\of_{k,\ut}^{i,r}:= \E(f_{k,\ut}^{i,r}| \GG_{L_j,Q_{i}})\,$, then by Lemma \ref{vN-1}, one has
\[\NN^1_{\la\De_i^0,Q_{t_i}(L_0)}(f_{1,\ut}^{i,r},\ldots,f_{n_i,\ut}^{i,r}) = \NN^1_{\la\De_i^0,Q_{t_i}(L_0)}(\of_{1,\ut}^{i,r},\ldots,\of_{n_i,\ut}^{i,r}) + O(\eps),\]\\
and
\[\MM^1_{\la,Q_{t_i}(L_0)}(f_{1,\ut}^{i,r},\ldots,f_{n_i,\ut}^{i,r}) = \MM^1_{\la,Q_{t_i}(L_0)}(\of_{1,\ut}^{i,r},\ldots,\of_{n_i,\ut}^{i,r}) + O(\eps)\]

\smallskip
\noindent
for all $\ut\notin T_\eps$  provided $\eps^{-6} L_{j+1}\ll\lm$. 
Finally, if we also have $\la\ll \eps L_j$ then it is easy to see that
\[\NN^1_{\la\De_i^0,Q_{t_i}(L_0)}(\of_{1,\ut}^{i,r},\ldots,\of_{n_i,\ut}^{i,r}) = \MM^1_{\la,Q_{t_i}(L_0)}(
\of_{1,\ut}^{i,r},\ldots,\of_{n_i,\ut}^{i,r})+O(\eps)\]
as the functions $\of_{k,\ut}^{i,r}$ are constant on cubes $Q_{t_i}(L_j)$ of $\GG_{L_j,Q_i}$, which are of size $L_j\ll\eps L_0$. 

Passing first to a subsequence of scales, for example $L_j'=L_{5j}$, the condition $\VE^{-6}L_{j+1}\ll\lm\ll\VE L_j$ can be replaced with $L_{j+1}\ll\lm\ll L_j$ so this  completes the proof of the Proposition.
\end{proof}

\smallskip

We conclude this section with a sketch of the proof of Lemma \ref{KvN-1'}. These arguments are standard, see for example the  proof of Lemma \ref{KvN-1} given in \cite{LM17}.

\begin{proof}[Proof of Lemma \ref{KvN-1'}] 

First we make an observation about the $U^1_L(Q)$-norm. 
Suppose $0<L'\ll\eps^2 L$ with $L'$ dividing $L$. If $s\in\Ga_{L',Q}$ and $t\in Q_s(L')$ then $|t-s|=O(L')$ and hence
\[\fint_{x\in Q_t(L)} g(x)\,dx = \fint_{x\in Q_s(L)} g(x)\,dx + O(L'/L)\]
for any function $g:Q\to [-1,1]$. Moreover, since the cube $Q_s(L)$ is partitioned into the smaller cubes $Q_t(L')$, we have by Cauchy-Schwarz
\[\Bigl|\fint_{x\in Q_s(L)}\, g(x)\,dx\Bigr|^2 \leq \E_{t\in \Ga_{L',Q_s(L)}}\Bigl|\fint_{x\in Q_t(L')} g(x)\,dx\Bigr|^2.\]

From these observations it is easy to see that
\[ \|g\|_{U^1_L(Q)}^2 = \fint_{t\in Q}\Bigl|\fint_{x\in Q_t(L)} g(x)\,dx\Bigr|^2\,dt \leq \E_{t\in \Ga_{L',Q}}\Bigl|\fint_{x\in Q_t(L')}g(x)\,dx\Bigr|^2 + O(L'/L)\]
and we note that the right side of the above expression is $\|\E(g|\GG_{L',Q})\|_{L^2(Q)}^2$ since the conditional expectation function $\E(g|\GG_{L',Q})$ is constant and equal to $\fint_{x\in Q_t(L')} g(x)\,dx$ on the cubes $Q_t(L')$.

Suppose that \eqref{KvN-12} does not hold for some $1\leq i\leq m$ for every $t$ in some set $T_{\eps}\subs \Ga_{L_0,Q}$ of size $|T_{\eps}|> \eps\,|\Ga_{L_0,Q}|$. If we apply the above observation to $g:=f_{i,t}-\E(f_{i,t}|\GG_{L_{j},Q_t(L_0)})$, for every $t\in T_{\eps}$, we obtain by orthogonality that
\[ \sum_{i=1}^m \|\E(f_{i,t}|\GG_{L_{j+2},Q_t(L_0)})\|_{L^2(Q_t(L_0))}^2 \geq \sum_{i=1}^m \|\E(f_{i,t}|\GG_{L_j,Q_t(L_0)})\|_{L^2(Q_t(L_0))}^2 +c\eps^2\]
for some constant $c>0$. 

If we now define $f_i:Q\to [-1,1]$ such that $f_i|_{(Q_t(L_0))}=f_{i,t}$, for $1\leq i\leq m$, average over $t\in \Ga_{L_0,Q}$, and use the fact $\|f_i\|_{L^2(Q)}^2=\E_{t\in \Ga_{L_0,Q}}  \|f_{i,t}\|_{L^2(Q_t(L_0))}^2$, we obtain
\eq\label{KvN-13}
\sum_{i=1}^m \|\E(f_{i}|\GG_{L_{j+2},Q})\|_{L^2(Q)}^2 \geq \sum_{i=1}^m \|\E(f_i|\GG_{L_j,Q})\|_{L^2(Q)}^2 +c\eps^3.\ee

It is clear that the sums in the above expressions are bounded by $m$ for all $j\geq 1$, thus \eqref{KvN-13} cannot hold for some $1\leq j\leq \bar{J}_1$ for $\bar{J}_1:= C\,m\,\eps^{-3}$. This implies that \eqref{KvN-12} must hold for some $1\leq j\leq \bar{J}_1$, for all $1\leq i\leq m$ and all $t\notin T_{\eps}$ for a set  $T_{\eps}\subs \Ga_{L_0,Q}$ of size $|T_{\eps}|\leq \eps\,|\Ga_{L_0,Q}|$.
\end{proof}



\section{Product of two simplices in $\R^n$}\label{two}

Although not strictly necessary, we discuss in this section the special case $d=2$ of Theorem \ref{Prod}. This already gives an
improvement of the main results of \cite{LM17}, but more importantly serves as a gentle preparation for the more complicated general case, presented in the Section \ref{generalcase}, which involve both a plethora of different scales and the hypergraph bundle notation introduced in Section \ref{hbnotation}.


\subsection{Proof of Theorem \ref{Prod} with $d=2$}\

Let $Q=Q_1\times Q_2$ with $Q_1\subs\R^{n_1}$ and $Q_2\subs\R^{n_2}$ be cubes of equal side length $l(Q)$ and $\De^0=\De^0_1\times\De^0_2$ with $\De^0_1=\{v_{11},\ldots, v_{1n_1}\}\subs\R^{n_1}$ and $\De^0_2=\{v_{11},\ldots, v_{2n_2}\}\subs\R^{n_2}$  two non-degenerate simplices. 

In order to ``count" configurations of the form $\De=\De_1\times\De_2\subs\R^{n_1+n_2}$ with $\De_1$ and $\De_2$ isometric copies of $\la\De^0_1$ and $\la\De^0_2$ respectively for some $0<\lm\ll l(Q)$ in a set $S\subs Q$ we introduce
 the multi-linear expression
\begin{align*}\label{multilin-config-2}
\NN^2_{\la\De^0,Q}(\{f_{kl}\}
) &:= \fint_{x_{11}\in Q_1}\fint_{x_{21}\in Q_2}\int_{x_{12},\ldots,x_{1n_1}} \int_{x_{22},\ldots,x_{2n_2}}\,
\prod_{k=1}^{n_1}\prod_{l=1}^{n_2} \
f_{kl}(x_{1k},x_{2l})\nonumber\\
 & \quad \quad \quad \quad d\si_{\la\De^0_1}(x_{12}-x_{11},\ldots,x_{1n_1}-x_{11})\, d\si_{\la\De^0_2}(x_{22}-x_{21},\ldots,x_{2n_2}-x_{21})\,dx_{21}\,dx_{11}
\end{align*}
for any family of functions $f_{kl}:Q_1\times Q_2\to [-1,1]$ with $1\leq k\leq n_1$ and $1\leq l\leq n_2$. 

Indeed, if $f_{kl}=1_S$ for all $1\leq k\leq n_1$ and $1\leq l\leq n_2$ then the above expression is 0 unless there exists a configuration $\De\subs S$ of the form $\De_1\times\De_2$ with $\De_1$ and $\De_2$ isometric copies of $\la\De^0_1$ and $\la\De^0_2$ respectively.

The short argument presented in Section \ref{d&s} demonstrating how both Theorem B and Corollary B follow from Proposition B, and hence from Proposition \ref{PropnSimp}, applies equally well to each of our main theorems. This reduces our main theorems to  analogous quantitative results involving an arbitrary lacunary sequence of scales. 
In the case $d=2$ of Theorem \ref{Prod} this  stronger quantitative result takes the following form:

\begin{prop}\label{PropnSimp2}

For any $0<\VE\ll 1$  there exists an integer $J=O(\exp(C\eps^{-13}))$ with the following property: 

Given any lacunary sequence $l(Q)\geq \lm_1\geq\cdots\geq\lm_J$  and $S\subseteq Q$, there is some $1\leq j< J$ such that 
\eq
\NN^2_{\la\De^0,Q}(\{1_S\})>\left(\frac{|S|}{|Q|}\right)^{n_1n_2}-\VE
\ee 
for all $\lambda\in [\lm_{j+1},\lm_j]$.
\end{prop}

Our approach to establishing Proposition \ref{PropnSimp2} is again to compare the above expressions to simpler ones for which it is easy to
obtain lower bounds.
For any  $0<\lm\ll l(Q)$ and family of functions $f_{kl}:Q_1\times Q_2\to [-1,1]$ with $1\leq k\leq n_1$ and $1\leq l\leq n_2$ we  consider  
\[\MM^2_{\lm,Q}(\{f_{kl}\}):= \fint_{\ut\in Q} \fint_{\ux_{1}\in (t_1+Q_1(\lm))^{n_1}} \fint_{\ux_{2}\in (t_2+Q_2(\lm))^{n_2}}\prod_{k=1}^{n_1}\prod_{l=2}^{n_2} f_{kl}(x_{1k},x_{2l})\,d\ux_{2}\,d\ux_{1}\,d\ut\]
where $\ut=(t_1,t_2)\in Q_1\times Q_2$, $\ux_i=(x_{i1},\dots,x_{in_i})$ and
 $Q_i(\lm)=[-\frac{\lm}{2},\frac{\lm}{2}]^{n_i}$ for $i=1,2$. 
 
Note that if $S\subs Q$ is
a set of
measure $|S|\geq \de |Q|$ for some $\de>0$, then careful applications of H\"older's inequality give
\[
\MM^2_{\lm,Q}(\{1_S\}) \geq \fint_{\ut\in Q} \left(\fint_{(x_1,x_2)\in \ut+Q(\lm)} 1_S(x_1,x_2)\,dx_1dx_2\right)^{n_1n_2} d\ut \,
\geq\de^{n_1n_2}-O(\eps)\]
for all scales $0<\lm\ll\eps\, l(Q)$.


\smallskip

In light of the observation above, and the discussion preceding Proposition \ref{approx-1}, we see that Proposition \ref{PropnSimp2}, and hence Theorem \ref{Prod} when $d=2$, will follows as a consequence of the following

\begin{prop}\label{approx-2}

Let $0<\eps\ll1$. 
There exists an integer $J_2=
O(\exp(C\VE^{-12}))$
such that for any 
$\VE$-admissible sequence of scales $l(Q)\geq L_1\geq\cdots \geq L_{J_2}$ and $S\subseteq Q$ there is some $1\leq j< J_2$ such that
\eq\label{approx-22}
\NN^2_{\la\De^0,Q}(\{1_S\})= \MM^2_{\la,Q}(\{1_S\}) + O(\eps)
\ee
for all $\lm\in[L_{j+1},L_j]$.
\end{prop}



There are again two main ingredients in the proof of Proposition \ref{approx-2}. 
The first establishes that the our multi-linear forms $\NN^2_{\la\De^0,Q}(\{f_{kl}\})$ are controlled by an appropriate box-type norm attached to a scale $L$. 

Let $Q=Q_1\times Q_2$ be a cube. For any  scale $0<L\ll l(Q)$  and function $f:Q\to\R$ we define its local box norm at scale $L$ to be
\eq\label{box-norm-L}
\|f\|_{\Box_L(Q_1\times Q_2)}^4 := \fint_{\ut\in Q} \|f\|_{\Box(\ut+Q(L))}^4 \,d\ut
\ee
where $Q(L)=[-\frac{L}{2},\frac{L}{2}]^{n_1+n_2}$ and 
\eq\label{box-norm}
\|f\|_{\Box(\widetilde{Q})}^4 := \fint_{x_{11},x_{12}\in \widetilde{Q}_1}\fint_{x_{21},x_{22}\in \widetilde{Q}_2} f(x_{11},x_{21}) f(x_{12},x_{21}) f(x_{11},x_{22}) f(x_{12},x_{22}) \,dx_{11}\ldots dx_{22}
\ee
for any cube $\widetilde{Q}\subseteq Q$ of the form $\widetilde{Q}=\widetilde{Q}_1\times \widetilde{Q}_2$ with $\widetilde{Q}_j\subseteq Q_j$ for $j=1,2$.


 \begin{lem}[A Generalized von-Neumann inequality \cite{LM17}]\label{vN-2}

Let $\VE>0$, $0<\la\ll l(Q)$, and $0<L\ll \eps^{24} \la$. 

For any collections of functions  $f_{kl}:Q_1\times Q_2\to [-1,1]$ with $1\leq k\leq n_1$ and $1\leq l\leq n_2$  we have both
\eq\label{vN-11}
|\NN^2_{\la\De^0,Q}(\{f_{kl}\})| \leq \min_{1\leq k\leq n_1 ,\,1\leq l\leq n_2} \|f_{kl}\|_{\Box_L(Q_1\times Q_2)} + O(\eps)
\ee
\eq\label{vN-12}
|\MM^2_{\la,Q}(\{f_{kl}\})| \leq \min_{1\leq k\leq n_1 ,\,1\leq l\leq n_2} \|f_{kl}\|_{\Box_L(Q_1\times Q_2)} .
\ee
\end{lem}



The result above was essentially  proved  in \cite{LM17} for the multi-linear forms $\NN^2_{\la\De^0,Q}$ when $Q=[0,1]^{n_1+n_2}$, however a simple scaling argument transfers the result to an
arbitrary cube $Q$. 
For completeness we include its short proof in Section \ref{vN-2 KvN-2} below.

\smallskip

The second and main ingredient is  an analogue of a weak form of Szemer\'edi's regularity lemma due to Frieze and Kannan \cite{FK96}. The more probabilistic formulation, we will use below, can be found for example in \cite{Tao05}, \cite{Tao06}, and \cite{TaoMontreal}, and is also sometimes referred to as a Koopman-von Neumann type decomposition. 

For any cube $Q\subs\R^n$ and scale 
$L>0$ that divides $l(Q)$ we will let $Q(L)=[-\frac{L}{2},\frac{L}{2}]^{n}$ and $\GG_{L,Q}$ denote the collection of cubes $Q_{\ut}(L)=\ut+Q(L)$ partitioning the cube $Q$ and let $\Ga_{L,Q}$ denote grid corresponding to the centers of these cubes.
 We will say that a finite $\si$-algebra $\BB$ on $Q$ is of \emph{scale} $L$ if it contains $\GG_{L,Q}$
and for simplicity of notation will write $\BB_{\ut}$ for $\BB|_{Q_{\ut}(L)}$. 
 
Recall that if we have two $\si$-algebras $\BB_1$ on a cube $Q_1$ and $\BB_2$ on $Q_2$ then by $\BB_1\vee \BB_2$ we mean the $\si$-algebra on $Q=Q_1\times Q_2$ generated by the sets $B_1\times B_2$ with $B_1\in\BB_1$ and $B_2\in\BB_2$.
Recall also that we say the complexity of a $\si$-algebra $\BB$ is at most $m$, and write $\comp(\BB)\leq m$, if it is generated by $m$ sets.
 
\begin{lem}[Weak regularity lemma in $\R^n$]\label{KvN-2}\

Let $0<\VE\ll 1$ and $Q=Q_1\times Q_2$ with $Q_1\subs\R^{n_1}$ and $Q_2\subs\R^{n_2}$ be cubes of equal side length $l(Q)$.

There exists an integer $\bar{J}_2=O(\VE^{-12})$ such that for any $\eps^4$-admissible sequence $l(Q)\geq L_1\geq\cdots \geq L_{\bar{J}_2}$ and function $f:Q\to[-1,1]$ there is some $1\leq j\leq \bar{J}_2$ and a $\si$-algebra $\BB$ of scale $L_j$ on $Q$  
such that
\eq\label{KvN-21}
\|f-\E(f|\BB)\|_{\Box_{L_{j+1}}(Q_1\times Q_2)} \leq \eps
\ee
which has the additional local structure that for each $\ut=(t_1,t_2)\in\Ga_{L_j,Q}$ there exist $\si$-algebras $\BB_{1,\ut}$ on $Q_{t_1}(L_j)$ and $\BB_{2,\ut}$ on $Q_{t_2}(L_j)$ with  $\comp(\BB_{i,\ut})=O(j)$ for $i=1,2$ such that
$\BB_{\ut} = \BB_{1,\ut}\vee \BB_{2,\ut}.$
\end{lem}

Comparing the above statement to Lemma \ref{KvN-ff} for $d=2$, i.e to the weak regularity lemma, note that the $\si$-algebra $\BB$ of scale $L_j$ has a direct product structure only locally, inside each cube $Q_{\ut}(L_j)$. Moreover this product structure varies with $\ut\in\Ga_{L_j,Q}$, however the ``local complexity" remains uniformly bounded.

\smallskip

Assuming for now the validity of  Lemmas \ref{vN-2} and  \ref{KvN-2} we prove Proposition \ref{approx-2}. We will make crucial use of  Proposition \ref{approx-par-1}, namely our parametric counting lemma on $\R^n$ for simplices.

\smallskip




\begin{proof}[Proof of Proposition \ref{approx-2}]

Let $0<\VE\ll 1$, $\VE_1:=\exp(-C_1\VE^{-12})$ for some $C_1\gg1$, and $\{L_j\}_{j\geq 1}$ be an $\VE_1$-admissible sequence of scales. 
Set  $R=\VE\,\VE_1^{-1}$ and $J_1(\eps_1,R)$  be the parameter appearing in  Proposition \ref{approx-par-1}, noting that $J_1(\eps_1,R)=O(\VE_1^{-5})$.

For $L\in\{L_j\}_{j\geq 1}$ write $\ind(L)=j$ if $L=L_j$.
We now choose a subsequence $\{L_j'\}\subs \{L_j\}$ so that $L'_1=L_1$ and 
$\ind(L'_{j+1})\geq \ind(L'_j) + J_1(\eps_1,R)+2.$
Applying Lemma \ref{KvN-2}, with $f_{kl}=f:=1_S$ for all $1\leq k\leq n_1$ and $1\leq l\leq n_2$, guarantees the existence of  a $\si$-algebra $\BB$ of scale $L'_j$ on $Q$  
such that
\eq\label{KvN-21}
\|f-\E(f|\BB)\|_{\Box_{L'_{j+1}}(Q_1\times Q_2)} \leq \eps
\ee
for some $1\leq j\leq C\eps^{-12}$.
Moreover, we know that $\BB$ has the additional local structure that for each $\ut=(t_1,t_2)\in\Ga_{L'_j,Q}$ there exist $\si$-algebras $\BB_{1,\ut}$ on $Q_{t_1}(L'_j)$ and $\BB_{2,\ut}$ on $Q_{t_2}(L'_j)$ with  $\comp(\BB_{i,\ut})=O(\VE^{-12})$ for $i=1,2$ such that
$\BB_{\ut} = \BB_{1,\ut}\vee \BB_{2,\ut}.$ 
Thus, if we let $R_{1,\ut}$ and $R_{2,\ut}$ denote the number of atoms in $\BB_{1,\ut}$ and $\BB_{2,\ut}$ respectively, then we can assume, by formally adding the empty set to these collections of atoms if necessary, that $R_{1,\ut}=R_{2,\ut}=R':=\exp(C\eps^{-12})$ for all $\ut\in\Ga_{L_j',Q}$.

If we let $\of:= \E(f|\BB_1\vee\BB_2)$, then by Lemma \ref{vN-2} and multi-linearity we have 
\eq\label{approx-11a}
\NN^2_{\la\De^0,Q}(\{f\}) = \NN^2_{\la\De^0,Q}(\{\of\}) + O(\eps)\quad \text{and}\quad
\MM^2_{\la,Q}(\{f\}) = \MM^2_{\la,Q}(\{\of\}) + O(\eps)
\ee
provided for $ \eps{^{-24}} L'_{j+1}\ll \lm$.
For a given $\ut\in \Ga_{Q,L_j'}$ write $\of_{\ut}$ for the restriction of $\of$ to the cube $Q_{\ut}(L_j')$. By localization, one then has
\eq\label{approx-12a}
\NN^2_{\la\De^0,Q}(\{\of\}) = \E_{\ut\in\Ga_{L_j',Q}}\, \NN^2_{\la\De^0,Q_{\ut}(L_j')}(\{\of_{\ut}\}) + O(\eps),
\ee
and
\eq\label{approx-12b}
\MM^2_{\la,Q}(\{\of\}) = \E_{\ut\in\Ga_{L_j',Q}}\, \MM^2_{\la,Q_{\ut}(L_j')}(\{\of_{\ut}\}) + O(\eps)
\ee
provided one also insists that $\la\ll\eps\, L'_j$.


For given $\ut\in\Ga_{L_j',Q}$, the functions $\of_{\ut}(x_1,x_2)$ are linear combinations of functions of the form $1_{A_{1,\ut}^{r_1}}(x_1)1_{A_{2,\ut}^{r_2}}(x_2)$, where $\{A_{1,\ut}^{r_1}\}_{1\leq r_1\leq R'}$ and $\{A_{2,\ut}^{r_2}\}_{1\leq r_2\leq R'}$ are the collections of the atoms of the $\si$-algebras $\BB_{1,\ut}$ and $\BB_{2,\ut}$ defined on the cubes $Q_{\ut_1}(L_j')$ and $Q_{\ut_2}(L_j')$. 
Thus for each $\ut\in\Ga_{L_j',Q}$ one has
\[\of_{\ut}=\sum_{r_{1}=1}^{R'}\sum_{r_{2}=1}^{R'} \al_{\ur,\ut} 1_{A_{1,\ut}^{r_{1}}}\times 1_{A_{2,\ut}^{r_{2}}}\]
where $\ur=(r_1,r_2)$.
Plugging these linear expansions into the multi-linear expressions in above one obtains
\[\NN^2_{\la\De^0,Q_{\ut}(L_j')}(\{\of_{\ut}\})=\sum_{\ur=\{\ur_{kl}\}_{kl}} \al_{\ur,\ut}\ \NN^2_{\la\De^0,Q_{\ut}(L_j')} (\{1_{A_{1,\ut}^{r_{1,kl}}}\times1_{A_{2,\ut}^{r_{2,kl}}}\})\]
using the notations $\ur_{kl}=(r_{1,kl},r_{2,kl})$,  $\al_{\ur,\ut}=\prod_{kl} \al_{\ur_{kl},\ut}$. Notice that the product
\[\prod_{k=1}^{n_1}\prod_{l=1}^{n_2} 1_{A_{1,\ut}^{r_{1,kl}}}(x_{1k})1_{A_{2,\ut}^{r_{2,kl}}}(x_{2l})
\]
is nonzero only if $A_{1,\ut}^{r_{1,kl}}=A_{1,\ut}^{r_{1,k}}$, that is if $r_{1,kl}=r_{1,k}$ for all $1\leq l\leq n_2$, as the atoms $A_{1,\ut}^r$ are all disjoint. Similarly, one has that $r_{2,kl}=r_{2,l}$ for all $1\leq k\leq n_1$. Thus, in fact
\eq\label{3.22a}
\NN^2_{\la\De^0,Q_{\ut}(L_j')}(\{\of_{\ut}\})=\sum_{\ur=\{\ur_{kl}\}_{kl}} \al_{\ur,\ut}\ \NN^2_{\la\De^0,Q_{\ut}(L_j')} (\{1_{A_{1,\ut}^{r_{1,k}}}\times1_{A_{2,\ut}^{r_{2,l}}}\})\ee
and similarly
\eq\label{3.23a}
\MM^2_{\la,Q_{\ut}(L_j')}(\{\of_{\ut}\})=\sum_{\ur=\{\ur_{kl}\}_{kl}} \al_{\ur,\ut}\ \MM^2_{\la,Q_{\ut}(L_j')} (\{1_{A_{1,\ut}^{r_{1,k}}}\times1_{A_{2,\ut}^{r_{2,l}}}\}).\ee
Note, that indices $\ur$ are running through the index set $[1,R']^{n_1}\times[1,R']^{n_2}$ of size at most $R$ if $C_1\gg1$.

\smallskip

The key observation is that
\eq\label{approx-13a}
\NN^2_{\la\De^0,Q_{\ut}(L_j')} (1_{A_{1,\ut}^{r_{1,k}}}\times1_{A_{2,\ut}^{r_{2,l}}})
= \NN^1_{\la\De^0_1,Q_{t_1}(L_j')}(1_{A_{1,\ut}^{r_{1,1}}},\dots,1_{A_{1,\ut}^{r_{1,n_1}}})\
\NN^1_{\la\De^0_2,Q_{t_2}(L_j')}(1_{A_{2,\ut}^{r_{2,1}}},\dots,1_{A_{2,\ut}^{r_{2,n_2}}})
\ee
and
\eq\label{approx-13b}
\MM^2_{\la,Q_{\ut}(L_j')} (1_{A_{1,\ut}^{r_{1,k}}}\times1_{A_{2,\ut}^{r_{2,l}}})= \MM^1_{\la,Q_{t_1}(L_j')}(1_{A_{1,\ut}^{r_{1,1}}},\dots,1_{A_{1,\ut}^{r_{1,n_1}}})\
\MM^1_{\la,Q_{t_2}(L_j')}(1_{A_{2,\ut}^{r_{2,1}}},\dots,1_{A_{2,\ut}^{r_{2,n_2}}}).
\ee

\smallskip

Let $\ur=\{(r_{1,k},r_{2,l})\}_{kl}$ and $g^{1,\ur}_{k,\ut}:= 1_{A_{1,\ut}^{r_{1,k}}}$, $g^{2,\ur}_{l,\ut}:= 1_{A_{2,\ut}^{r_{2,l}}}$. Writing $j':=\ind(L_j')$ and $J':=\ind(L_{j+1}')$, one may apply Proposition \ref{approx-par-1} for the families of functions $g^{1,\ur}_{k,\ut},\,g^{2,\ur}_{l,\ut}$, where $1\leq k\leq n_1,\,1\leq l\leq n_2$ and $\,\ur=(r_{1,k},r_{2,l})_{kl}\in [1,R']^{n_1}\times[1,R']^{n_2}$, with respect to the $\VE_1$-admissible sequence of scales \[L_{j'+1}\geq L_{j'+2}\geq\cdots\geq L_{J'-1}.\] 
This is possible as $J'-j'=J_1(\VE_1,R)$. Then there is a scale $L_j$ with $j'\leq j <J'$ so that 
\eq\label{approx-14a}
\NN^1_{\la\De^0_1,Q_{t_1}(L_j')}(g_{1,\ut}^{1,\ur},\ldots,
g_{n_1\ut}^{1,\ur}) =  \MM^1_{\la,Q_{t_1}(L_j')}(g_{1,\ut}^{1,\ur},\ldots,g_{n_1,\ut}^{1,\ur}) + O(\eps_1)
\ee
and
\eq\label{approx-14b}
\NN^1_{\la\De^0_2,Q_{t_2}(L_j')}(g_{1,\ut}^{2,\ur},\ldots,g_{n_2,\ut}^{2,\ur}) = \MM^1_{\la,Q_{t_2}(L_j')}(g_{1,\ut}^{2,\ur},\ldots,g_{n_2,\ut}^{2,\ur}) + O(\eps_1),
\ee
for all $\lm\in[L_{j+1},L_j]$ 
uniformly in $\ur=\{(r_{1,k},r_{2,l})\}_{kl}$ and $\ut\notin T_{\eps_1}\subs \Ga_{L_j',Q}$, for a set of size $|T_{\eps_1}|\leq\eps_1 |\Ga_{L_j',Q}|$. Then, by \eqref{approx-13a}-\eqref{approx-13b} and \eqref{3.22a}-\eqref{3.23a}, we have
\eq\label{approx-15}
\NN^2_{\la\De^0,Q_{\ut}(L_j')}(\{\of_{\ut}\}) = \MM^2_{\la,Q_{\ut}(L'_j)}\,(\{\of_{\ut}\}) + O(\eps)\ee
for $\ut\notin T_{\eps_1}$, as $|\al_{\ur,\ut}|\leq 1$ and $R\eps_1 \leq \eps$. Finally, since $|T_{\eps_1}|\leq \eps_1 |\Ga_{L_j',Q}|$, by averaging in $\ut\in\Ga_{L_j',Q}$, one has
\[ \NN^2_{\la\De^0,Q}(\{\of\}) = \MM^2_{\la,Q}\,(\{\of\}) + O(\eps)\]
using \eqref{approx-12a}-\eqref{approx-12b} and the Proposition follows by
\eqref{approx-11a} with an index $1\leq j< J_2=O(\VE^{-12}\VE_1^{-5})$. \end{proof}



\subsection{Proof of Lemmas \ref{vN-2} and \ref{KvN-2}}\label{vN-2 KvN-2}



\begin{proof}[Proof of Lemma \ref{vN-2}]

 First we note that if $\chi_L:= L^{-n}1_{[-L/2,L/2]^n}$ and $\psi_L:=\chi_L\ast\chi_L$, then \[\psi_L(x_2-x_1)=\int_t \chi_L(x_1-t)\chi_L(x_2-t)\,dt\] and hence for any function $f:Q\to [-1,1]$, with $Q\subs\R^n$ being a cube of side length $l(Q)$, one has
\[\|f\|_{U^1_L(Q)}^2 
= \fint_{x_1\in Q} \int_{x_2} f(x_1)f(x_2)\psi_L(x_2-x_1)\,dx_1dx_2 + O(L/l(Q)).
\]

Write $\ux':=(x_{21},\ldots,x_{2n_2})$ and let $g_{k,\ux'}(x):=\prod_{l=1}^{n_2} f_{kl}(x,x_{2l})$. Then one may write
\[
\NN^2_{\la\De^0,Q}(\{f_{kl}\}) = \fint_{x_{21}\in Q_2}\int_{x_{22},\ldots,x_{2n_2}}
\NN^1_{\la\De^0_1,Q_1}(g_{1,\ux'},\dots,g_{{n_1},\ux'})
\ d\si_{\la\De^0_2}(x_{22}-x_{21},\dots,x_{2 n_2}-x_{21})\,dx_{21}.
\]
Using estimate \eqref{vN-11a}, the above observation, and Cauchy-Schwarz one has 
\[|\NN^2_{\la\De^0,Q}(\{f_{kl}\})|^2 \leq
 \fint_{x_{11}\in Q_1}\int_{x_{12}}  \psi_L(x_{12}-x_{11}) \
\NN^1_{\la\De^0_2,Q_2} (h_{1,x_{11},x_{12}},\ldots,h_{n_2,x_{11},x_{12}}) \,dx_{11}dx_{12}+O(\eps^4)\]
provided $0<\la\ll l(Q)$ and $0<L\ll \eps^{24} \la$
where $h_{l,x_{11},x_{12}}(x)=f_{1l}(x_{11},x)f_{1l}(x_{12},x)$ for $1\leq l\leq n_2$.
Applying the same procedure again ultimately gives
\[|\NN^2_{\la\De^0,Q}(\{f_{kl}\})|^4 \leq
 \|f_{11}\|_{\Box_L(Q_1\times Q_2)}^4 + O(\eps^4).\]

The same estimate can of course be given for any function $f_{kl}$ in place of $f_{11}$. This establishes (\ref{vN-11}).
Estimate (\ref{vN-12}) is established similarly. \end{proof}


\begin{proof}[Proof of Lemma \ref{KvN-2}]

For each $\ut=(t_1,t_2)\in\Ga_{L_1,Q}$ we will let $\BB_{1,\ut}(L_1):=\{\emptyset,Q_{t_1}(L_1)\}$ and $\BB_{2,\ut}(L_1):=\{\emptyset,Q_{t_2}(L_1)\}$, in other words the trivial $\si$-algebras on $Q_{t_1}(L_1)$ and $Q_{t_2}(L_1)$ respectively. 
 If \eqref{KvN-21} holds with $\BB(L_1)=\mathcal{G}_{L_1,Q}$, noting that $\BB_{\ut}(L_1):=\BB_{1,\ut}(L_1)\vee \BB_{2,\ut}(L_1)$ in this case, then we are done.


We now assume that we have developed, for each $\ut=(t_1,t_2)\in\Ga_{L_j,Q}$,  $\si$-algebras $\BB_{1,\ut}(L_j)$  on $Q_{t_1}(L_j)$ and $\BB_{2,\ut}(L_j)$ on $Q_{t_2}(L_j)$ with  $\comp(\BB_{i,\ut}(L_j))\leq j$ for $i=1,2$. 
Let $\BB(L_j)$ be the $\si$-algebra such that $\BB_{\ut}(L_j)=\BB_{1,\ut}(L_j)\vee \BB_{2,\ut}(L_j)$ for all $\ut\in\Ga_{L_j,Q}$ and assume that \eqref{KvN-21} does not hold, namely that
\[\|g\|_{\Box_{L_{j+1}}(Q)} \geq \eps\] where $g:=f-\E(f|\BB(L_j))$.
 By the definition of the local box norm this means that
\[\fint_{\ut\in Q} \|g\|_{\Box(\ut+Q(L_{j+1}))}^4\,dt \geq \eps^4\]
and hence, as $L_{j+2}\ll\,\eps^4 L_{j+1}$, it is easy to see that
\[\E_{\us\in \Ga_{L_{j+2},Q}}\ \|g\|_{\Box(\us+Q(L_{j+2}))}^4 \geq \eps^4/2.\]

This implies that there is a set $S\subs \Ga_{L_{j+2},Q}$ of size $|S|\geq (\eps^4/4) |\Ga_{L_{j+2},Q}|$ such that for all $\us=(s_1,s_2)\in S$, one has that
$\,\|g\|_{\Box(Q_{\us}(L_{j+2}))}^4 \geq \eps^4/4$. It therefore follows, as is well-known see for example \cite{LM17} or \cite{TaoMontreal}, that there exist sets $B_{1,\us}\subs Q_{s_1}(L_{j+2})$ and $B_{2,\us}\subs Q_{s_2}(L_{j+2})$ such that
\eq\label{KvN-23}
\fint_{x_1\in Q_{s_1}(L_{j+2})}\fint_{x_2\in Q_{s_2}(L_{j+2})} g(x_1,x_2)\,1_{B_{1,\us}}(x_1)1_{B_{2,\us}}(x_2)\,dx_1\, dx_2 \geq \eps^4/16.
\ee

For a given $\us\in\Ga_{L_{j+2},Q}$ there is a unique $\ut=\ut(\us)$ such that $Q_{\us}(L_{j+2})\subs Q_{\ut}(L_j)$. Let $\BB'_{1,\us}(L_{j+2}):= \BB_{1,\ut}(L_j)|_{Q_{s_1}(L_{j+2})}$ and $\BB'_{2,\us}(L_{j+2}):= \BB_{2,\ut}(L_j)|_{Q_{s_2}(L_{j+2})}$  noting that $\comp(\BB'_{i,\us}(L_{j+2}))\leq j$ for $i=1,2$, as the complexity of a $\si$-algebra does not increase when restricted to a set. If, for $i=1,2$, we  let $\BB_{i,\us}(L_{j+2})$ denote the $\si$-algebra generated by $\BB'_{i,\us}(L_{j+2})$ and the set
$B_{i,\us}$ if $\us\in S$ and let $\BB_{i,\us}(L_{j+2}):=\BB'_{i,\us}(L_{j+2})$ otherwise, 
then clearly $\comp(\BB_{i,\us}(L_{j+2}))$ is at most $j+1$.
We now define $\BB(L_{j+2})$ to the the sigma algebra of scale $L_{j+2}$ with the property  that
$\BB_{\us}(L_{j+2})=\BB_{1,\us}(L_{j+2})\vee \BB_{2,\us}(L_{j+2})$
for all $\us\in\Ga_{L_{j+2},Q}$. 

Using the inner product notation $\langle f,g\rangle_Q=\fint_Q f(x)g(x)\,dx$ we can rewrite \eqref{KvN-23} as
\[\langle f-\E(f|\BB(L_j))\,,\,1_{B_{1,\us}}\times
1_{B_{2,\us}}\,\rangle_{Q_{\us}(L_{j+2})} \geq \eps^4/16\]
for all $\us\in S$. 
Since the function $1_{B_{1,\us}}\times1_{B_{2,\us}}$ is measurable with respect to  $\BB(L_{j+2})$ one clearly has
\[\langle f-\E(f|\BB(L_{j+2}))\,,\,1_{B_{1,\us}}\times
1_{B_{2,\us}}\rangle_{Q_{\us}(L_{j+2})} = 0\]
and hence 
\[\langle \E(f|\BB(L_{j+2}))-\E(f|\BB(L_j))
\,,\,1_{B_{1,\us}}\times1_{B_{2,\us}}\rangle_{Q_{\ut}(L_{j+2})} \geq \eps^4/16.\]

It then follows from 
 Cauchy-Schwarz and orthogonality that
\[\|\E(f|\BB(L_{j+2}))\|_{L^2(Q_{\us}(L_{j+2}))}^2 -
\|\E(f|\BB_1(L_j))\|_{L^2(Q_{\us}(L_{j+2}))}^2\geq\eps^8/256.\]

Since $|S|\geq (\eps^4/4) |\Ga_{L_{j+2},Q}|$ averaging over all $\us\in \Ga_{L_{j+2},Q}$ gives
\[\|\E(f|\BB(L_{j+2}))\|_{L^2(Q)}^2 \geq \|\E(f|\BB(L_j))\|_{L^2(Q)}^2 + \eps^{12}/2^{10}.\]
Trivially both sides are at most $1$ thus the process must stop at a step $j=O(\eps^{-12})$ where \eqref{KvN-21} holds for a $\sigma$-algebra of ``local complexity" at most $j$. This proves the Lemma. \end{proof}


\section{Proof of Theorem \ref{Prod}: The general case.}\label{generalcase}


After these preparations we will now consider the general case of Theorem \ref{Prod}. 
Let $Q=Q_1\times\cdots \times Q_d\subs\R^n$ with $Q_i\subs \R^{n_i}$ cubes of equal side length $l(Q)$ and $\De^0=\De_1^0\times\cdots\times\De_d^0$ with each $\De_i\subs\R^{n_i}$ a non-degenerate simplex of $n_i$ points for  $1\leq i\leq d$.

We will use a generalized version of the hypergraph terminology introduced in Section \ref{FFSECTION}. In particular, for a vertex set $I=\{1,2,\ldots,d\}$ and set $K=\{il;\ 1\leq i\leq d,\,1\leq l\leq n_i\}$ we will let $\pi:K\to I$ denote the projection defined by $\pi (il):=i$. As before we will let $\HH_{d,k}:=\{e\subs I;\ |e|=k\}$ denote the complete $k$-regular hypergraph with vertex set $I$, and for the multi-index $\un=(n_1,\ldots,n_d)$ define the hypergraph bundle 
\[\HH_{d,k}^{\un}:=\{e\subs K;\ |e|=|\pi(e)|=k\}\]
noting that $|\pi^{-1}(i)|=n_i$ for all $i\in I$.


In order to parameterize the vertices of direct products of simplices, i.e. sets of the form $\De=\De_1\times\cdots\times\De_d\,$ with $\De_i\subs Q_i$, we consider points
$\ux=(\ux_1,\dots,\ux_d)$
with $\ux_i=(x_{i1},\dots,x_{in_i})\in Q_i^{n_i}$ for each $i\in I$.  Now for any $1\leq k\leq d$ and any edge $e'\in\HH_{d,k}$ we will write $Q_{e'}:=\prod_{i\in e'} Q_i$, and for every $\ux\in Q_1^{n_1}\times\cdots\times Q_d^{n_d}$ and $e\in\HH_{d,k}^{\un}$ we define $\ux_e:=\pi_e(\ux)$, where $\pi_e:Q_1^{n_1}\times\cdots\times Q_d^{n_d}\to Q_{\pi(e)}$ is the natural projection map. 
Writing $\De_i=\{x_{i1},\ldots,x_{in_i}\}$ we have that
$\De_1\times\cdots\times \De_d=\{\ux_e:\ e\in\HH_{d,d}^{\un}\}$ since every edge $\ux_e$ is of the form $(x_{1l_1},\ldots,x_{dl_d})$. We can therefore identify points $\ux$ with configurations  of the form $\De_1\times\cdots\times\De_d$.

For any $0<\la\ll l(Q)$ the measures $d\si_{\la\De_i^0}$, introduced in Section \ref{S1}, are supported on points $(y_2,\ldots,y_{n_i})$ for which the simplex $\De_i=\{0,y_2,\ldots,y_{n_i}\}$ is isometric to $\la\De_i^0$. For simplicity of notation we will write
\[\int_{\ux_i} f(\ux_i)\,d\si^\la_i(\ux_i):= \fint_{x_{i1}\in Q_i}\int_{x_{i2},\ldots,x_{in_i}} f(\ux_i) \,d\si_{\la\De_i^0}(x_{i2}-x_{i1},\ldots,x_{in_i}-x_{i1})\,dx_{i1}\]

Note that the support of the measure $d\si^\la_i$ is the set of points $\ux_i$ so that the simplex $\De_i:=\{x_{i1},\ldots,x_{in_i}\}$ is isometric to $\la\De_i^0$ and $x_{i1}\in Q_i$, moreover the measure is normalized. Thus if $S\subs Q$ is a set then the density of configurations $\De$ in $S$ of the form $\De=\De_1\times\ldots\times\De_d$ with each $\De_i\subseteq Q_i$ an isometric copy of $\la\De_i^0$ is given by the expression
\eq\label{4.1}
\NN^d_{\la\De^0,Q}(1_S\,;\,e\in\HH_{d,d}^{\un}):= \int_{\ux_1}\cdots\int_{\ux_d}\prod_{e\in\HH_{d,d}^{\un}} 1_S(\ux_e)\ d\si^{\la}_1(\ux_1)\ldots d\si^{\la}_d(\ux_d).
\ee

The proof of Theorem \ref{Prod} reduces to establishing the following stronger quantitative result.

\begin{prop}\label{PropnSimpd}

For any $0<\VE\ll 1$  there exists an integer $J_d=J_d(\VE)$ with the following property: 

Given any lacunary sequence $l(Q)\geq \lm_1\geq\cdots\geq\lm_{J_d}$  and $S\subseteq Q$, there is some $1\leq j< J_d$ such that 
\eq
\NN^d_{\la\De^0,Q}(1_S\,;\,e\in\HH_{d,d}^{\un})>\left(\frac{|S|}{|Q|}\right)^{n_1\cdots\, n_d}-\VE
\ee 
for all $\lambda\in [\lm_{j+1},\lm_j]$.
\end{prop}

\emph{Quantitative Remark.} A careful analysis of our proof reveals that there is a choice of $J_d(\VE)$ which is less than $W_d(\log (C_\Delta\VE^{-3}))$, where $W_k(m)$ is again the tower-exponential function defined by $W_1(m)=\exp(m)$ and $W_{k+1}(m)=\exp(W_k(m))$ for $k\geq1$.

\medskip

For any  $0<\lm\ll l(Q)$ and set $S\subs Q$ we define the expression:
\eq
\MM^d_{\lm,Q}(1_S\,;\,e\in\HH_{d,d}^{\un}):=\fint_{\ut\in Q} \MM^d_{\ut+Q(\lm)}(1_S\,;\,e\in\HH_{d,d}^{\un})\,d\ut
\ee 
where $Q(\lm)=[-\frac{\lm}{2},\frac{\lm}{2}]^{n}$ and
\eq
\MM^d_{\widetilde{Q}}(1_S\,;\,e\in\HH_{d,d}^{\un}):=\fint_{\ux_1\in \widetilde{Q}_1^{n_1}}\!\!\!\!\!\!\!\!\!\cdots \ \fint_{\ux_d\in \widetilde{Q}_d^{n_d}}\prod_{e\in\HH_{d,d}^{\un}} 1_S(\ux_e)\,d\ux_1\ldots d\ux_d
\ee 
for any cube $\widetilde{Q}\subseteq Q$ of the form $\widetilde{Q}=\widetilde{Q}_1\times\cdots\times \widetilde{Q}_d$ with $\widetilde{Q}_i\subseteq Q_i$ for $1\leq i\leq d$.
 Note that if $S\subs Q$ is
a set of
measure $|S|\geq \de |Q|$ for some $\de>0$, then careful applications of H\"older's inequality give
\[
\MM^d_{\lm,Q}(1_S\,;\,e\in\HH_{d,d}^{\un}) \geq \fint_{\ut\in Q} \left(\fint_{(x_1,\dots,x_d)\in \ut+Q(\lm)} 1_S(x_1,\dots,x_d)\,dx_1\dots dx_d\right)^{n_1\cdots\, n_d} d\ut \,
\geq\de^{n_1\cdots\, n_d}-O(\eps)\]
for all scales $0<\lm\ll\eps\, l(Q)$.

In light of the discussion above, and that preceding Proposition \ref{approx-1}, we see that Proposition \ref{PropnSimpd}, and hence Theorem \ref{Prod} in general, will follows as a consequence of the following

\begin{prop}\label{approx-d}

Let $0<\eps\ll1$. There exists an integer $J_d=J_d(\VE)$ such that for any $\eps$-admissible sequence of scales $l(Q)\geq L_1\geq\cdots \geq L_{J_d}$ and $S\subseteq Q$ there is some $1\leq j< J_d$ such that
\eq\label{approx-dd}
\NN^d_{\la\De^0,Q}(1_S\,;\,e\in\HH_{d,d}^{\un})= \MM^d_{\la,Q}(1_S\,;\,e\in\HH_{d,d}^{\un}) + O(\eps)
\ee
for all $\lm\in[L_{j+1},L_j]$.
\end{prop}

The validity of Proposition \ref{approx-d} will follow immediately from the $d=k$ case of Proposition \ref{Lem4.4} below. 

\subsection{Reduction of Proposition \ref{approx-d} to a more general ``local" counting lemma}\


For any given $1\leq k\leq d$ and collection of functions  $f_e: Q_{\pi(e)}\to[-1,1]$ with $e\in\HH_{d,k}^{\un}$ we define the following multi-linear expressions
\eq\label{4.2}
\NN^d_{\la\De^0,Q}(f_e;e\in\HH_{d,k}^{\un}):= \int_{\ux_1}\cdots\int_{\ux_d} \prod_{e\in\HH_{d,k}^{\un}} f_e(\ux_e)\ d\si^{\la}_1(\ux_1)\ldots. d\si^{\la}_d(\ux_d)
\ee
and
\eq
\MM^d_{\lm,Q}(f_e\,;\,e\in\HH_{d,k}^{\un}):=\fint_{\ut\in Q} \MM^d_{\ut+Q(\lm)}(f_e\,;\,e\in\HH_{d,k}^{\un})\,d\ut
\ee 
where $Q(\lm)=[-\frac{\lm}{2},\frac{\lm}{2}]^{n}$ and
\eq
\MM^d_{\widetilde{Q}}(f_e\,;\,e\in\HH_{d,k}^{\un}):=\fint_{\ux_1\in \widetilde{Q}_1^{n_1}}\!\!\!\!\!\!\!\!\!\cdots \ \fint_{\ux_d\in \widetilde{Q}_d^{n_d}}\prod_{e\in\HH_{d,k}^{\un}} f_e(\ux_e)\,d\ux_1\ldots d\ux_d
\ee 
for any cube $\widetilde{Q}\subseteq Q$ of the form $\widetilde{Q}=\widetilde{Q}_1\times\cdots\times \widetilde{Q}_d$ with $\widetilde{Q}_i\subseteq Q_i$ for $1\leq i\leq d$.


Our strategy to proving Proposition \ref{approx-d} is the same as illustrated in the finite field settings, that is we would like to compare averages $\NN_{\la\De^0,Q}(f_e;e\in\HH_{d,k}^{\un})$ to those of $\MM^d_{\lm,Q}(f_e\,;\,e\in\HH_{d,k}^{\un})$, at certain scales $\lm\in[L_{j+1}, L_j]$, inductively for $1\leq k\leq d$.
However in the Euclidean case, an extra complication emerges due to the fact the (hypergraph) regularity lemma, the analogue of Lemma \ref{KvN-ff}, does not produce $\si$-algebras $\BB_{\f}$, for $\f\in\HH_{d,k-1}^{\un}$, on the cubes $Q_{\f}$.  In a similar manner to the case for $d=2$ discussed in the previous section, we will only obtain $\si$-algebras  ``local" on cubes $Q_{\ut_{\f}}(L_0)$ at some scale $L_0>0$. This will have the effect that the functions $f_e$ will be replaced by a family of functions $f_{e,\ut}$, where $\ut$ runs through a grid $\Ga_{L_0,Q}$.

To be more precise, let $L>0$ be a scale dividing the side-length $l(Q)$. For $\ut\in\Ga_{L,Q}$ and $e'\in \HH_{d,k}$ we will use $\ut_{e'}$ to denote the projection of $\ut$ onto $Q_{e'}$ and
$Q_{\ut_{e'}}(L):=\ut_{e'}+ Q_{e'}(L)$ 
to denote the projection of the cube $Q_{\ut}(L)$ centered at $\ut$ onto $Q_{e'}$.
It is then easy to see that for any $\VE>0$ we have
\eq\label{4.7}
\NN^d_{\la\De^0,Q}(f_{e};e\in\HH_{d,k}^{\un})= \E_{\ut\in\Ga_{L,Q}}\, \NN^d_{\la\De^0,Q_{\ut}(L)}(f_{e,\ut}\,;e\in\HH_{d,k}^{\un})+ O(\eps)
\ee
and
\eq\label{4.8}
\MM^d_{\la,Q}(f_{e};e\in\HH_{d,k}^{\un})=\E_{\ut\in\Ga_{L,Q}}\,
\MM^d_{\la,Q_{\ut}(L)}(f_{e,\ut}\,;e\in\HH_{d,k}^{\un})+ O(\eps)
\ee
provided $0<\la\ll \eps L$ where $f_{e,\ut}$ denotes the restriction of a function  $f_e$ to the cube $Q_{\ut}(L)$.

At this point the proof of Proposition \ref{approx-d} reduces to showing that the expressions in \eqref{4.7} and \eqref{4.8} only differ by $O(\eps)$ at some scales $\lm\in[L_{j+1},L_j]$, given an $\VE$-admissible sequence $L_0\geq L_1\geq \cdots\geq L_J$, for any collection of bounded functions $f_{e,\ut}$, $e\in \HH_{d,k}^{\un}$, $\ut\in\Ga_{L_0,Q}$. Indeed, our crucial result will be the following



\begin{prop}[Local Counting Lemma]\label{Lem4.4} Let 
$0<\eps\ll1$ and $M\geq 1$. There exists an integer $J_k=J_k(\eps,M)$ such that for any $\eps$-admissible sequence of scales $L_0\geq L_1\geq\cdots \geq L_{J_k}$ with the property that  $L_0$ divides $l(Q)$, and collection of functions  
\[\text{$f_{e,\ut}^m:Q_{t_{\pi(e)}}(L_0):\to [-1,1]$ \ with \ $e\in \HH_{d,k}^{\un}$, $1\leq m\leq M$ and  $\ut\in\Ga_{L_0,Q}$}\] 
there exists $1\leq j< J_k$ and a set $T_{\eps}\subs \Ga_{L_0,Q}$ of size $|T_{\eps}|\leq \eps |\Ga_{L_0,Q}|$ such that
\eq\label{4.11}
\NN^d_{\la\De^0,Q_{\ut}(L_0)}(f_{e,\ut}^m;\,e\in \HH_{d,k}^{\un}) =\MM_{\la,Q_{\ut}(L_0)}(f_{e,\ut}^m;\,e\in \HH_{d,k}^{\un}) + O(\eps)
\ee
for all $\lm\in[L_{j+1}, L_j]$ and $\ut\notin T_\eps$ 
uniformly in $e\in \HH_{d,k}^{\un}$ and $1\leq m\leq M$.
\end{prop}

\smallskip

\subsection{Proof of Proposition \ref{Lem4.4}}\

We will prove Proposition \ref{Lem4.4} by induction on $1\leq k\leq d$. For $k=1$ this is basically Proposition \ref{approx-par-1}. 

Indeed, in this case for a given $\ut=(t_1,\ldots,t_d)\in \Ga_{L_0,Q}$ and edge $e\in\HH_{d,1}^{\un}=\{il\,:\,1\leq i\leq d,\, 1\leq l\leq n_i\}$ we have that $f_{e,\ut}^m(\ux_e)=f_{il,\ut}^m(x_{il})$ with $x_{il}\in Q_{t_i}(L_0)$ and hence both
\[\NN^d_{\la\De^0,Q_{\ut}(L_0)}(f_{e,\ut}^m;\,e\in\HH_{d,1}^{\un}) = \prod_{i=1}^d\ \NN^1_{\la\De_i^0,Q_{t_i}(L_0)}
(f_{i1,\ut}^m,\ldots,f_{in_i,\ut}^m)\]
\[\MM^d_{\la,Q_{\ut}(L_0)}(f_{e,\ut}^m;\,e\in\HH_{d,1}^{\un}) =
\prod_{i=1}^d\ \MM^1_{\la,Q_{t_i}(L_0)} (f_{i1,\ut}^m,\ldots,f_{in_i,\ut}^m).\]

By Proposition \ref{approx-par-1} there exists an $1\leq j< J_1=O(M\eps^{-4})$ and an exceptional set $T_\eps\subs \Ga_{L_0,Q}$ of size $|T_\eps|\leq \eps |\Ga_{L_0,Q}|$, such that uniformly for $\ut\notin T_\eps$ and for $1\leq i\leq d$, one has
\[ \NN^1_{\la\De_i^0,Q_{t_i}(L_0)} (f_{i1,\ut}^m,\ldots,f_{in_i,\ut}^m) =
\MM^1_{\la,Q_{t_i}(L_0)} (f_{i1,\ut}^m,\ldots,f_{in_i,\ut}^m) + O(\eps)\]
hence
\[\NN^d_{\la\De^0,Q_{\ut}(L_0)}(f_{e,\ut}^m;\,e\in\HH_{d,1}^{\un}) =
\MM^d_{\la,Q_{\ut}(L_0)}(f_{e,\ut}^m;\,e\in\HH_{d,1}^{\un}) + O(\eps)\]
as the all factors are trivially bounded by 1 in magnitude. This implies \eqref{4.11} for $k=1$.

\smallskip

For the induction step we again need two main ingredients. 
The first establishes that the our multi-linear forms $\NN^d_{\la\De^0,Q}(f_e;\,e\in \HH_{d,k}^{\un})$ are controlled by an appropriate box-type norm attached to a scale $L$. 

Let $Q=Q_1\times\cdots\times Q_d$ and $1\leq k\leq d$. For any  scale $0<L\ll l(Q)$  and function $f:Q_{e'}\to[-1,1]$ with $e'\in\HH_{d,k}$  we define its local box norm at scale $L$ by
\eq\label{box-norm-Lk}
\|f\|_{\Box_L(Q_{e'})}^{2^k} := \fint_{\us\in Q_{e'}}  \|f\|_{\Box(\us+Q(L))}^{2^k}  \,d\us
\ee
where
\eq\label{box-norm-k}
\|f\|_{\Box(\widetilde{Q})}^{2^k}:= \fint_{x_{11},x_{12}\in \widetilde{Q}_1}\!\!\!\!\!\!\cdots \ \fint_{x_{k1},x_{k2}\in \widetilde{Q}_k} \prod_{(\ell_1,\dots,\ell_k)\in\{1,2\}^k}f(x_{1\ell_1},\dots,x_{k\ell_k})  \,dx_{11}\,dx_{12}\ldots \,dx_{k1}\,dx_{k2}
\ee
for any cube $\widetilde{Q}$ of the form $\widetilde{Q}=\widetilde{Q}_1\times\cdots\times\widetilde{Q}_k$.



\begin{lem}[Generalized von-Neumann inequality]\label{Lem4.5}
Let $\VE>0$, $0<\la\ll l(Q)$ and $0<L\ll (\eps^{2^k})^6\la$.

For any $1\leq k\leq d$ and collection of functions  $f_e: Q_{\pi(e)}\to[-1,1]$ with $e\in\HH_{d,k}^{\un}$ we have both
\eq\label{4.12}
|\NN^d_{\la\De^0,Q}(f_e;\,e\in \HH_{d,k}^{\un})| \leq \min_{e\in \HH_{d,k}^{\un}}\|f_e\|_{\Box_L(Q_{\pi(e)})} + O(\eps)
\ee
\eq\label{MMM}
|\MM^d_{\la,Q}(f_e;\,e\in \HH_{d,k}^{\un}) | \leq \min_{e\in \HH_{d,k}^{\un}} \|f_{e}\|_{\Box_L(Q_{\pi(e)})}.
\ee
\end{lem}

The crucial ingredient is the following analogue of the weak hypergraph regularity lemma.



\begin{lem}[Parametric weak hypergraph regularity lemma for $\R^n$]\label{Lem4.6}

Let $0<\eps\ll1$, $M\geq 1$, and $1\leq k\leq d$. 

There exists  $\bar{J}_k=O(M\eps^{-2^{k+3}})$ such that for any $\eps^{2^k}$-admissible sequence $L_0\geq L_1\geq\cdots \geq L_{\bar{J}_k}$ with the property that  $L_0$ divides $l(Q)$ and collection of functions  
\[\text{$f^m_{e,\ut}: Q_{\ut_{\pi(e)}}(L_0)\to [-1,1]$ \ with \ 
$e\in\HH_{d,k}^{\un}$, $1\leq m\leq M$, and $\ut\in\Ga_{L_0,Q}$}\] 
there is some $1\leq j< \bar{J}_k$ and $\si$-algebras $\BB_{e',\ut}$ of scale $L_j$ on $Q_{\ut_{e'}}(L_0)$ for each $\ut \in\Ga_{L_0,Q}$ and $e'\in\HH_{d,k}$ such that
\eq\label{4.13}
\|f_{e,\ut}^m-\E(f_{e,\ut}^m|\BB_{\pi(e),\ut})\|_{\Box_{L_{j+1}}(Q_{\ut_{\pi(e)}}(L_0))} \leq \eps
\ee
uniformly for all $t\notin T_\VE$, $e\in\HH_{d,k}^{\un}$, and $1\leq m\leq M$, where $T_{\eps}\subs \Ga_{L_0,Q}$ with $|T_{\eps}|\leq \eps |\Ga_{L_0,Q}|$.

Moreover, the $\si$-algebras $\BB_{e',\ut}$ 
have the additional local structure that  the exist $\si$-algebras $\BB_{e',\f',\us}$  on $Q_{\us_{\f'}}(L_j)$ with $\comp(\BB_{e',\f',\us}) = O(j)$  for each  $\us\in\Ga_{L_j,Q}$, $e'\in\HH_{d,k}$, and $\f'\in\partial e'$ such that
 if $\us\in Q_{\ut}(L_0)$, then
\eq\label{4.14}
\BB_{e',\ut}\bigr\rvert_{Q_{\us_{e'}}(L_j)} = \bigvee_{\f'\in\partial e'}
\BB_{e',\f',\us}.\ee
\end{lem}

\medskip

Lemma \ref{Lem4.6} is the parametric and simultaneous version of the extension of Lemma 3.7 to the product of $d$ simplices. The difference is that in the general case one has to deal with a parametric family of functions $f_{e,\ut}^m$ as $\ut$ is running through a grid $\Ga_{L_0,Q}$. The essential new content of Lemma \ref{Lem4.6} is that one can develop  $\si$-algebras $\BB_{e',\ut}$ on the cubes $Q_{\ut}(L_0)$ with respect to the family of functions $f_{e,\ut}^m$ such that  the  local structure described above and \eqref{4.13} hold simultaneously for almost all $\ut\in \Ga_{L_0,Q}$.

\smallskip


\begin{proof}[Proof of Proposition \ref{Lem4.4}]

Assume the Proposition holds for $k-1$. 

Let $\VE>0$, $\eps_1:=\exp\,(-C_1\eps^{-2^{k+3}})$ for some large constant $C_1=C_1(n,k,d)\gg 1$, and $\{L_j\}_{j\geq 1}$ be an $\VE_1$-admissible sequence of scales.
Set  $F(\eps):=J_{k-1}(\eps_1, M)$ with $M=\VE\,\VE_1^{-1}$.

For $L\in\{L_j\}_{j\geq 1}$ we again write $\ind(L)=j$ if $L=L_j$.
We now choose a subsequence $\{L_j'\}\subs \{L_j\}$ so that $L'_0=L_0$ and 
$\ind(L'_{j+1})\geq \ind(L'_j) + F(\eps)+2.$
Lemma \ref{Lem4.6} then guarantees the existence of  $\si$-algebras $\BB_{e',\ut}$ of scale $L'_j$ on $Q_{\ut_{e'}}(L_0)$ for each $\ut \in\Ga_{L_0,Q}$ and $e'\in\HH_{d,k}$, with the  local structure described above, such that
\eq\label{4.13333}
\|f_{e,\ut}^m-\E(f_{e,\ut}^m|\BB_{\pi(e),\ut})\|_{\Box_{L'_{j+1}}(Q_{\ut_{\pi(e)}}(L_0))} \leq \eps
\ee
uniformly for all $t\notin T'_\VE$, $e\in\HH_{d,k}^{\un}$, and $1\leq m\leq M$, for some $1\leq j<\bar{J}_k(\eps,M)=O(M\eps^{-2^{k+3}})$, where $T'_{\eps}\subs \Ga_{L_0,Q}$ with $|T'_{\eps}|\leq \eps |\Ga_{L_0,Q}|$.
Let 
$\of_{e,\ut}^m := \E(f_{e,\ut}^m|\BB_{\pi(e),\ut})$
for $\ut \in\Ga_{L_0,Q}$ and $e\in\HH_{d,k}^{\un}$. If $t\notin T'_\VE$, then by \eqref{4.12}, \eqref{MMM}, and \eqref{4.13} we have both
\eq\label{4.16}
\NN^d_{\la\De^0,Q_{\ut}(L_0)} (f_{e,\ut}^m;e\in\HH_{d,k}^{\un}) =
\NN^d_{\la\De^0,Q_{\ut}(L_0)} (\of_{e,\ut}^m;e\in\HH_{d,k}^{\un})+O(\eps)
\ee
\eq\label{4.17}
\MM^d_{\la,Q_{\ut}(L_0)} (f_{e,\ut}^m;e\in\HH_{d,k}^{\un}) =
\MM^d_{\la,Q_{\ut}(L_0)} (\of_{e,\ut}^m;e\in\HH_{d,k}^{\un})+O(\eps).
\ee
provided  $ (\eps^{-2^k})^6 L'_{j+1}\ll \lm$.
For given $\us\in\Ga_{L'_j,Q_{\ut}(L_0)}$ one may write $\of_{e,\us}^m$ for the restriction of $\of_{e,\ut}^m$ on the cube $Q_{\us}(L'_j)\subs Q_{\ut}(L_0)$, as $\us$ uniquely determines $\ut$. By localization, provided $\la\ll\eps L'_j$,  we then have both
\eq\label{4.18}
\NN^d_{\la\De^0,Q_{\ut}(L_0)} (\of_{e,\ut}^m;e\in\HH_{d,k}^{\un}) =
\E_{\us\in\Ga_{L'_j,Q_{\ut}(L_0)}} \NN^d_{\la\De^0,Q_{\us}(L'_j)} (\of_{e,\us}^m;e\in\HH_{d,k}^{\un}) + O(\eps),
\ee
\eq\label{4.19}
\MM^d_{\la,Q_{\ut}(L_0)} (\of_{e,\ut}^m;e\in\HH_{d,k}^{\un}) =
\E_{\us\in\Ga_{L'_j,Q_{\ut}(L_0)}} \MM^d_{\la,Q_{\us}(L'_j)} (\of_{e,\us}^m;e\in\HH_{d,k}^{\un}) + O(\eps).
\ee

For a fixed cube $Q_{\us}(L'_j)$ we have that
\eq\label{4.20}
\of_{e,\us}^m=\sum_{r_e=1}^{R_{e,\us}} \alpha_{\us,r_e,m} \ 1_{A_{\pi(e),\us}^{r_e}}
\ee
where $\{A_{\pi(e),\us}^{r_e}\}_{1\leq r\leq R_{e,\us}}$ is the family of atoms of the $\si$-algebra $\BB_{\pi(e),\ut}$ restricted to the cube $Q_{\us}(L'_j)$. Note that $|\alpha_{\us,r_e}|\leq 1$ and $|R_{e,\us}|=O(\exp\,(C\eps^{-2^{k+3}}))$. By adding the empty set to the collection of atoms one may assume $|R_{e,\us}|=R:=\exp\,(C\eps^{-2^{k+3}})$ for all $e\in\HH_{d,k}^{\un}$ and $\us\in\Ga_{L'_j,Q}$. Then, by multi-linearity, using the notations $\ur=(r_e)_{e\in\HH_{d,k}^{\un}}$ and $\al_{\ur,\us}=\prod_e \al_{\us,r_e}$, one has both
\eq\label{4.21}
\NN^d_{\la\De^0,Q_{\us}(L'_j)}(\of_{\us,e}^m;\,e\in\HH_{d,k}^{\un})
=\sum_{\ur} \al_{\us,\ur,m}\  \NN^d_{\la\De^0,Q_{\us}(L'_j)}(1_{A_{\pi(e),\us}^{r_e}};\,e\in\HH_{d,k}^{\un})
\ee
\eq\label{4.22}
\MM^d_{\la,Q_{\us}(L'_j)} (\of_{\us,e}^m;\,e\in\HH_{d,k}^{\un})
=\sum_{\ur} \al_{\us,\ur,m}\  \MM^d_{\la,Q_{\us}(L'_j)}(1_{A_{\pi(e),\us}^{r_e}};\,e\in\HH_{d,k}^{\un}).
\ee

The key observation is that these expressions in the sum above are all at level $k-1$ instead of $k$. 
To see this let $e=(i_1l_1,\ldots,i_ml_m,\ldots,i_kl_k)$ so $e'=\pi(e)=(i_1,\ldots,i_m,\ldots,i_k)$. If $\f'=e'\backslash\{i_m\}$ then recall that the edge $p_{\f'}(e)=(i_1l_1,\ldots,i_kl_k)\in\HH_{d,k-1}^{\un}$ is obtained from $e$ by removing the $i_ml_m$-entry. Thus, for any atom $A_{e',\us}$ of $\BB_{\us,e'}(L'_j)$ we have by \eqref{4.14}, that
\eq\label{4.23}
1_{A_{e',\us}}(\ux_e) = \prod_{\f'\in\partial e'} 1_{A_{e',\f',\us,}}(\ux_{p_{\f'}(e)})
\ee
where $A_{e',\f',\us}$ is an atom of the $\si$-algebra $\BB_{e',\f',\us}$. Thus
\eq\label{4.24}
\prod_{e\in\HH_{d,k}^{\un}} 1_{A_{\pi(e),\us}^{r_e}} (\ux_e) =
\prod_{\f\in\HH_{d,k-1}^{\un}} \prod_{\substack{e\in\HH_{d,k}^{\un},\f'\in\partial \pi(e)\\p_{\f'}(e)=f}} 1_{A_{\pi(e),\f',\us}^{r_e}}(\ux_{\f})
= \prod_{\f\in\HH_{d,k-1}^{\un}} g_{\f,\us}^{\ur}\,(\ux_{\f}).
\ee
It follows that
\eq\label{4.25}
\NN^d_{\la\De^0,Q_{\us}(L'_j)}(1_{A_{\pi(e),\us}^{r_e}};\,e\in\HH_{d,k}^{\un})=
\NN^d_{\la\De^0,Q_{\us}(L'_j)}\,(g_{\f,\us}^{\ur};\,\f\in\HH_{d,k-1}^{\un})
\ee
and hence that
\eq\label{4.26}
\NN^d_{\la\De^0,Q_{\us}(L'_j)} (\of_{e,\us}^m;\,e\in\HH_{d,k}^{\un})
=\sum_{\ur} \al_{\us,\ur,m}\  \NN^d_{\la\De^0,Q_{\us}(L'_j)}\,(g_{\f,\us}^{\ur};\,\f\in\HH_{d,k-1}^{\un})
\ee
and similarly
\eq\label{4.27}
\MM^d_{\la,Q_{\us}(L'_j)} (\of_{e,\us}^m;\,e\in\HH_{d,k}^{\un})
=\sum_{\ur} \al_{\ur,\us,m}\  \MM^d_{\la,Q_{\us}(L'_j)}\,(g_{\f,\us}^{\ur};\,\f\in\HH_{d,k-1}^{\un}).
\ee

Note that number of index vectors $\ur=(r_e)_{e\in\HH_{d,k}^{\un}}$ is $R^D$ with $D:=|\HH_{d,k}^{\un}|$ and hence $R^D\leq M$ if $C_1\gg1$.

Writing $j':=\ind(L_j')$ and $J':=\ind(L_{j+1}')$ it then follows from our inductive hypothesis functions, applied with respect to the $\VE_1$-admissible sequence of scales \[L_{j'+1}\geq L_{j'+2}\geq\cdots\geq L_{J'-1}\] which is possible as $J'-j'\gg J_{k-1}(\eps_1,R^D)$, that there is a scale $L_j$ with $j'\leq j <J'$ so that 
\eq\label{4.28}
\NN_{\la\De^0,Q_{\us}(L'_j)}\,(g_{\us,\f}^{\ur};\,\f\in\HH_{d,k-1}^{\un}) =
\MM_{\la,Q_{\us}(L'_j)}\,(g_{\us,\f}^{\ur};\,\f\in\HH_{d,k-1}^{\un}) + O(\eps_1)
\ee
for all $\lm\in[L_{j+1},L_j]$ 
uniformly in $\ur$ for $\us\notin S_{\eps_1}$, where $S_{\eps_1}\subs \Ga_{L'_j,Q}$ is a set of size $|S_{\eps_1}|\leq \eps_1 |\Ga_{L'_j,Q}|$. 

Since the cubes $Q_{\ut}(L_0)$ form a partition of $Q$ as $\ut$ runs through the grid $\Ga_{L_0,Q}$ the relative density of the set $S_{\eps_1}$ can substantially increase only of a few cubes $Q_{\ut}(L_0)$. Indeed, it is easy to see that $|T''_{\eps_1}|\leq \eps_1^{1/2} |\Ga_{L_0,Q}|$ for the set
\[T''_{\eps_1}:= \{\ut\in\Ga_{L_0,Q}:\ |S_{\eps_1}\cap Q_{\ut}(L_0)|\geq\eps_1^{1/2}\,|\Ga_{L_j',Q}\cap Q_{\ut}(L_0)|\}.\]


We claim that \eqref{4.11} holds for $\lm\in[L_{j+1},L_j]$  uniformly in  $t\notin T_{\eps}:=T'_{\eps}\cup T''_{\eps_1}$, $e\in \HH_{d,k}^{\un}$, and $1\leq m\leq M$.
Indeed, from \eqref{4.26}, \eqref{4.27}, and \eqref{4.28} and the fact that $|\al_{\us,\ur}|\leq 1$, it follows
\[\NN^d_{\la\De^0,Q_{\us}(L'_j)}\, (\of_{e,\us};\,e\in\HH_{d,k}^{\un}) = \MM^d_{\la,Q_{\us}(L'_j)}\, (\of_{e,\us};\,e\in\HH_{d,k}^{\un}) + O(\eps)\]
for $\us\notin S_{\eps_1}\cap Q_{\ut}(L_0)$ since $R^D\eps_1\ll\eps$. Finally, the fact that $\ut\notin T''_{\eps_1}$ together with localization, namely \eqref{4.18} and \eqref{4.19},  ensures that averaging over $\Ga_{L'_j,Q_{\ut}(L_0)}$ gives
\[\NN^d_{\la\De^0,Q_{\ut}(L_0)}\,(\of_{e,\ut};\,e\in\HH_{d,k}^{\un}) = \MM^d_{\la,Q_{\ut}(L_0)}\, (\of_{e,\ut};\,e\in\HH_{d,k}^{\un}) + O(\eps) + O(\eps_1^{1/2})\]
which in light of \eqref{4.16}, \eqref{4.17}, and the fact that $\eps_1\ll \eps^2$ complete the proof. 
\end{proof}



\subsection{Proof of Lemmas \ref{Lem4.5} and \ref{Lem4.6}}

\begin{proof}[Proof of Lemma \ref{Lem4.5}]

The argument is similar to that of Lemma \ref{vN-rect-ff}. Fix an edge, say $e_0=(11,12,\ldots,1k)$, and partition the edges $e\in\HH_{d,k}^{\un}$ in to as follows. Let $\HH_0$ be the set of those edges $e$ for which $1\notin \pi(e)$, and for $l=1,\ldots,n_1$ let $\HH_l$ denote the collection of edges of the form $e=(1l,j_2l_2,\ldots,j_kl_k)$, in other words $e\in \HH_l$ if $e=(1l,e')$ for some edge $e'=(j_2l_2,\ldots,j_kl_k)\in\HH_{d-1,k-1}^{\un}$. Accordingly write
\[\prod_{e\in\HH_{d,k}^{\un}} f_e(\ux_e) = \prod_{e\in \HH_0} f_e(\ux_e)\  \prod_{l=1}^{n_1}\prod_{e'\in \HH_{d-1,k-1}^{\un}} f_{1l,e'}(x_{1l},\ux_{e'}).\]

For $x\in Q_1$ and $\ux'=(\ux_2,\ldots,\ux_d)$ with $\ux_i\in Q_i^{n_i}$, define
\eq\label{4.38}
g_l(x,\ux') := \prod_{e'\in \HH_{d-1,k-1}^{\un}} f_{1l,e'}(x_{1l},\ux_{e'})\ee
Then one may write
\eq\label{4.39}
\NN^d_{\la\De^0,Q}(f_e;\,e\in\HH_{d,k}^{\un}) = \fint_{\ux_2}\ldots\fint_{\ux_d}\prod_{e\in \HH_0} f_e(\ux_e)
\left(\fint_{\ux_1}\prod_{l=1}^{n_1} g_l(x_{1l},\ux')\,d\si^{\la}_1(\ux_1)\right)\,d\si^{\la}_d(\ux_d)\ldots d\si^{\la}_2(\ux_2).\ee

For the inner integrals we have, using  \eqref{vN-11a}, the estimate 
\[\left(\fint_{\ux_1}\prod_{l=1}^{n_1}
g_l(x_{1l},\ux')\,d\si^{\la}_1\right)^2 \leq \|g_1\|_{U^1_L(Q)}^2 +O(\eps^{2^k}) = \fint_{y_{11}}\int_{y_{12}} g_1(y_{11})g_1(y_{12}) \psi^1_L(y_{12}-y_{11})\,dy_{11}\,dy_{12}+O(\eps^{2^k}).\]
provided  $0<L\ll (\eps^{2^k})^6 \la$, where as in the proof of Lemma \ref{vN-2} we use the notation
\[\psi^i_L(y_2-y_1)=\int_t \chi^i_L(y_1-t)\chi^i_L(y_2-t)\,dt\]
with $\chi^i_L:= L^{-n_i}1_{[-L/2,L/2]^{n_i}}$ for $1\leq i\leq k$. 
By Cauchy-Schwarz we then have
\[
\left|\NN^d_{\la\De^0,Q}(f_e;\,e\in\HH_{d,k}^{\un}\right|^2 \leq
\int_{\uy_1}\fint_{\ux_2}\ldots\fint_{\ux_d} \prod_{e'\in\HH_{d-1,k-1}^{\un}}\!\!\!\! f_{11,e'}(x_{11},\ux_{e'})f_{11,e'}(x_{12},\ux_{e'})\,d\si^{\la}_d\ldots d\si^{\la}_2\,d\omega^1_L(\uy_1)+O(\eps^{2^k})\]
where $d\omega^i_L(\uy_i)=|Q_i|^{-1}\psi^i_L(y_{i2}-y_{i1})\,dy_{i1}\,dy_{i2}$ with $\uy_i=(y_{i1},y_{i2})\in Q_i^2$ for $1\leq i\leq k$. 

The expression we have obtained above is similar to the one in \eqref{4.2} except for the following changes. The variable $\ux_1\in Q_1^{n_1}$ is replaced by $\uy_1\in Q_1^2$ and the measure $d\si^{\la}_1$ by $d\omega^1_L$. The functions $f_{1l,e'}$ are replaced by $f_{11,e'}$, for $1\leq l\leq n_1$, while the functions $f_e$ for all $e\in \HH_{d,k}^{\un}$ such that $1\notin \pi(e)$ are eliminated, that is replaced by 1. Repeating the same procedure for $i=2,\ldots,k$ replaces all variables $\ux_i$ with variables $\uy_i$ as well as the measures $d\si^{\la}_i$ with $d\omega^i_L$. The procedure eliminates all functions $f_e$ when $e$ is an edge such that $i\notin \pi(e)$ for some $1\leq i\leq k$; for the remaining edges, when $\pi(e)=(1,\ldots,k)$, it replaces the functions $f_e$ with $f_{e_0}=f_{11,21,\ldots,1k}$. For $k<i$ the variables $\ux_i$ and the measures $d\si^{\la}_i$ are not changed, however integrating in these variables will have no contribution as the measures are normalized. Thus one obtains the following final estimate
\eq\label{4.41}
\left|\NN_{\la\De^0,Q}(f_e;\,e\in\HH_{d,k}^{\un}\right|^{2^k} \leq  \frac{1}{|Q_1|}\int_{\uy_1}\ldots\frac{1}{|Q_k|}\int_{\uy_k} \prod_{e\in\HH_{k,k}^{\uu}} f_{e_0}(\uy_e) \prod_{i=1}^k \psi^i_L(y_{i2}-y_{i1})\,dy_{i1}\,dy_{i2}+O(\eps^{2^k})\ee
noting that these integrals are not normalized.
 Thus, one may write the expression in \eqref{4.41}, using a change of variables $y_{i1}:=y_{i1}-t_i$, $y_{i2}:=y_{i2}-t_i$, as
\eq\label{4.42}
\frac{1}{|Q_1|}\int_{t_1}\fint_{\uy_1\in t_1+Q_1}\!\!\!\!\!\! \ldots
 \ \frac{1}{|Q_k|}\int_{t_k}\fint_{\uy_k\in t_k+Q_k} \prod_{e\in\HH_{k,k}^{\uu}} f_{e_0}(\uy_e)\,d\uy_1\ldots d\uy_k\,d\ut =
\|f_{e_0}\|_{\Box_L (Q_{\pi(e_0)})}^{2^k} +O(\eps^{2^k})\ee
where the last equality follows from the facts that the function $f_{e_0}$ is supported on the cube $Q_{\pi(e_0)}$ and hence the integration in $\ut$ is restricted to the cube $Q+Q(L)$, giving rise an error of $O(L/l(Q))$. Estimate (\ref{4.12}) follows from \eqref{4.41} and \eqref{4.42} noting that the above procedure can be applied to any  $e\in\HH_{d,k}^{\un}$ in place of $e_0$. Estimate (\ref{MMM}) is established similarly. \end{proof}

\begin{proof}[Proof of Lemma \ref{Lem4.6}]

 For $j=0$ we set $\BB_{e',\ut}(L_0):=\{Q_{\ut}(L_0),\emptyset\}$ and $\BB_{e',\f',\us}(L_0):=\{Q_{\us_{\f'}}(L_0),\emptyset\}$ for $e'\in\HH_{d,k}$, $\f'\in\partial e'$, and  $\ut,\us\in\Ga_{L_0,Q}$. We will develop $\si$-algebras $\BB_{e',\ut}(L_j)$ of scale $L_j$ such that \eqref{4.14} holds with $\comp(\BB_{e',\f',\us}(L_j))\leq j$. 
 
 We define the \emph{total energy} of a family of functions $f^m_{e,\ut}$ with respect to a family of $\si$-algebras $\BB_{e',\ut}(L_j)$ as
\eq\label{4.29}
\mathcal{E}(f^m_{e,\ut}|\BB_{e',\ut}(L_j)):=\E_{\ut\in\Ga_{L_0,Q}} \sum_{m=1}^M\sum_{e\in \HH_{d,k}^{\un}} \|\E(f^m_{e,\ut}|\BB_{\pi(e),\ut}(L_j))\|_{L^2(Q_{\ut_{\pi(e)}}
(L_0))}^2.\ee

Since $|f^m_{e,\ut}|\leq 1$ for all $e$, $m$, and $\ut$ it follows that the total energy is bounded by $M\cdot|\HH_{d,k}^{\un}|=O(M)$. 
Our strategy will be to show that if \eqref{4.13} does not hold then there exist a family of $\si$-algebras $\BB_{e',\ut}(L_{j+2})$ such that the total energy of the family of functions $f^m_{e,\ut}$ is increased by at least $c_k \eps^{2^{k+3}}$ with respect to this new family of $\si$-algebras, and at the same time ensuring that \eqref{4.14} remains valid with  $\comp(\BB_{e',\f',\us}(L_{j+2}))\leq j+2$. This iterative process must stop at some $j=O(M\,\eps^{-2^{k+3}})$ proving the Lemma.

Assume that we have developed $\si$-algebras $\BB_{e',\ut}(L_j)$ and $\BB_{e',\f',\us}(L_{j})$ of scale $L_j$ such that \eqref{4.14} holds with $\comp(\BB_{e',\f',\us}(L_j))\leq j$. If \eqref{4.13} does not hold then $|T_{\eps}|\geq\eps |\Ga_{L_0,Q}|$ for the set
\[
T_{\eps}:= \{\ut\in\Ga_{L_0,Q}\,:\, \|f^m_{e,\ut}-\E(f^m_{e,\ut}|\BB_{\pi(e),\ut}(L_j))\|_
{\Box_{L_{j+1}}(Q_{\ut_{\pi(e)}}(L_0))}\geq\eps\text{ \ for some $e\in\HH_{d,k}^{\un}$ and $1\leq m\leq M$}\}.
\]

Fix $\ut\in T_{\eps}$ and let $e\in\HH_{d,k}^{\un}$ and $1\leq m\leq M$ be such that \[\|f^m_{e,\ut}-\E(f^m_{e,\ut}|\BB_{\pi(e),\ut}(L_j))\|_
{\Box_{L_{j+1}}(Q_{\ut_{\pi(e)}}(L_0))}\geq\eps\] and write $e':=\pi(e)$. Consider the partition of the cube $Q_{\ut_{e'}}(L_0)$ into small cubes $Q_{\us_{e'}}(L_{j+2})$ where $\us_{e'}\in\Ga_{L_{j+2},Q_{e'}}\cap Q_{\ut_{e'}}(L_0)$. By the localization properties of the $\Box_{L_{j+1}}(Q)$-norm, and the fact that $L_{j+2}\ll \eps^{2^k} L_{j+1}$ we have that
\[\|f\|_{\Box_{L_{j+1}}(Q_{\ut_{e'}}(L_0))}^{2^k}\leq
\E_{\us_{e'}\in\Ga_{L_{j+2},Q_{\ut_{e'}}(L_0)}}\ \|f\|_{\Box(Q_{\us_{e'}}(L_{j+2}))}^{2^k} +\frac{\eps^{2^k}}{2}\]
for any function $f:Q_{\ut_{e'}}(L_0)\to [-1,1]$.
Thus there exists a set $S_{\eps,e,\ut}\subs \Ga_{L_{j+2},Q_{\ut_{e'}}(L_0)}$ of size \[|S_{\eps,e,\ut}|\geq \frac{\eps^{2^k}}{4}|\Ga_{L_{j+2},Q_{\ut_{e'}}(L_0)}|\] such that 
\eq\label{4.31}
\|f^m_{e,\ut}-\E(f^m_{e,\ut}|\BB_{e',\ut}(L_j))\|_
{\Box(Q_{\us_{e'}}(L_{j+2})}^{2^k} \geq \frac{\eps^{2^k}}{4}
\ee
for all $\us_{e'}\in S_{\eps,e,\ut}$.

For a given cube $Q$ and functions $f,g:Q\to\R$, define the normalized inner product of $f$ and $g$ as
\[\langle f,g \rangle_Q := \fint_Q f(x)g(x)\,dx.\] Then by the well-known property of the $\Box$-norm, see for example \cite{TaoMontreal} or the proof of Lemma \ref{KvN-ff}, it follows from \eqref{4.31} that there exits sets \[B_{\f',\us_{e'},\ut}\subs Q_{\us_{\f'}}(L_{j+2})\] for $\f'\in\partial e'$ such that
\eq\label{4.32}
\Bigl\langle f^m_{e,\ut}-\E(f^m_{e,\ut}|\BB_{e',\ut}(L_j))\,,\,\prod_{\f'\in\partial e'} 1_{B_{\f',\us_{e'},\ut}}\Bigr\rangle_{Q_{\us_{e'}}(L_{j+2})}\geq \frac{\eps^{2^k}}{2^{k+2}}.
\ee

If $\us\in\Ga_{L_{j+2},Q}\,$ then there is a unique $\,\ut=\ut(\us)\in\Ga_{L_0,Q}$ such that $\,\us\in Q_{\ut}(L_0)$. If $\ut\in T_{\eps}$ and $\us_{e'}\in S_{\eps,e,\ut}$ then we define the $\si$-algebras $\BB_{\f',e',\us}(L_{j+2})$ on $\,Q_{\us_{\f'}}(L_{j+2})$ as follows. Write $B_{\f',e',\us}=B_{\f',\us_{e'},\ut}$ where $\ut=\ut(\us)$ and let $\BB_{\f',e',\us}(L_{j+2})$ be the $\si$-algebra generated by the set $B_{\f',e',\us}$ and the $\si$-algebra $\BB_{\f',e',\us'}(L_j)$ restricted to $Q_{\us_{\f'}}(L_{j+2})$ where $\us'\in\Ga_{L_j,Q}$ is the unique element so that $\us\in Q_{\us'}(L_j)$. Note that that the complexity of the $\si$-algebra $\BB_{\f',e',\us}(L_{j+2})$ is at most one larger then the  complexity of the $\si$-algebra $\BB_{\f',e',\us'}(L_j)$ as restricting a $\si$-algebra to a set does not increase its complexity. If $\ut=\ut(\us)\notin T_{\eps}$ or $\us_{e'}\notin S_{\eps,e,\ut}$ then let $\BB_{\f',e',\us}(L_{j+2})$ be simply the restriction of $\BB_{\f',e',\us'}(L_j)$ to the cube $Q_{\us_{\f'}}(L_{j+2})$, or equivalently define the sets $B_{\f',e',\us}:=Q_{\us_{\f'}}(L_{j+2})$.  Finally, let
\eq\label{4.33}
\BB_{e',\us}(L_{j+2}):=\bigvee_{\f'\in \partial e'} \BB_{\f',e',\us}(L_{j+2})
\ee
be the corresponding $\si$-algebra on the cube $Q_{\us_{e'}}(L_{j+2})$. 

Since the cubes $Q_{\us_{e'}}(L_{j+2})$ partition the cube $Q_{\ut_{e'}}(L_0)$ as $\us_{e'}$ runs through the grid $\Ga_{L_{j+2},Q_{e'}}\cap Q_{\ut_{e'}}(L_0)$, these $\si$-algebras define a $\si$-algebra $\BB_{e',\ut}(L_{j+2})$ on $Q_{\ut_{e'}}(L_0)$, such that its restriction to the cubes $Q_{\us_{e'}}(L_{j+2})$ is equal to the $\si$-algebras $\BB_{e',\us}(L_{j+2})$.

Since the function $\prod_{\f'\in \partial e'} 1_{B_{\f',e',\us}}$ is measurable with respect to the $\si$-algebra $\BB_{e',\ut}(L_{j+2})$ restricted to the cube $Q_{\us_{e'}}(L_{j+2})$ one clearly has
\eq\label{4.34}
\langle\,f^m_{e,\ut}-\E(f^m_{e,\ut}|\BB_{e',\ut}(L_{j+2})),\,\prod_{\f'\in\partial e'} 1_{B_{\f',e',\us}}\,\rangle_{Q_{\us_{e'}}(L_{j+2})} = 0.
\ee
and hence, by \eqref{4.32}, that
\eq\label{4.35}
\langle\,\E(f^m_{e,\ut}|\BB_{e',\ut}(L_{j+2}))-\E(f^m_{e,\ut}|\BB_{e',\ut}(L_j)),\,\prod_{\f'\in\partial e'} 1_{B_{\f',e',\us}}\,\rangle_{Q_{\us_{e'}}(L_{j+2})}\geq \frac{\eps^{2^k}}{2^{k+2}}.
\ee

It then follows from Cauchy-Schwarz and orthogonality, using the fact that the $\si$-algebra $\BB_{e',\ut}(L_{j+2}))$ is a refinement of $\BB_{e',\ut}(L_{j+2})$, that
\begin{align}\label{4.35}
 \quad\ \|\E(f^m_{e,\ut}|\BB_{e',\ut}(L_{j+2}))-&\E(f^m_{e,\ut}|\BB_{e',\ut}(L_j))\|_{L^2(Q_{\us_{e'}}(L_{j+2}))}^2\\
&= \|\E(f^m_{e,\ut}|\BB_{e',\ut}(L_{j+2}))\|_{L^2(Q_{\us_{e'}}
(L_{j+2}))}^2 -
\|\E(f^m_{e,\ut}|\BB_{e',\ut}(L_j))\|_{L^2(Q_{\us_{e'}}(L_{j+2}))}^2 \nonumber\\
&\geq\Bigl( \frac{\eps^{2^k}}{2^{k+2}} \Bigr)^2
\nonumber
\end{align}
for $\us_{e'}\in S_{\eps,e,\ut}$.
Since $|S_{\eps,e,\ut}|\geq \dfrac{\eps^{2^k}}{4} |\Ga_{L_{j+2},Q_{\ut_{e'}}}(L_0)|$ averaging over $\,\us_{e'}\in
\Ga_{L_{j+2},Q_{\ut_{e'}}}(L_0)$ implies
\eq\label{4.36}
\|\E(f^m_{e,\ut}|\BB_{e',\ut}(L_{j+2}))\|_{L^2(Q_{\ut_{e'}}
(L_0))}^2\,\geq\,\|\E(f^m_{e,\ut}|\BB_{e',\ut}(L_j))\|_{L^2(Q_{\ut_{e'}}(L_0))}^2 + \frac{\eps^{2^{k+2}}}{2^{2k+6}}.
\ee

At this point we have shown that if $\ut\in T_{\eps}$ then there exists an edge $e\in\HH_{d,k}^{\un}$, $1\leq m\leq M$, and $\si$-algebras $\BB_{e',\ut}(L_{j+2}))$ of scale $L_{j+2}$ on $Q_{\ut_{e'}}(L_0)$, with $e'=\pi(e)$, such that \eqref{4.36} holds. 

For all $e''\in\HH_{d,k}$ with $e''\neq e'$ let $\BB_{\f',e'',\us}(L_{j+2})$ be the restriction of the $\si$-algebra $\BB_{\f',e'',\us'}(L_j)$ to the cube $Q_{\us_{\f'}}(L_{j+2})$, where $\us'$ is such that $\us\in Q_{\us'}(L_j)$. By \eqref{4.33} this implies that $\BB_{e'',\us}(L_{j+2})$ is also the restriction of $\BB_{e'',\us'}(L_j)$ to the cube $Q_{\us_{e''}}(L_{j+2})$, and hence the $\si$-algebra $\BB_{e'',\ut}(L_{j+2})$ is generated by the grid $\GG_{L_{j+2},Q_{\ut_{e''}}(L_0)}$ and the $\si$-algebra $\BB_{e'',\ut}(L_j)$. 

We have therefore defined a family of the $\si$-algebras $\BB_{e',\ut}(L_{j+2})$ for $e'\in\HH_{d,k}$, satisfying
\[
\sum_{m=1}^M\sum_{e\in\HH_{d,k}^{\un}}\|\E(f^m_{e,\ut}|\BB_{\pi(e),\ut}(L_{j+2}))\|_{L^2(Q_{\ut_{\pi(e)}}
(L_0))}^2\geq\sum_{m=1}^M\sum_{e'\in\HH_{d,k}^{\un}}\|\E(f^m_{e,\ut}|\BB_{\pi(e),\ut}(L_j))\|_{L^2(Q_{\ut_{\pi(e)}}(L_0))}^2 +\frac{\eps^{2^{k+2}}}{2^{2k+6}}. \]
Using the fact that $|T_{\eps}|\geq \eps |\Ga_{L_0,Q}|$ and averaging over $\ut\in \Ga_{L_0,Q}$ it follows using the notations of \eqref{4.29} that
\[\mathcal{E}(f^m_{e,\ut}|\BB_{e',\ut}(L_{j+2}))\,\geq\,
\mathcal{E}(f^m_{e,\ut}|\BB_{e',\ut}(L_j)) +\frac{\eps^{2^{k+3}}}{2^{2k+6}}.\]

As the total energy $\mathcal{E}(f^m_{e,\ut}|\BB_{e',\ut}(L_j))$ is bounded by $O(M)$, the process must stop at a step $j=O(M\,\eps^{-2^{k+3}})$ where \eqref{4.13} holds for a $\sigma$-algebra of ``local complexity" at most $j$,  completing the proof of Lemma \ref{Lem4.6}. \end{proof}


\section{The base case of an inductive strategy to establish Theorem \ref{ProdZ}}\label{BaseZ}

In this section we will ultimately establish the base case of our more general inductive argument. 
We will however start by giving a (new) proof of Theorem B$^\prime$, namely the case $d=1$ of Theorem \ref{ProdZ}. 


\subsection{A Single Simplex in $\Z^n$}

Let $\De^0=\{v_1=0,v_2,\dots,v_{n_1}\}$ be a fixed non-degenerate simplex of $n_1$ points in $\Z^n$ with $n=2n_1+3$ and define $t_{kl}:=v_k\cdot v_l$ for $2\leq
k,l\leq n_1$.
Recall, see \cite{Magy09}, that a simplex $\De=\{m_1=0,\dots,m_{n_1}\}\subs \Z^n$ is isometric to $\la\De^0$ if and only if
$m_k\cdot m_l =\la^2 t_{kl}$ for all $2\leq k,l\leq n_1$.

For any positive integer $q$ and $\lm\in q\sqrt{\N}$ we define $S_{\la\De^0,q}(m_2,\dots,m_{n_1}):\Z^{n(n_1-1)}\to\{0,1\}$ be the function whose value is 1 if $m_k\cdot m_l=\la^2 t_{kl}$ with both $m_k$ and $m_l$ in $(q\Z)^n$
for all $2\leq k,l\leq n_1$   and is equal to 0 otherwise. It is a well-known fact in number theory, see \cite{Kitaoka} or \cite{Magy09}, that for $n\geq2n_1+1$ we have that
\[\sum_{m_2,\ldots,m_{n_1}} S_{\la\De^0,q} (m_2,\ldots,m_{n_1}) =\rho(\De^0)\, (\la/q)^{(n-n_1)(n_1-1)}(1+O(\la^{-\tau}))\]
for some absolute constant $\tau>0$ and some constant $\rho(\De^0)>0$, the so-called singular series, which can be interpreted as the product of the densities of the solutions of the above system of equations among the $p$-adics and among the reals.
Thus if we define \[\si_{\la\De^0,q}:=\rho(\De^0)^{-1}(\la/q)^{-(n-n_1)(n_1-1)}S_{\lm\De^0,q}\] then $\si_{\la\De^0,q}$ is normalized in so much that \[\sum_{m_2,\ldots,m_{n_1}}\si_{\la\De^0,q}(m_2,\ldots,m_{n_1})= 1+O(\la^{-\tau})\] for some absolute constant $\tau>0$.

\smallskip

Let $Q\subs\Z^n$ be a fixed cube and let $l(Q)$ denotes its side length.
For any family of functions \[f_1,\ldots,f_{n_1}:Q\to [-1,1]\] and $0<\lm\ll l(Q)$ we define the following two multi-linear expressions
\eq\label{multilin-config-1}
\NN^1_{\la\De^0,q,Q}(f_1,\dots,f_{n_1}) := \E_{m_1\in Q} \sum_{m_2,\ldots,m_{n_1}} f_1(m_1)\ldots f_{n_1}(m_{n_1})\,\si_{\la\De^0,q} (m_2-m_1,\ldots,m_{n_1}-m_1)
\ee
and 
\eq
\MM^1_{\lm,q,Q}(f_1,\dots,f_{n_1}) := \E_{t\in Q} \, \E_{m_1,\ldots,m_{n_1}\in t+Q(q,\lm)}\ f_1(m_1)\ldots f_{n_1}(m_{n_1})
\ee
where $Q(q,\lm):=[-\frac{\lm}{2},\frac{\lm}{2}]^n\cap(q\Z)^n$.
Note that if $S\subs Q$ and $\NN^1_{\la\De^0,q,Q}(1_S,\dots,1_S)>0$ then $S$ must contain an isometric copy of $\la\De^0$, while if $|S|\geq \de |Q|$ for some $\de>0$ then as before H\"older implies that
 \eq\label{ML-lowerb-1Z}
\MM^1_{\lm,q,Q}(1_S,\dots,1_S) \geq\de^n-O(\eps)\ee
for all scales $\lm\in q\sqrt{\N}$ with $0<\lm\ll\eps\, l(Q)$.

Recall that for any  $0<\VE\ll1$ and positive integer $q$ we call a sequence $L_1\geq \cdots \geq L_J$ $(\eps,q)$-\emph{admissible} if $L_j/L_{j+1}\in\N$ and $L_{j+1}\ll\eps^2 L_j$ for all $1\leq j<J$ and $L_J/q\in\N$.  Note that if $ \lm_1\geq\cdots\geq\lm_{J'}\geq1$ is any lacunary sequence in $q\sqrt{\N}$ with $J'\gg ( \log\VE^{-1})\,J+\log q$,
one can always finds an $(\eps,q)$-admissible sequence of scales  $ L_1\geq\cdots \geq L_{J}$ with the property that for each $1\leq j<J$ the interval $[L_{j+1}, L_j]$  contains at least two consecutive elements  from the original lacunary sequence.

In light of these observations we see that the following ``counting lemma" ultimately establishes a quantitatively stronger version of Proposition B$^\prime$ that appeared in Section \ref{d&sZ} and hence immediately establishes Theorem \ref{ProdZ} for $d=1$.



\begin{prop}\label{Lem5.5}

Let $0<\eps\ll1$ and $q_j:=q_1(\eps)^j$ for $j\geq 1$ with $q_1(\VE):=\lcm\{1\leq q\leq C\VE^{-10}\}$. 

There exists $J_1=O(\VE^{-2})$ such that for any $(\eps,q_{J_1})$-admissible sequence of scales $l(Q)\geq L_1\geq\cdots \geq L_{J_1}$ and $S\subseteq Q$ there is some $1\leq j< J_1$ such that
\eq\label{1.4}
\NN^1_{\la\De^0,q_j,Q}(1_S,\dots,1_S)=\MM^1_{\la,q_j,Q}(1_S,\dots,1_S)+O(\eps)\ee
for all $\lm\in q_j\sqrt{\N}$ with $L_{j+1}\leq \lm\leq L_j$.
\end{prop}

As in the continuous setting the proof of Proposition \ref{Lem5.5} has two main ingredients, namely Lemmas \ref{Lem5.4} and \ref{Lem5.3} below. 
 In these lemmas, and for the remainder Sections \ref{BaseZ} and \ref{generalcaseZ}, we will continue to use the notation 
\[q_1(\VE):=\lcm\{1\leq q\leq C\VE^{-10}\}\] for any given $\VE>0$.

\begin{lem}[A Generalized von Neumann inequality]\label{Lem5.4}\

Let $0<\eps\ll 1$, $q,q'\in\N$ with $qq_1(\eps)|q'$, and $\lm\in q\sqrt{\N}$ with $\lm\ll l(Q)$ and $1\ll L\ll\VE^{10}\lm$. 
For any collection of functions $f_1,\dots,f_{n_1}:Q\to [-1,1]$ we have
\eq\label{5.24}
|\NN^1_{\la\Delta^0,q,Q}(f_1,\ldots,f_{n_1})|\leq
\min_{1\leq i\leq n_1} \|f_i\|_{U^1_{q',L}(Q)}+ O(\eps)\ee
where for any function $f:Q\to[-1,1]$ we define
\eq
\|f\|_{U^1_{q,L}(Q)}:=\Bigl(\frac{1}{|Q|}\sum_{t\in Q}|f*\chi_{q,L}(t)|^2\Bigr)^{1/2}
\ee
with $\chi_{q,L}$ denoting the  normalized characteristic function of the cubes $Q(q,L):=[-\frac{L}{2},\frac{L}{2}]^n\cap (q \Z)^n$.
\end{lem}

For any cube $Q\subseteq\Z^n$ of side length $l(Q)$ and $q,L\in \N$ satisfying $q\ll L$ with $L$ dividing $l(Q)$, we shall now partition $Q$ into cubic grids $Q_{t}(q,L)=t+((q\Z)^n \cap Q(L))$, with $Q(L)=[-\frac{L}{2},\frac{L}{2}]^n$ as usual. 
These grids form the atoms of a $\si$-algebra $\GG_{q,L,Q}$. 
Note that if $q|q'$ and $L'|L$ then $\GG_{q,L,Q}\subs \GG_{q',L',Q}$.

\begin{lem}[A Koopman-von Neumann type decomposition]\label{Lem5.3}\

Let $0<\eps\ll1$ and  $q_j:=q_1(\eps)^j$ for all $j\geq 1$.
There exists an integer
$\bar{J}_1=O(\VE^{-2})$ such that any $(\eps,q_{\bar{J}_1})$-admissible sequence of scales $l(Q)\geq L_1\geq\cdots \geq L_{\bar{J}_1}$  and function $f: Q\to [-1,1]$ there is some $1\leq j< \bar{J}_1$ such that
\eq\label{5.21}
\|f-\E(f|\GG_{q_j,L_j,Q})\|_{U^1_{q_{j+1},L_{j+1}}(Q)}\,\leq\,
\eps.\ee
\end{lem}

The  reduction of Proposition \ref{Lem5.5} to these two lemmas is essentially identical to the analogous argument in the continuous setting as presented at the end of Section \ref{S4.1}, we choose to omit the details.

\begin{proof}[Proof of Lemma \ref{Lem5.4}]

We will rely on some prior exponential sum estimates, specifically Propositions 4.2 and 4.4  in \cite{Magy09}.
First we deal with the case $n_1\geq 3$. By the change of variables $m_1:=m_1,\ m_i:=m_i-m_1$ for $2\leq i\leq n_1$, one may write
\[\NN^1_{\la\Delta^0,q,Q}(f_1,\ldots,f_{n_1}):=\E_{m_1\in Q_N} \sum_{m_2,\ldots,m_{n_1}} f_1(m_1)f_2(m_1+m_2)\cdots f_{n_1}(m_1+m_{n_1})\,\si_{\la\De^0,q} (m_2,\ldots,m_{n_1}).\]


We now write
\[\si_{\la\De^0,q}(m_2,\ldots,m_{n_1}) = \si_{\la\De^{0\prime},q}(m_2,\ldots,m_{n_1-1}) \, \si_{\la,q}^{m_2,\ldots,m_{n_1-1}}(m_{n_1})\]
where $\De^{0\prime}=\{v_1=0,v_2,\dots,v_{n_1-1}\}$ and for each $m_2,\ldots,m_{n_1-1} \in (q\Z)^n$ we are using $\si_{\la,q}^{m_2,\dots,m_{n_1-1}}(m)$ denote the (essentially) normalized indicator function of the subset of $(q\Z)^n$ that contains $m$ if and only if $m\cdot m_k=\la^2 t_{kn_1}$  for all $2\leq k\leq n_1$. 

Using the fact that $|f_i|\leq 1$, together with Cauchy-Schwarz and Plancherel, one can then easily see that
\eq\label{5.25}
|\NN^1_{\la\Delta^0,q,Q}(f_1,\ldots,f_{n_1})|^2 \leq |Q|^{-1} \int_{\xi\in \mathbb{T}^n} |\wh{f}_{n_1}(\xi)|^2 H_{\la,q}(\xi)\,d\xi\ee
with
\[H_{\la,q}(\xi) = \sum_{m_2,\ldots,m_{n_1}} \si_{\la\De^{0\prime},q}(m_2,\ldots,m_{n_1-1})\,|\widehat{\si_{\la,q}^{m_2,\ldots,m_{n_1-1}}}(\xi)|^2.\]

It then follows by Propositions 4.2 and  4.4 in \cite{Magy09}, with $\de=\eps^4$ and after rescaling by $q$, that
in addition to being non-negative and
uniformly bounded in $\xi$ we in fact have
\eq\label{5.26}
H_{\la,q}(\xi) =O(\eps)\quad \text{
whenever}\quad
\left| q\xi -\frac{l}{q_1(\eps)}\right| \geq \frac{q}{\eps^4\la},\ee
for all $l\in\Z^n$.

We note that the expression $H_{\la,q}(\xi)$ may be interpreted as the Fourier transform of the indicator function of the set of integer points on a certain variety, and  estimate \eqref{5.26} indicates that this concentrates near rational points of small denominator. It is this crucial fact from number theory which makes results like Theorem B$^\prime$ possible.

\smallskip

Since
\[
\widehat{\chi}_{q,L}(\xi) = \frac{q^n}{L^n} \sum_{m\in [-\frac{L}{2},\frac{L}{2})^n,\ q|m} e^{-2\pi i m\cdot\xi}\]
it is easy to see that  $\widehat{\chi}_{q,L}(l/q)=1$ for all $l\in\Z^n$ and that there exists some absolute constant $C>0$ such that
\eq\label{5.19}
0\leq 1-\widehat{\chi}_{q,L}(\xi)^2\leq C\, L\,|\xi-l/q|\ee
for all $\xi\in\mathbb{T}^n$ and $l\in\Z^n$. It is then easy to see using our assumption that $q q_1(\eps)|q'$ that
\eq
0\leq H_{\la,q}(\xi) (1- \wh{\chi}_{q',L}(\xi)^2) \leq C \eps
\ee
for some constant $C>0$ uniformly in $\xi\in\mathbb{T}^n$ provided $L\ll\eps^5\la$.
Substituting inequality \eqref{5.21} into \eqref{5.25}, we obtain
\begin{align*}
|\NN^1_{\la\Delta^0,q,Q}(f_1,\ldots,f_{n_1})|^2 &\leq |Q|^{-1} \left(
\int |\hat{f}_{n_1}(\xi)|^2 H_{\la}(\xi) \wh{\chi}_{q',L}(\xi)^2\,d\xi\,+\,\int |\hat{f}_{n_1}(\xi)|^2 H_{\la}(\xi) (1-\wh{\chi}_{q',L}(\xi)^2)\,d\xi\right)\\
&\leq \|f_{n_1}\|_{U^1_{q',L}(Q)}^2 + O(\eps)
\end{align*}
provided $L\ll\eps^5\la$.
This proves Lemma \ref{Lem5.4} for $k\geq 3$, as it is clear that by re-indexing the above estimate holds for any of the functions $f_i$ in place of $f_{n_1}$.
 For $n_1=2$ an easy modification of arguments in \cite{LM19}, specifically the proof of Lemma 3 therein, establishes that
\[|\NN^1_{\la\Delta^0,q,Q}(f_1,f_{2})|^2\leq \|f_i\|_{U^1_{q',L}(Q)}^2 + O(\eps)\]for $i=1,2$ provided $L\ll\eps^5\la$. \end{proof}

\begin{proof}[Proof of Lemma \ref{Lem5.3}]

Let $q,L\in\N$ such that $L|N$, $q|L$. The ``modulo $q$" grids $Q_t(q,L)=t+Q(q,L)$ partition the cube $Q$ with $t$ running through the set $\Ga_{q,L,Q}=\{1,\ldots,q\}^n + \Ga_{L,Q}$, where $\Ga_{L,Q}$ denote the centers of the ``integer" grids $t+Q(L)$ in an initial partition of $Q$. Let $q',L'$ be positive integers so that $q|q'$, $L'|L$ and $L'\ll \eps^2 L$. If $s\in\Ga_{q',L',Q}$ and $t\in Q_s(q',L')$ then $|t-s|=O(L')$ and hence
\[\E_{x\in Q_t(q,L)} g(x) = \E_{x\in Q_s(q,L)} g(x) + O(L'/L)\]
for any function $g:Q\to [-1,1]$. Moreover, since the cube $Q_s(q,L)$ is partitioned into the smaller cubes $Q_t(q',L')$, we have by Cauchy-Schwarz
\[|\E_{x\in Q_s(q,L)}\, g(x)|^2 \leq \E_{t\in \Ga_{q',L',Q_s(q,L)}} |\E_{x\in Q_t(q',L')} g(x)|^2.\]

From this it is easy to see that
\[ \|g\|_{U^1_{q,L}(Q)}^2 = \E_{t\in Q}|\E_{x\in Q_t(q,L)} g(x)|^2 \leq \E_{t\in \Ga_{q',L',Q}}\ |\E_{x\in Q_t(q',L')}g(x)|^2+O(L'/L)\]
and we note that the right side of the above expression is $\|\E(g|\GG_{q',L',Q})\|_{L^2(Q)}^2$ since the conditional expectation function $\E(g|\GG_{q',L',Q})$ is constant and equal to $\E_{x\in Q_t(q',L')} g(x)$ on the cubes $Q_t(q',L')$.

Now suppose \eqref{5.21} does not hold for some $j\geq 1$, that is 
\[\|f-\E(f|\GG_{q_j,L_j,Q})\|_{U^1_{q_{j+1},L_{j+1}}(Q)}^2\geq\eps^2.\]
Since $L_{j+2}\ll\eps^2 L_{j+1}$, $L_{j+2}|L_j$, and $q_{j+1}|q_{j+2}$ we
 can apply the above observations to $g:=f-\E(f|\GG_{q_j,L_j,Q})$ and obtain, by orthogonality, that
 \eq
 \|\E(f|\GG_{q_{j+2},L_{j+2},Q})\|_{L^2(Q)}^2 \geq \|\E(f|\GG_{q_j,L_j,Q})\|_{L^2(Q)}^2 +c\eps^2
 \ee
 for some constant $c>0$. 
Since the above expressions are clearly bounded by $1$, the above procedure must stop in $O(\eps^{-2})$ steps at which \eqref{5.21} must hold for some $1\leq j\leq \bar{J}_1(\VE)$ with $\bar{J}_1(\VE)=O(\eps^{-2})$.
\end{proof}

\subsection{The base case of our general inductive strategy}\label{baseZ}\

Let $Q=Q_1\times\ldots\times Q_d$ with $Q_i\subs\Z^{2n_i+3}$ be cubes of equal side length $l(Q)$ and $\De_i^0\subs \Z^{2n_i+3}$ be a non-degenerate simplex of $n_i$ points for $1\leq i\leq d$. 

We note that for any $q_0\in\N$ and scale $L_0$ dividing $l(Q)$ if $\ut=(t_1,\ldots,t_d)\in\Ga_{q_0,L_0,Q}$, then the corresponding grids $Q_{\ut}(q_0,L_0)$ in the partition of $Q$ take the form $Q_{\ut}(q_0,L_0)=Q_{t_1}(q_0,L_0)\times\dots\times Q_{t_d}(q_0,L_0)$.

As in the continuous setting we will ultimately need a parametric version of Proposition \ref{Lem5.5}, namely Proposition \ref{Cor5.1} below.

\begin{prop}[Parametric  Counting Lemma on $\Z^n$ for Simplices]\label{Cor5.1} 

Let $0<\VE\leq 1$ and $R\geq1$.

There exists an integer $J_1=J_1(\VE,R)=O(R\,\VE^{-4})$ such that for any  $(\eps,q_{J_1})$-admissible sequence of scales $L_0\geq L_1\geq\cdots \geq L_{J_1}$
with $L_0$ dividing $l(Q)$ and $q_j:=q_0 q_1(\eps)^j$ for $0\leq j\leq J_1$ with $q_0\in\N$, and collection of functions 
\[\text{$f_{k,\ut}^{i,r}:\,Q_{t_i}(q_0,L_0)\to [-1,1]$ \ with \ $1\leq i\leq d$, $1\leq k\leq n_i$, $1\leq r\leq R$ and $\ut\in \Ga_{q_0,L_0,Q}$}\] 
there exists  $1\leq j< J_1$ and a set $T_\eps\subs\Ga_{q_0,L_0,Q}$ of size $|T_\eps |\leq \eps |\Ga_{q_0,L_0,Q}|$ such that
\eq\label{5.37}
\NN^1_{\la\De_i^0,q_j,Q_{t_i}(q_0,L_0)}(f_{1,\ut}^{i,r},\ldots,f_{n_i,\ut}^{i,r}) = \MM^1_{\la,q_j,Q_{t_i}(q_0,L_0)}(f_{1,\ut}^{i,r},\ldots,f_{n_i,\ut}^{i,r}) + O(\eps)
\ee
for all $\lm\in q_j\sqrt{\N}$ with $L_{j+1}\leq \lm\leq L_j$ and $\ut\notin T_\eps$ 
uniformly in $1\leq i\leq d$ and $1\leq r\leq R$.
\end{prop}

This proposition follows, as the analogous result did in the continuous setting, from Lemma \ref{Lem5.4} and the follow
 parametric version of Lemma \ref{Lem5.3}.


\begin{lem}[A simultaneous Koopman-von Neumann type decomposition]\label{Lem5.8}\

Let $0<\VE\ll1$, $m\geq1$, and $Q\subseteq\Z^n$ be a cube. 
There exists an integer $\bar{J}_1=O(m\VE^{-3})$ such that for any  $(\eps,q_{\bar{J}_1})$-admissible sequence $L_0\geq L_1\geq\cdots \geq L_{\bar{J}_1}$
with $L_0$ dividing $l(Q)$ and $q_j:=q_0 q_1(\eps)^j$ for $0\leq j\leq \bar{J}_1$ with $q_0\in\N$, and collection of functions \[f_{1,t},\ldots f_{m,t}: Q_t(q_0,L_0)\to [-1,1]\] defined for each $t\in\Ga_{q_0,L_0,Q}$,
there is some $1\leq j< \bar{J}_1$ and a set $T_{\eps}\subs \Ga_{q_0,L_0,Q}$ of size $|T_{\eps}|\leq \eps |\Ga_{q_0,L_0,Q}|$ such that
\eq\label{5.35}
\|f_{i,t}-\E(f_{i,t}|\GG_{q_j,L_j,Q_t(q_0,L_0)}\|_{U^1_{q_{j+1},L_{j+1}}(Q_t(q_0,L_0))}\,\leq\eps
\ee
for all $1\leq i\leq m$ and  $t\notin T_{\eps}$.
\end{lem}

Lemma \ref{Lem5.8} above is of course the discrete analogue of Lemma \ref{KvN-1}.
Since the proofs of Proposition \ref{Cor5.1} and Lemma \ref{Lem5.8} are almost identical to the arguments presented in Section \ref{base} we choose to omit these details.



\section{Proof of Theorem \ref{ProdZ}: The general case}\label{generalcaseZ}

 After the preparations in Section \ref{BaseZ} we can proceed very similarly as in Section \ref{generalcase} to prove our main result in the discrete case, namely Theorem \ref{ProdZ}. The main difference will be that given $0<\eps\ll1$ and $1\leq k\leq d$, we construct a positive integer $q_k(\eps)$ and assume that all our sequences of scales will be $(\eps,q_k(\eps))$-admissible. The cubes $Q_{\ut}(L)$ will be naturally now be replaced by the grids $Q_{\ut}(q,L)$ of the form that already appear in Section \ref{BaseZ} where we always assume $q|L$. 

Let $\De^0=\De_1^0\times\ldots\times\De_d^0$ with each $\De_i^0\subs \Z^{2n_i+3}$ a non-degenerate simplex of $n_i$ points for $1\leq i\leq d$ and $Q=Q_1\times\ldots\times Q_d\subseteq\Z^n$ with $Q_i\subs\Z^{2n_i+3}$  cubes of equal side length $l(Q)$ (taken much larger than the diameter of $\De^0$).   We will use the same parameterizations in terms of hypergraph bundles $\HH_{d,k}^{\un}$ and corresponding notations as in Section \ref{generalcase} to count the configurations $\De=\De_1\times\ldots\times\De_d\subs Q$
 with each $\De_i\subseteq Q_i$ an isometric copy of $\la\De_i^0$ for some $\la\in\sqrt{\N}$.

Given any positive integer $q$ and $\la\in q\sqrt{\N}$ we will make use of the notation
\eq\label{5.28}
\sum_{\ux_i} f(\ux_i)\,\si^i_{\la,q}(\ux_i):= \E_{x_{i1}\in Q_i}\sum_{x_{i2},\ldots,x_{in_i}} f(\ux_i) \,\si_{\la\Delta^0_i,q}(x_{i2}-x_{i1},\ldots,x_{in_i}-x_{i1})\,dx_{i1}\ee
with $\si_{\la\Delta^0_i,q}$ as defined in the previous section and $\ux_i=(x_{i1},\ldots,x_{in_i})\in Q_i^{n_i}$.

 Note that if $S\subs Q$ then the density of configurations $\De$ in $S$, of the form $\De=\De_1\times\ldots\times\De_d$ with each $\De_i\subseteq Q_i$ an isometric copy of $\la\De_i^0$ for some $\la\in q\sqrt{\N}$ is given by the expression
\eq\label{5.29}
\NN^d_{\la\Delta^0,q,Q}(1_S\,;\,e\in\HH_{d,d}^{\un}):= \sum_{\ux_1}\cdots\sum_{\ux_d}\prod_{e\in\HH_{d,d}^{\un}} 1_S(\ux_e)\ \si^1_{\la,q}(\ux_1)\ldots \si^d_{\la,q}(\ux_d).
\ee

More generally, for any given $1\leq k\leq d$ and a family of functions $f_e: Q_{\pi(e)}\to[-1,1]$ with $e\in\HH_{d,k}^{\un}$ we define the multi-linear expression
\eq\label{5.30}
\NN^d_{\la\Delta^0,q,Q}(f_e;e\in\HH_{d,k}^{\un}):= \sum_{\ux_1}\cdots\sum_{\ux_d} \prod_{e\in\HH_{d,k}^{\un}} f_e(\ux_e)\ \si^1_{\la,q}(\ux_1)\ldots. \si^d_{\la,q}(\ux_d).
\ee
as well as
\eq\label{5.31}
\MM^d_{\lm,q,Q}(f_e;e\in\HH_{d,k}^{\un}):= \E_{\ut\in Q}\  \MM^d_{\ut+Q(q,L)}\,(f_e;e\in\HH_{d,k}^{\un})
\ee
where $Q(q,L)=Q_1(q,L)\times\cdots\times Q_d(q,L)$ with each $Q_i(q,L)=(q\Z\cap[-\frac{L}{2},\frac{L}{2}])^{2n_i+3}$ and
\eq\label{5.32}
\MM^d_{\widetilde{Q}}(f_e;e\in\HH_{d,k}^{\un}):= \E_{\ux_1\in \widetilde{Q}_1^{n_1}}\cdots\ \E_{\ux_d\in \widetilde{Q}_d^{n_d}}\,\prod_{e\in\HH_{d,k}^{\un}} f_e(\ux_e)
\ee
for any cube $\widetilde{Q}\subseteq Q$ of the form $\widetilde{Q}=\widetilde{Q}_1\times\cdots\times \widetilde{Q}_d$ with $\widetilde{Q}_i\subseteq Q_i$ for $1\leq i\leq d$.


We note that it is easy to show, as in the continuous, that
if $S\subseteq Q$ with
$|S|\geq \de |Q|$ for some $\de>0$ then
\eq
\MM^d_{\lm,q,Q}(1_S;e\in\HH_{d,d}^{\un})\geq\de^{n_1\cdots\, n_d}-O(\eps)
\ee
for all scales $\lm\in q\sqrt{\N}$ with $0<\lm\ll\eps\, l(Q)$. In light of this observation and the discussion preceding Proposition \ref{Lem5.5}  the proof of Theorem \ref{ProdZ} reduces, as it did in the continuous setting, to the following

\begin{prop}\label{Lem5.7}
Let $0<\eps\ll1$. 
There exist   positive integers $J_d=J_d(\eps)$ and $q_d(\eps)$ such that for any $(\eps,q_d(\VE)^{J_d})$-admissible sequence of scales $l(Q)\geq L_1\geq\cdots \geq L_{J_1}$  and $S\subseteq Q$ there is some $1\leq j< J_d$ such that
\eq\label{5.34}
\NN^d_{\la\Delta^0,q_j,Q}(1_S\,;\,e\in\HH_{d,d}^{\un}) =\MM^d_{\lm,q_j,Q}(1_S;e\in\HH_{d,d}^{\un}) + O(\eps),
\ee
for all $\lm\in q_j\sqrt{\N}$ with $L_{j+1}\leq \lm\leq L_j$ with $q_j:=q_d(\eps)^j$.
\end{prop}

\emph{Quantitative Remark.} A careful analysis of our proof reveals that there exist choices of $J_d(\VE)$ and $q_d(\eps)$ which are less than $W_d(\log (C_\Delta\VE^{-3}))$ and $W_d(C_\Delta\VE^{-13})$ respectively where $W_k(m)$ is again the tower-exponential function defined by $W_1(m)=\exp(m)$ and $W_{k+1}(m)=\exp(W_k(m))$ for $k\geq1$.

The proof of Proposition \ref{Lem5.7} follows along the same lines as the analogous result in the continuous setting. As before we will compare the averages $\NN^d_{\la\Delta^0,q,Q}(f_e;e\in\HH_{d,k}^{\un})$ to those of $\MM^d_{\lm,q,Q}(f_e;e\in\HH_{d,k}^{\un})$, at certain scales $q$ and $\la\in q\sqrt{\N}$ with with $L_{j+1}\leq \lm\leq L_j$, inductively for $1\leq k\leq d$. As the arguments closely follow those given in Section \ref{generalcase} we will be brief and emphasize mainly just the additional features.



\subsection{Reduction of Proposition \ref{Lem5.7} to a more general ``local" counting lemma}\

For any given $1\leq k\leq d$ and a family of functions $f_e: Q_{\pi(e)}\to[-1,1]$ with $e\in\HH_{d,k}^{\un}$
it is easy to see that for any $\VE>0$, scale $L_0>0$ dividing the side-length $l(Q)$, and $q_0|q$ we have
\eq\label{4.7}
\NN^d_{\la\Delta^0,q,Q}(f_e;e\in\HH_{d,k}^{\un})= \E_{\ut\in\Ga_{q_0,L_0,Q}}\, \NN^d_{\la\Delta^0,q,Q_{\ut}(q_0,L_0)}(f_{e,\ut};e\in\HH_{d,k}^{\un})+ O(\eps)
\ee
and
\eq\label{4.8}
\MM^d_{\lm,q,Q}(f_e;e\in\HH_{d,k}^{\un})=\E_{\ut\in\Ga_{L,Q}}\,
\MM^d_{\lm,q,Q_{\ut}(q_0,L_0)}(f_{e,\ut};e\in\HH_{d,k}^{\un})+ O(\eps)
\ee
provided $0<\la\ll \eps L_0$ where $f_{e,\ut}$ denotes the restriction of a function  $f_e$ to the cube $Q_{\ut}(q_0,L_0)$.

Thus the proof of Proposition \ref{Lem5.7} reduces to showing that the expressions in \eqref{4.7} and \eqref{4.8} only differ by $O(\eps)$ for all scales $\lm\in q\sqrt{\N}$ with $L_{j+1}\leq \lm\leq L_j$, given an $(\VE,q)$-admissible sequence $L_0\geq L_1\geq \cdots\geq L_J$, for any collection of bounded functions $f_{e,\ut}$, $e\in \HH_{d,k}^{\un}$, $\ut\in\Ga_{q_0,L_0,Q}$. Indeed, our crucial result will be the following


\begin{prop}[Local Counting Lemma in $\Z^n$]\label{Lem5.10} Let $0<\eps\ll 1$ and $q_0,M\in\N$. 

There exist positive integers $J_k=J_k(\eps,M)$ and $q_k(\eps)$ such that for any $(\eps,q_{J_d})$-admissible sequence of scales $L_0\geq L_1\geq\cdots \geq L_{J_1}$
with $L_0$ dividing $l(Q)$ and $q_j:=q_0\,q_k(\eps)^j$ for $j\geq 1$, and collection of functions  
\[\text{$f_{e,\ut}^m:Q_{t_{\pi(e)}}(q_0,L_0):\to [-1,1]$ \ with \ $e\in \HH_{d,k}^{\un}$, $1\leq m\leq M$ and  $\ut\in\Ga_{q_0,L_0,Q}$}\] 
there exists $1\leq j< J_k$ and a set $T_{\eps}\subs \Ga_{q_0,L_0,Q}$ of size $|T_{\eps}|\leq \eps |\Ga_{q_0,L_0,Q}|$ such that
\eq\label{5.42}
\NN^d_{\la\Delta^0,q_j,Q_{\ut}(q_0,L_0)}(f_{e,\ut};e\in\HH_{d,k}^{\un})=\MM^d_{\lm,q_j,Q_{\ut}(q_0,L_0)}(f_{e,\ut};e\in\HH_{d,k}^{\un})+ O(\eps)
\ee
for all $\lm\in q_j\sqrt{\N}$ with $L_{j+1}\leq \lm\leq L_j$ and $\ut\notin T_\eps$ 
uniformly in $e\in \HH_{d,k}^{\un}$ and $1\leq m\leq M$.
\end{prop}


Note that if $k=d$, $L_0=l(Q)$, $q_0=M=1$, then $|\Ga_{q_0,L_0,Q}|=1$, and moreover if $f_{e,\ut}=1_S$
for all $e\in\HH_{d,k}^{\un}$ for a set $S\subs Q$, then Proposition \ref{Lem5.10}  reduces to precisely Proposition \ref{Lem5.7}. In fact, Proposition \ref{Lem5.10} is a parametric, multi-linear and simultaneous extension of Proposition \ref{Lem5.7} which we need in the induction step, i.e. when going from level $k-1$ to level $k$.

\subsection{Proof of Proposition \ref{Lem5.10}}

We will prove Proposition \ref{Lem5.10} by induction on $1\leq k\leq d$. 

For $k=1$ this is basically Proposition \ref{Cor5.1}, exactly as it was in the base case of the proof of Proposition \ref{Lem4.4}. 

For the induction step we will again need two main ingredients. 
The first establishes that the our multi-linear forms $\NN^d_{\la\Delta^0,q,Q}(f_e;e\in\HH_{d,k}^{\un})$ are controlled by a box-type norm attached to scales $q'$ and $L$. 

Let $Q=Q_1\times\ldots\times Q_d$ with $Q_i\subs\Z^{2n_i+3}$ be cubes of equal side length $l(Q)$ and $1\leq k\leq d$. For any  scale $0<L\ll l(Q)$  and function $f:Q_{e'}\to[-1,1]$ with $e'\in\HH_{d,k}$  we define its local box norm at scales $q'$ and $L$ by
\eq\label{5.43}
\|f\|_{\Box_{q',L}(Q_{e'})}^{2^k} := \E_{\us\in Q_{e'}}  \|f\|_{\Box(Q_{\us}(q',L))}^{2^k}  
\ee
where
\eq\label{5.44}
\|f\|_{\Box(\widetilde{Q})}^{2^k}:= \E_{x_{11},x_{12}\in \widetilde{Q}_{1}}\cdots \ \E_{x_{k1},x_{k2}\in \widetilde{Q}_{k}} \prod_{(\ell_1,\dots,\ell_k)\in\{1,2\}^k}f(x_{1\ell_1},\dots,x_{k\ell_k})
\ee
for any cube $\widetilde{Q}$ of the form $\widetilde{Q}=\widetilde{Q}_1\times\cdots\times\widetilde{Q}_k$.
We note that \eqref{5.31} and \eqref{5.32}  are special cases of \eqref{5.43} and \eqref{5.44} with $k=d$, $\un=(2,\ldots,2)$, and $f_e=f$ for all $e\in\HH_{d,d}^{\un}$.



\begin{lem}[A Generalized von-Neumann inequality on $\Z^n$]\label{Lem5.11}

Let $1\leq k\leq d$.

Let $0<\eps\ll1$, $q,q'\in\N$ with $q q_1(\eps)|q'$, and $\lm\in q\sqrt{\N}$
with $\lm\ll l(Q)$ and $1\ll L\ll(\VE^{2^k})^{10}\lm$.
For any collection of functions  $f_e: Q_{\pi(e)}\to[-1,1]$ with $e\in\HH_{d,k}^{\un}$ we have both
\eq\label{5.45}
|\NN^d_{\la\Delta^0,q,Q}(f_e;e\in\HH_{d,k}^{\un})| \leq \min_{e\in \HH_{d,k}^{\un}}\,\|f_e\|_{\Box_{q',L'}(Q_{\pi(e)})} + O(\eps)
\ee
and 
\eq\label{5.46}
|\MM^d_{\lm,q,Q}(f_e;e\in\HH_{d,k}^{\un})|\leq \min_{e\in \HH_{d,k}^{\un}}\,\|f_e\|_{\Box_{q',L'}(Q_{\pi(e)})}.
\ee
\end{lem}

The proof of inequalities \eqref{5.45} and \eqref{5.46} follow exactly as in the continuous case, see Lemma \ref{Lem4.5}, using Lemma \ref{Lem5.4} in place of Lemma \ref{vN-1}. We omit the details.



\medskip

The crucial ingredient is again a parametric weak hypergraph regularity lemma, i.e. Lemma \ref{Lem4.6} adapted to the discrete settings. The proof is essentially the same as in the continuous case, with exception that the $\Box_{L_j}$-norms are replaced by $\Box_{q_j,L_j}$-norms where $q_j=q_0 q^j$ is a given sequence of positive integers and $L_0\geq L_1\geq\cdots \geq L_{J}$ is an $(\eps,q_J)$-admissible sequence of scales. To state it we say that a $\si$-algebra $\BB$ on a cube $Q$ is of \emph{scale} $(q,L)$ if it is refinement of the grid $\GG_{q,L,Q}$, i.e. if its atoms partition each cube $Q_{\ut}(q,L)$ of the grid. We will always assume that $q|L$ and $L|l(Q)$. 
Recall also that we say the complexity of a $\si$-algebra $\BB$ is at most $m$, and write $\comp(\BB)\leq m$, if it is generated by $m$ sets.

\begin{lem}[Parametric weak hypergraph regularity lemma for $\Z^n$]\label{Lem5.12}\

Let $0<\eps\ll 1$, $1\leq k\leq d$, $q_0,q,L_0, M\in\N$, and let $q_j:=q_0 q^j$ for $j\geq 1$. 
There exists $\bar{J}_k= O(M\eps^{-2^{k+3}})$ such that for any $(\eps^{2^k},q_{\bar{J}_k})$-admissible sequence $L_0\geq L_1\geq\cdots \geq L_{\bar{J}_k}$ with the property that  $L_0$ divides $l(Q)$ and collection of functions 
\[\text{$f^m_{e,\ut}: Q_{\ut_{\pi(e)}}(q_0,L_0)\to [-1,1]$ \ with \ 
$e\in\HH_{d,k}^{\un}$, $1\leq m\leq M$, and $\ut\in\Ga_{q_0,L_0,Q}$}\] 
there is some $1\leq j< \bar{J}_k$ and $\si$-algebras $\BB_{e',\ut}$ of scale $(q_j,L_j)$ on $Q_{\ut_{e'}}(q_0,L_0)$ for each $\ut \in\Ga_{q_0,L_0,Q}$ and $e'\in\HH_{d,k}$ such that
\eq\label{5.47}
\|f_{e,\ut}^m-\E(f_{e,\ut}^m|\BB_{\pi(e),\ut})\|_{\Box_{q_{j+1},L_{j+1}}(Q_{\ut_{\pi(e)}}(L_0))} \leq \eps
\ee
uniformly for all $t\notin T_\VE$, $e\in\HH_{d,k}^{\un}$, and $1\leq m\leq M$, where $T_{\eps}\subs \Ga_{q_0,L_0,Q}$ with $|T_{\eps}|\leq \eps |\Ga_{q_0,L_0,Q}|$.

Moreover, the $\si$-algebras $\BB_{e',\ut}$ 
have the additional local structure that  the exist $\si$-algebras $\BB_{e',\f',\us}$  on $Q_{\us_{\f'}}(q_j,L_j)$ with $\comp(\BB_{e',\f',\us}) = O(j)$  for each  $\us\in\Ga_{q_j,L_j,Q}$, $e'\in\HH_{d,k}$, and $\f'\in\partial e'$ such that
 if $\us\in Q_{\ut}(q_0,L_0)$, then
\eq\label{5.48}
\BB_{e',\ut}\bigr\rvert_{Q_{\us_{e'}}(q_j,L_j)} = \bigvee_{\f'\in\partial e'}
\BB_{e',\f',\us}.\ee
\end{lem}

The proof of Lemma \ref{Lem5.12} follows exactly as the corresponding proof of Lemma \ref{Lem4.6} in the continuous setting, so we will omit the details.
We will however provide some details of how one deduces Proposition  \ref{Lem5.10},  from Lemmas \ref{Lem5.11} and \ref{Lem5.12}. The arguments are again very similar to those in the continuous setting, however one needs to make a careful choice of the integers $q_k(\eps)$, appearing in the statement of the Proposition.



\begin{proof}[Proof of Proposition \ref{Lem5.10}] Let $2\leq k\leq d$ and assume that the lemma holds for $k-1$. 

Let  $0<\eps\ll1$ and $\eps_1:=\exp\,(-C_1\eps^{-2^{k+3}})$ for some large constant $C_1=C_1(n,k,d)\gg 1$. 

We then define $q_k(\eps):=q_{k-1}(\eps_1)$ recalling that $q_1(\eps):=\lcm\{1\leq q\leq C\eps^{-10}\}$ and note that it is easy to see by induction that $q_k(\eps)|q_k(\eps')$ for $0<\eps'\leq\eps$ and $q_{k-1}(\eps)|q_k(\eps)$. 
We further define the function $F(\eps):=J_{k-1}(\eps_1, M)$ with $M=\VE\,\VE_1^{-1}$ 
and recall that $q_j:=q_0\,q_k(\eps)^j$ for $j\geq 1$. 

We now proceed exactly as in the proof of Proposition \ref{Lem4.4} but with $\{L_j\}_{j\geq 1}$ being a
$(\VE_1,q_{\widetilde{J}})$-admissible sequence of scales, with $\widetilde{J}\gg F(\eps)\,\bar{J}_k(\eps,M)$. We again choose a subsequence $\{L_j'\}\subs \{L_j\}$ so that $L'_0=L_0$ and 
$\ind(L'_{j+1})\geq \ind(L'_j) + F(\eps)+2$, but also now set $q'_j=q_{j'}$, where $j':=\ind (L'_j)$.  Lemma \ref{Lem5.12} then  guarantees the existence of  $\si$-algebras $\BB_{e',\ut}$ of scale $(q'_j,L'_j)$ on $Q_{\ut_{e'}}(q_0,L_0)$ for each $\ut \in\Ga_{q_0,L_0,Q}$ and $e'\in\HH_{d,k}$, with the  local structure described above, such that
(\ref{5.47}) holds
uniformly for all $t\notin T'_\VE$, $e\in\HH_{d,k}^{\un}$, and $1\leq m\leq M$, for some $1\leq j<\bar{J}_k(\eps,M)=O(M\eps^{-2^{k+3}})$, where $T'_{\eps}\subs \Ga_{q_0,L_0,Q}$ with $|T'_{\eps}|\leq \eps |\Ga_{q_0,L_0,Q}|$.

Arguing as in the proof of Proposition \ref{Lem4.4} we can  conclude from this that for each $j'\leq l<J'$ we have
\eq\label{4.26}
\NN^d_{\la\De^0,q_l,Q_{\us}(q_j',L'_j)} (f_{e,\us}^m;\,e\in\HH_{d,k}^{\un})
=\sum_{\ur} \al_{\us,\ur,m}\  \NN^d_{\la\De^0,q_l,Q_{\us}(q_j',L'_j)}\,(g_{\f,\us}^{\ur};\,\f\in\HH_{d,k-1}^{\un})+O(\VE)
\ee
and
\eq\label{4.27}
\MM^d_{\la,q_l,Q_{\us}(q_j',L'_j)} (f_{e,\us}^m;\,e\in\HH_{d,k}^{\un})
=\sum_{\ur} \al_{\ur,\us,m}\  \MM^d_{\la,q_l,Q_{\us}(q_j',L'_j)}\,(g_{\f,\us}^{\ur};\,\f\in\HH_{d,k-1}^{\un})+O(\VE)
\ee
provided $ (\eps^{-2^k})^{10} L'_{j+1}\ll \lm$ with $\lm\in q_l\sqrt{\N}$, where each $|\alpha_{\us,r_e}|\leq 1$ and number of index vectors $\ur=(r_e)_{e\in\HH_{d,k}^{\un}}$ is $R^D$ with $D:=|\HH_{d,k}^{\un}|$ and hence $R^D\leq M$ if $C_1\gg1$.

By induction, we apply Proposition \ref{Lem5.10} to the sequence of scales $L'_j=L_{j'}\geq L_{j'+1}\geq\cdots \geq L_{J'}=L'_{j+1}$ with $\eps_1>0$ and for $q_l:=q'_j\,q_k(\eps)^{l-j'}=q_{j'}\,q_{k-1}(\eps_1)^{l-j'}$ where $j'\leq l\leq J'$ with respect to the family of functions $\ g_{\us,\f}^{\ur}: Q_{\us_{\f}}(q'_j,L'_j)\to [-1,1]\,$. This is possible as $J'-j'\gg J_{k-1}(\eps_1,R^D)$ and our sequence of scales is $(\eps_1,q_{J'})$-admissible. Thus there exists an index $j'\leq l<J'$ such that  for all $\lm\in q_l\sqrt{\N}$ with $L_{l+1}\leq \la\leq L_l$ we have
\eq\label{5.65}
\NN^d_{\la\De^0,q_l,Q_{\us}(q_j',L'_j)}\,(g_{\f,\us}^{\ur};\,\f\in\HH_{d,k-1}^{\un}) =
\MM^d_{\la,q_l,Q_{\us}(q_j',L'_j)}\,(g_{\f,\us}^{\ur};\,\f\in\HH_{d,k-1}^{\un}) + O(\eps_1)
\ee 
uniformly in $\ur$ for $\us\notin S_{\eps_1}$, where $S_{\eps_1}\subs \Ga_{q_j',L'_j,Q}$ is a set of size $|S_{\eps_1}|\leq \eps_1 |\Ga_{q_j',L'_j,Q}|$. 

The remainder of the proof follows as just as it did for Proposition \ref{Lem4.4}.
\end{proof}




\section{Appendix: A short direct proof of Part (i) of Theorem B$^\prime$}\label{Appendix}

We conclude by providing a short direct proof of Part (i) of Theorem B$^\prime$, namely the following
 \begin{thm}[Magyar \cite{Magy09}]\label{9}
Let $0<\delta\leq1$ and $\Delta\subseteq\Z^{2k+3}$ be a non-degenerate simplex of $k$ points.

If $S\subseteq\Z^{2k+3}$ has upper Banach density at least $\delta$, 
then
there exists an integer $q_0=q_0(\delta)$ and $\lm_0=\lm_0(S, \Delta)$ such that $S$ contains an isometric copy of $q_0\lm\Delta$ for all $\lm\in\sqrt{\N}$ with $\lambda\geq \lm_0$.
\end{thm}

For any $\VE>0$ we define
\[q_\VE:=\lcm\{1\leq q\leq C\VE^{-10}\}\]
with $C>0$ a (sufficiently) large absolute constant.
Following \cite{LM19} we further define $S\subseteq\Z^n$  to be \emph{$\VE$-uniformly distributed (modulo $q_\VE$)}
 if its \emph{relative} upper Banach density on any residue class modulo $q_\VE$ never exceeds $(1+\VE^2)$ times its density on $\Z^n$, namely if
\[\D^*(S\,|\,s+(q_\VE \Z)^d)\leq(1+\VE^2)\,\D^*(S)\]
for all $s\in\{1,\dots,q_\VE\}^d$. 
It turns out that this notion is closely related to the $U^1_{q,L}(Q)$-norm introduced in Section \ref{BaseZ}. Recall that for any cube $Q\subseteq\Z^n$ and function $f:Q\to[-1,1]$ we define
\eq
\|f\|_{U^1_{q,L}(Q)}:=\Bigl(\frac{1}{|Q|}\sum_{t\in Q}|f*\chi_{q,L}(t)|^2\Bigr)^{1/2}
\ee
with $\chi_{q,L}$ denoting the  normalized characteristic function of the cubes $Q(q,L):=[-\frac{L}{2},\frac{L}{2}]^n\cap (q \Z)^n$.
Note that the $U^1_{q,L}(Q)$-norm measures the mean square oscillation of a function with respect to cubic grids of size $L$ and gap $q$.

The following observation from \cite{LM19} (specifically Lemmas 1 and 2) is  key to our short proof of Theorem \ref{9}.

\begin{lem}\label{Lem5.1} Let $\eps >0$. If $S\subs\Z^n$ be $\VE$-uniformly distributed with $\delta:=\delta^*(S)>0$, then there exists an integer $L=L(S,\VE)>0$ and cubes $Q$ of arbitrarily large side length $l(Q)$ with $l(Q)\gg\VE^{-4}L$ such that
 \[\|1_S-\delta1_Q\|_{U^1_{q_\VE,L}(Q)}=O(\VE).\]  
 \end{lem}

 Let $\De^0=\{v_1=0,v_2,\dots,v_{k}\}$ be a fixed non-degenerate simplex of $k$ points in $\Z^n$ with $n=2k+3$ and define $t_{ij}:=v_i\cdot v_j$ for $2\leq
i,j\leq k$. We now define a function which counts isometric copies of $\lm\De^0$.

Recall, see \cite{Magy09}, that a simplex $\De=\{m_1=0,\dots,m_{k}\}\subs \Z^n$ is isometric to $\la\De^0$ if and only if
$m_i\cdot m_j =\la^2 t_{ij}$ for all $2\leq i,j\leq k$.
For any  $\lm\in \sqrt{\N}$ we define $S_{\la\De^0}(m_2,\dots,m_{k}):\Z^{n(k-1)}\to\{0,1\}$ be the function whose value is 1 if $m_i\cdot m_j=\la^2 t_{ij}$ 
for all $2\leq i,j\leq k$   and is equal to 0 otherwise. It is a well-known fact in number theory, see \cite{Kitaoka} or \cite{Magy09}, that for $n\geq2k+1$ we have that
\[\sum_{m_2,\ldots,m_{k}} S_{\la\De^0} (m_2,\ldots,m_{k}) =\rho(\De^0)\, \la^{(n-k)(k-1)}(1+O(\la^{-\tau}))\]
for some absolute constant $\tau>0$ and constant $\rho(\De^0)>0$, the so-called singular series, which can be interpreted as the product of the densities of the solutions of the above system of equations among the $p$-adics and among the reals.
Thus if we define \[\si_{\la\De^0}:=\rho(\De^0)^{-1}\la^{-(n-k)(k-1)}S_{\lm\De^0}\] then $\si_{\la\De^0}$ is normalized in so much that \[\sum_{m_2,\ldots,m_{k}}\si_{\la\De^0}(m_2,\ldots,m_{k})= 1+O(\la^{-\tau})\] for some absolute constant $\tau>0$.

Let $Q\subs\Z^n$ be a fixed cube and let $l(Q)$ denotes its side length.
For any family of functions \[f_1,\ldots,f_{k}:Q\to [-1,1]\] and $0<\lm\ll l(Q)$ we define 
\eq
\NN^1_{\la\De^0,Q}(f_1,\dots,f_{k}) := \E_{m_1\in Q} \sum_{m_2,\ldots,m_{k}} f_1(m_1)\ldots f_{k}(m_{k})\,\si_{\la\De^0} (m_2-m_1,\ldots,m_{k}-m_1).
\ee

It is clear that if $f_1=\cdots =f_k=1_S$ restricted to $Q$, then the above expression is a normalized count of the isometric copies of $\lm\De^0$ in $S\cap Q$. Thus, Theorem \ref{9} will follow from Lemma \ref{Lem5.1} and the following special case (with $q=1$) of Lemma \ref{Lem5.4}.



\begin{lem}[A Generalized von Neumann inequality]\label{Lem5.2}

Let $0<\eps\ll1$.

If $\lm\in\sqrt{\N}$ with $\lm\ll l(Q)$ and $1\ll L\ll \VE^{10}\lm$ then for any collection of functions $f_1,\dots,f_k:Q\to [-1,1]$ we have
\eq\label{5.10}
|\NN^1_{\la\De^0,Q}(f_1,\dots,f_{k})|\leq
\min_{1\leq j\leq k} \|f_j\|_{U^1_{q_\VE,L}(Q)}+ O(\eps).\ee
\end{lem}

This compares with the purely number theoretic fact that the number of simplices $\De=\{v_1=0,v_2,\ldots,v_k\}\subs \Z^n$ isometric to $\la\De^0$ is asymptotic to $\rho(\De^0)\,\la^{(n-k)(k-1)}$. Thus, under the same conditions as in Lemma \ref{Lem5.2}, we have 
\eq\label{5.11}
\NN^1_{\la\De^0,Q}(1_Q,\dots,1_Q)=  1+O(\la^{-\tau})+O(\eps)\ee
provided one also has $\lm\ll \VE l(Q)$.

\begin{proof}[Proof of Theorem \ref{9}] 
Let $0<\VE\ll \delta^k$
and  $S\subs\Z^n$ be a set of upper Banach density $\de$. 

We assume first that $S$ is $\VE$-uniformly distributed.
Select a scale $L=L(\eps,S)$ and a sufficiently large cube $Q$ so that the conclusion of Lemma \ref{Lem5.1} holds. For a given $\la\in\sqrt{\N}$ with $\lm\ll \VE l(Q)$ and $L\ll \eps^{10}\la$ write $1_S=\de1_Q+g$ and substitute this decomposition into the multi-linear expression $\NN^1_{\la\De^0,Q}(1_S,\dots,1_S)$. Then by Lemma \ref{Lem5.2} and \eqref{5.10}-\eqref{5.11}, we have that
\eq\label{5.12}
\NN^1_{\la\De^0,Q}(1_S,\dots,1_S)\geq\de^k-O(\eps)\ee
and we can conclude that $S$ must contain an isometric copy of $\lm \De^0$.

If $S$ is not $\eps$-uniformly distributed, then its upper Banach density is increased to at least $\de_1:=(1+\eps^2)\de$ when restricted to a residue class $s+(q_\VE\Z)^n$. Identify $s+(q_\VE\Z)^n$ with $\Z^n$ and simultaneously the set $S|_{s+(q_\VE\Z)^n}$ with a set $S_1\subs\Z^n$, via the map $y\to q_\VE^{-1}(y-s)$. Note that if $S_1$ is $\eps$-uniformly distributed then it contains an isometric copy of $\la\De^0$ for all sufficiently large $\la\in\sqrt{\N}$ and hence $S$ contains an isometric copy of $q_\VE\la\De^0$. 

Repeating the above procedure one arrives to a set $S_j=q_\VE^{-j}(S-s_j)\subs\Z^n$ for some $s_j\in\Z^n$ in $j=O(\log\,\eps^{-1})$ steps which contains an isometric copy of $\la\De^0$ for all sufficiently large $\la\in\sqrt{\N}$. \end{proof}





\begin{thebibliography}{19}


\bibitem{Bourg87}
{\sc J. Bourgain}, {\em A Szemer\'edi type theorem for sets of positive density in $\R^k$},
Israel J. Math. 54 (1986), no. 3, 307--316.

\bibitem{CFZ} {\sc D. Conlon, J. Fox, Y. Zhao}, {\em A relative Szemeredi theorem}, Geometric and Functional Analysis 25.3 (2015): 733-762.

\bibitem{CMP} {\sc B. Cook, \'A.  Magyar, M.  Pramanik}, {\em A Roth-type theorem for dense subsets of $\mathbb{R}^d$}, Bull. Lond. Math. Soc. 49 (2017), no. 4, 676-689.

\bibitem{DK} {\sc P. Durcik, V. Kova\v{c}}, {\em Boxes, extended boxes, and sets of positive upper density in the Euclidean space}, arXiv 1809.08692
\bibitem{FK78}
{\sc H. Furstenberg, Y. Katznelson}, {\em An ergodic Szemer\'{e}di theorem for commuting trnasformations}, J. Analyse Math. 31 (1978), 275-291

\bibitem{FKW86}
{\sc H. Furstenberg, Y. Katznelson and B. Weiss}, {\em
Ergodic theory and configurations in sets of positive density},
Mathematics of Ramsey theory, 184--198, 
Algorithms Combin., 5, Springer, Berlin, 1990. 

\bibitem{FK96}
{\sc A. Frieze, R. Kannan}, {\em The regularity lemma and approximation schemes for dense problems}, In Foundations of Computer Science (1996) Proc. 37th Annual Symp. IEEE., 12-20

\bibitem{Gow06}
{\sc W. T. Gowers}. {\em Hypergraph regularity and the multidimensional Szemer\'edi theorem}, Annals of Mathematics (2007), 897-946.

\bibitem{Graham90}
{\sc R.L. Graham}. {\em Recent trends in Euclidean Ramsey theory}, Discrete Mathematics 136, no. 1-3 (1994), 119-127.

\bibitem{IR07} {\sc A. Iosevich and M. Rudnev}, {\em Erd\H{o}s distance problem in vector spaces over finite fields}, Trans. Amer. Math. Soc. 359, no. 12 (2007), 6127-6142.

\bibitem{Kitaoka}
{\sc Y.Kitaoka}, {\em Siegel modular forms and representation by quadratic forms} Lectures on Mathe-
matics and Physics, Tata Institute of Fundamental Research, Springer-Verlag, (1986)

\bibitem{LM17}
{\sc N. Lyall and \'A. Magyar}, {\em Product of simplices and sets of positive upper density in $\R^d$}, Math. Proc. of the Cambridge Philos. Soc. 165. no. 1. (2018), 25-51

\bibitem{LM18}
{\sc N. Lyall and \'A. Magyar}, {\em Distance Graphs and sets of positive upper density in $\R^d$}, to appear in Anal. PDE 

\bibitem{LM19}
{\sc N. Lyall and \'A. Magyar}, {\em Distances and trees in dense subsets of $\Z^d$}, arXiv 1509.09298

\bibitem{LMP}
{\sc N. Lyall,  \'A. Magyar, H. Parshall}, {\em Spherical configurations over finite fields},            to appear in Amer. J. Math.

\bibitem{Magy08}
{\sc \'A. Magyar}, {\em Distance sets of large sets of integer points}, Israel J. Math., v (2008) pp.

\bibitem{Magy09}
{\sc \'A. Magyar},
{\em k-point configurations in sets of positive density of $Z^n$}, Duke Math. J., v 146/1, (2009) pp. 1-34.

\bibitem{Hans17}
{\sc H. Parshall}, {\em Simplices over finite fields}, Proc. Amer. Math. Soc. 145.6 (2017), 2323-2334.

\bibitem{Siegel}
{\sc C. L. Siegel}, {\em On the theory of indefinite quadratic forms}, Ann. of Math. (2) 45 (1944), 577-622

\bibitem{Szem75}
{\sc E., Szemer\'edi}, {\em On sets of integers containing no k elements in arithmetic progression}, Acta Arith. 27 (1975), 199-245.

\bibitem{Tao05}
{\sc T. Tao}, {\em Szemer\'edi's regularity lemma revisited}, arXiv preprint math/0504472 (2005)

\bibitem{Tao06}
{\sc T. Tao}, {\em A variant of the hypergraph removal lemma}, Journal of Combinatorial Theory, Series A 113.7 (2006): 1257-1280

\bibitem{TaoMontreal}
{\sc T. Tao}, {\em The ergodic and combinatorial approaches to Szemer\'edi's theorem}, Additive combinatorics, 145--193, 
CRM Proc. Lecture Notes, 43, Amer. Math. Soc., Providence, RI, 2007. 

\bibitem{TV}
{\sc T. Tao and V. Vu}, {\em Additive combinatorics}, Cambridge Studies in Advanced Mathematics, 105. Cambridge University Press, Cambridge, 2006. xviii+512 pp.

\bibitem{Ziegler06}
{\sc
T. Ziegler}, {\em Nilfactors of $\R^m$-actions and configurations in sets of positive upper density in $\R^m$}, J. Anal. Math. 99 (2006), 249-266.

\end{thebibliography}
\end{document}